\documentclass[oneside]{amsart}
\usepackage[utf8]{inputenc}
\usepackage{pdfpages}
\numberwithin{equation}{section}

\usepackage[a4paper]{geometry}
\usepackage{pinlabel}
\usepackage{amssymb,amsfonts,amsmath}
\usepackage{comment} 
\usepackage[all,arc]{xy}
\usepackage{enumerate}
\usepackage{float}
\usepackage[caption = false]{subfig}
\usepackage{mathrsfs,mathtools}
\usepackage{todonotes,booktabs}
\usepackage{stmaryrd}
\usepackage{marvosym}
\usepackage{graphicx}
\usepackage[backref=page]{hyperref}
\usepackage{pdflscape}
\usepackage{microtype}
\usepackage[all]{xy}
\usepackage{tikz-cd}
\usepackage[capitalise]{cleveref}
\usepackage{marginnote}
\usepackage{todonotes}
\usepackage{longtable}
\usepackage{lscape}
\definecolor{chartgray}{gray}{0.5}
\definecolor{darkcyan}{rgb}{0, 0.7, 0.7}
\definecolor{truecyan}{rgb}{0, 1, 1}
\definecolor{darkgreen}{rgb}{0, 0.65, 0}
\definecolor{truemagenta}{rgb}{1, 0, 1}

\hypersetup{
 colorlinks=true,%
  linkcolor=blue,%
  citecolor=blue,%
  filecolor=blue,%
  menucolor=blue,%
 urlcolor=blue,%
 pdfnewwindow=true,%
  pdfstartview=FitBH}   
  \usepackage[skip=0pt,font=scriptsize]{caption}

\newtheorem{thm}{Theorem}[section]

\newtheorem*{thm*}{Theorem}
\newtheorem{thmx}{Theorem}

\newtheorem*{cor*}{Corollary}
\newtheorem*{prop*}{Proposition}
\newtheorem{cor}[thm]{Corollary}
\newtheorem*{notation*}{Notation}
\newtheorem{example}[thm]{Example}
\newtheorem{defn}[thm]{Definition}
\newtheorem*{defn*}{Definition}
\newtheorem{prop}[thm]{Proposition}
\newtheorem{lem}[thm]{Lemma}

\newtheorem*{conj*}{Conjecture}
\newtheorem*{quest*}{Question}
\newtheorem{quest}[thm]{Question}
\theoremstyle{definition}
\newtheorem{rem}[thm]{Remark}
\newtheorem{met}[thm]{Method}

\usepackage{soul}

\definecolor{violet}{rgb}{0.56, 0.0, 1.0}

\newcommand{\RNum}[1]{\uppercase\expandafter{\romannumeral #1\relax}}

\newcommand{\Z}{\mathbb{Z}}

\newcommand{\W}{\mathbb{W}}
\newcommand{\F}{\mathbb{F}}
\newcommand{\E}{\mathbf{E}}

\newcommand{\usi}{u_{\sigma_i}}

\newcommand{\BP}{BP}
\newcommand{\BPG}[1][G]{\BP^{(\!(#1)\!)}}
\newcommand{\BPC}{\BPG[C_{4}]}

\newcommand{\BPone}{\BPC\!\langle 1 \rangle}

\newcommand{\done}{\bar{\mathfrak{d}}_1}
\newcommand{\usig}{{u_{\sigma_i}}}

\newcommand{\sone}{\bar{s}_1}

\DeclareMathOperator{\Ind}{Ind}
\DeclareMathOperator{\Gal}{Gal}

\DeclareMathOperator{\Res}{res}
\DeclareMathOperator{\Tr}{tr}
\usepackage{setspace}

\title[$RO(G)$-Graded HFPSS for Height $2$ Morava $E$-Theory]{$RO(G)$-graded homotopy fixed point spectral sequence for height $2$ Morava $E$-theory}
\usepackage[foot]{amsaddr}

\author{Zhipeng Duan $^{1}$}\address{$^{1}$ School of Mathematical Sciences, Nanjing Normal University}\email{$^{1}$ zhipeng@njnu.edu.cn}
\author{Hana Jia Kong $^{2}$} \address{$^{2}$ 
Mathematics Department, Harvard University}\email{$^{2}$ hana.jia.kong@gmail.com}
\author{Guchuan Li $^{3}$}\address{$^{3}$ School of Mathematical Sciences, Peking University}\email{$^{3}$ liguchuan@math.pku.edu.cn}
\author{Yunze Lu $^{4}$}\address{$^{4}$ Department of Mathematics, University of California San Diego}\email{$^{4}$ yul248@ucsd.edu}

\author{Guozhen Wang $^{5}$}\address{$^{5}$ Shanghai Center for Mathematical Sciences, Fudan University}\email{$^{5}$ wangguozhen@fudan.edu.cn}

\usepackage{babel}
\usepackage[square,numbers,sort&compress]{natbib}
\usepackage{url,etoolbox}

\usepackage{lipsum}

\makeatletter
\patchcmd{\abstract}{3pc}{0pt}{}{} 
\makeatother

\patchcmd{\abstract}{\scshape\abstractname}{\textbf{\abstractname}}{}{}

\begin{document}

\maketitle

\begin{abstract}
We consider $G=Q_8,SD_{16},G_{24},$ and $G_{48}$ as finite subgroups of the Morava stabilizer group which acts on the height $2$ Morava $E$-theory $\mathbf{E}_2$ at the prime $2$.
We completely compute the $G$-homotopy fixed point spectral sequences of $\mathbf{E}_2$.
Our computation uses recently developed equivariant techniques since Hill, Hopkins, and Ravenel. We also compute the $(*-\sigma_i)$-graded $Q_8$- and $SD_{16}$-homotopy fixed point spectral sequences, where $\sigma_i$ is a non-trivial one-dimensional representation of $Q_8$.
\end{abstract}
\smallskip


\tableofcontents

\section{Introduction and main results}

\subsection{Motivation and main results}

Chromatic homotopy theory studies large-scale phenomena in the stable homotopy category using the algebraic geometry of smooth $1$-parameter formal groups \cite{Qui69,Mor85}. The moduli stack of formal groups has a stratification by heights, which in the stable homotopy category corresponds to localizations with respect to the Morava $E$-theories $\mathbf{E}_n$ of height $n \geq 0$.

We fix a prime $p$.
Let $\Gamma_n$ be the $p$-typical height-$n$ Honda formal group law over $\mathbb{F}_p$ and let $\mathbb{S}_n$ be the automorphism group of $\Gamma_n$ (extended to $\mathbb{F}_{p^n}$). Let $\mathbb{G}_n = \mathbb{S}_n \rtimes \text{Gal}(\mathbb{F}_{p^n}/\mathbb{F}_p)$ be the (extended) Morava stabilizer group. Goerss--Hopkins--Miller showed that the continuous action of $\mathbb{G}_n$ on $\pi_* \mathbf{E}_n$ can be refined to a unique $\mathbb{E}_\infty$-action of $\mathbb{G}_n$ on $\mathbf{E}_n$ \cite{Rez98, GH04,Lur18}.

At a prime $p$, one can assemble the information of $\mathbf{E}_n$ with the $\mathbb{G}_n$-action of height $n$ for all $n\geqslant 0$ to recover the $p$-local sphere. More precisely, the chromatic convergence theorem due to Hopkins and Ravenel \cite{Rav92} exhibits the $p$-local sphere spectrum $S^0_{(p)}$ as the homotopy inverse limit of the $\mathbf{E}_n$-local spheres (in the sense of Bousfield \cite{Bou79})
\[\cdots \longrightarrow L_{\mathbf{E}_n}S^0 \longrightarrow \cdots \longrightarrow L_{\mathbf{E}_1}S^0 \longrightarrow L_{\mathbf{E}_0}S^0.\]
Furthermore, these localizations can be built inductively  via the following homotopy pullback square (the chromatic fracture square)
\[\begin{tikzcd}
L_{\mathbf{E}_n}S^0 \ar[r] \ar[d] & L_{K(n)}S^0 \ar[d] \\
L_{\mathbf{E}_{n-1}}S^0 \ar[r] & L_{\mathbf{E}_{n-1}}L_{K(n)}S^0,
\end{tikzcd}
\]
where $L_{K(n)}$ denotes the localization functor with respect to $K(n)$, the $n^\text{th}$ Morava K-theory.
From this perspective, the $K(n)$-local spheres $ L_{K(n)}S^0 $ are the building blocks of the $p$-local stable homotopy category. Devinatz and Hopkins showed that $ L_{K(n)}S^0 $ is equivalent to the homotopy fixed point spectrum $\mathbf{E}_n^{h\mathbb{G}_n}$ \cite{DH04}.

A framework for building the $K(n)$-local sphere from more computable spectra is developed in \cite{GHMR05,Hen07}. The more computable spectra are of the form $\mathbf{E}_n^{hG}$ for various finite subgroups $G$ of the Morava stabilizer group $\mathbb{G}_n$. This generalizes the height $1$ resolution 
$$L_{K(1)}S^0\simeq \mathbf{E}_1^{h\mathbb{G}_1} \rightarrow \mathbf{E}_1^{hG} \rightarrow \mathbf{E}_1^{hG}$$
where $G$ is a certain finite subgroup of $\mathbb{G}_1$ (see \cite{HMS94,GHMR05}). Explicit resolutions of the $K(2)$-local sphere from assembling various $\mathbf{E}_2^{hG}$ at the prime $2$ \cite{Bea15,BG18,Hen19} and the prime $3$ \cite{GHMR05} have led to important progress in the study of $K(2)$-local category including the chromatic splitting of the $K(2)$-local sphere \cite{Bea17',GHM14,BGH22}. From this finite resolution perspective, the spectra $\mathbf{E}_n^{hG}$ are the building blocks of the $K(n)$-local stable homotopy category. In particular, the homotopy groups $\pi_* \mathbf{E}_n^{hG}$ detect important families of classes in the stable homotopy groups of spheres \cite{HHR16, LSWX19, BMQ20}. Therefore, computations with $\mathbf{E}_n^{hG}$ constitute a central topic in chromatic homotopy theory and in general are extremely challenging.

Hewett classified all the finite subgroups of $\mathbb{S}_n$ \cite{Hew95} (see also \cite{Buj12}). From now on, we focus on the prime $p=2$, which is the only prime $p$ that there are non-cyclic finite $p$-subgroups in the Morava stabilizer group. If $n=2^{m-1}\ell$ where $\ell$ is odd, then when $m\neq 2$, the maximal finite $2$-subgroups of $\mathbb{G}_n$ are isomorphic to $C_{2^m}$, the cyclic group of order $2^m$; when $m=2$, $n$ is of the form $4k+2$, and the maximal finite $2$-subgroups are isomorphic to $Q_8$, the quaternion group.  

There are breakthroughs of computations of $\mathbf{E}_n^{hG}$ when $G$ is cyclic due to the recent development of equivariant methods \cite{HHR17,HSWX2018,BBHS20,HS20}. These computations are done by a new tool called the slice spectral sequence. The slice spectral sequence computations of the norm of real cobordism theories induce computations of $\mathbf{E}_n^{hG}$ at the prime $2$ for the case $G=C_{2^m}$. As far as the authors are aware, there are no such computations for the case $G=Q_8$ due to the lack of the slice information.

At height $2$, the group $Q_8$ first appears as a subgroup of the (small) Morava stabilizer group $\mathbb{S}_2$. Maximal finite subgroups of $\mathbb{S}_2$ are isomorphic to $G_{24} = Q_8 \rtimes C_3$. Similarly, in the (extended) Morava stabilizer group $\mathbb{G}_2$, there are subgroups isomorphic to $SD_{16}$ and $G_{48}$. Homotopy fixed points of $\mathbf{E}_2$ with respect to the above subgroups appear in the finite resolution of $\mathbf{E}_2^{h\mathbb{G}_2}$, the $K(2)$-local sphere at the prime 2, as building blocks \cite{Bea15, BG18}. Moreover, they also appear in the interplay between chromatic layer 2 and the theory of elliptic curves (see for example \cite{Hop02,HM14,BO16,HL16}). Important examples such as $tmf$ are related to computations of $\mathbf{E}_2^{hG_{48}}$.

In this paper, we use equivariant methods and a new method, which we called ``the vanishing line method", to compute the $G$-homotopy fixed point spectral sequence ($G$-HFPSS) of the height $2$ Morava $E$-theory $\mathbf{E}_2$ at the prime $2$ for $G=Q_8, SD_{16}, G_{24}$ and $G_{48}$.

Let $\sigma_i$ (resp. $\sigma_j$, $\sigma_k$) be the one-dimensional non-trivial representation of $Q_8$ that $i\in Q_8$ (resp. $j, k \in Q_8$) acts trivially. We compute the integer-graded as well as $(*-\sigma_i)$-graded $G$-HFPSS for $\E_2$. By symmetry, the $(*-\sigma_i)$-graded $G$-HFPSS gives the $(*-\sigma_j)$-graded and the $(*-\sigma_k)$-graded $G$-HFPSS for $\E_2$. 

\begin{thmx}\label{theorem:A}
\begin{enumerate}
    \item The integer-graded $Q_8$-$\mathrm{HFPSS}$ for $\mathbf{E}_2$ has differentials as listed in \cref{table:HPFSS_integer_diff} (also see \cref{fig:integerE2,fig:integerE5,fig:integerE9,fig:integerE11}). The $E_\infty$-page with all $2$ extensions is presented in \cref{fig:integerEinf}.
    
Furthermore, we have 
$$SD_{16}\text{-}\mathrm{HFPSS}(\mathbf{E}_2) \otimes_{\mathbb Z_2}\mathbb W(\mathbb F_4) = Q_{8}\text{-}\mathrm{HFPSS}(\mathbf{E}_2),$$
where the tensor products happen on $E_r$ and $d_r$ for every $2\leq r\leq \infty$.

\item 
The $(*-\sigma_i)$-graded $Q_8$-HFPSS for $\mathbf{E}_2$ has differentials in \cref{table:HPFSS_sigmai_diff} (also see \cref{fig:sigmaE2,fig:sigmaE5,fig:sigmaE7,fig:sigmaE9}) and the $E_\infty$-page is presented as \cref{fig:sigmaEinf}. 

Furthermore, we have 
$$SD_{16}\text{-}\mathrm{HFPSS}(\mathbf{E}_2) \otimes_{\mathbb Z_2}\mathbb W(\mathbb F_4) = Q_{8}\text{-}\mathrm{HFPSS}(\mathbf{E}_2),$$
where the tensor products happen on $E_r$ and $d_r$ for every $2\leq r\leq \infty$.
\end{enumerate}

\end{thmx}

\begin{thmx}\label{theorem:B}
 The integer-graded $G_{24}$-HFPSS for $\mathbf{E}_2$ is a subobject of the integer-graded $Q_8$-HFPSS for $\mathbf{E}_2$ which consists of classes with $D^m$ where $3\mid m$, and the differentials are the same. The $E_\infty$-page with all $2$ extensions is presented as in \cref{fig:G24integerEinf1}. 
 Furthermore, we have
$$G_{48}\text{-}\mathrm{HFPSS}(\mathbf{E}_2)\otimes_{\mathbb Z_2}\mathbb W(\mathbb F_4)= G_{24}\text{-}\mathrm{HFPSS}(\mathbf{E}_2),$$
where the tensor products happen on $E_r$ and $d_r$ for every $2\leq r\leq \infty$.
\end{thmx}

\cref{theorem:A} gives the complete computation of the $Q_{8}$-HFPSS of $\mathbf{E}_2$ for the integer-graded part\footnote{While the integer-graded result can be deduced from the $tmf$ computation \cite{Bau08} and should be known to experts, it was not written down in the literature as far as the authors are aware.}, and the $(*-\sigma_i)$-graded part\footnote{
The $Q_8$-representation $\sigma_i$ is not a restriction of any $G_{24}$-representation.
}.

Our methods for $Q_8$-HFPSS computations are independent of previous computations and can potentially work for higher heights. 
The first method is the recently developed equivariant method which uses the restriction, transfer, and norm structures of the spectral sequence to deduce differentials and hidden extensions. More precisely, we deduce differentials and hidden extensions in the $Q_8$-HFPSS for $\mathbf{E}_2$ from differentials in the $C_4$-HFPSS for $\mathbf{E}_2$ (computed in \cite{HHR17,BBHS20}) via restrictions, transfers, and norms. For example, the restriction functor from $Q_8$ to $C_4$ implies a hidden $2$-extension from a class at $(54,2)$ to a class at $(54,10)$ in the $Q_8$-HFPSS for $\E_2$ (See \cref{lem:hiiden2Q8}) which is crucial to deduce the $d_{13}$-differential proved in \cref{prop:d13two}. This exempts us from using the Toda-bracket-shuffling method as in \cite[Proposition~ 8.5(3)]{Bau08}. Moreover, $RO(G)$-gradings have been proven to be helpful in computations \cite{HHR17,BBHS20}. For example, for groups $H\subset G$, the norm map from the $H$-HFPSS to the $G$-HFPSS on the $E_2$-page is only defined after extending to $RO(G)$-gradings \cite{Ull13,HHR17,MSZ20}. 
Norm maps allow us to pull back and push forward known differentials for new differential information. For example, our computation of the $(*-\sigma_i)$-graded $G$-HFPSS for $\mathbf{E}_2$ gives an alternative proof of the existence of a $d_9$-differential in the integer-graded $Q_8$-HFPSS for $\mathbf{E}_2$ using the norm method (See \cref{prop:d9fivenorm}).

We also introduce a new method: ``the vanishing line method". The vanishing line result \cite[Theorem~6.1]{DLS2022} states that at the prime $2$, the $G$-HFPSS for $\mathbf{E}_n$ admits a strong vanishing line of filtration $N$, an explicit number depending on height $n$ and $G$. Recall that having a strong horizontal vanishing line of filtration $f$ means that the spectral sequence collapses after the $E_f$-page, and any element of filtration greater than or equal to $f$ supports a differential or is hit. In the case for $n=2$ and $G=Q_8$, the number $N$ is $25$ and therefore all permanent cycles in filtration $\geq 25$ must be hit, which forces differentials to happen in many cases. For example, in \cref{prop:d13one} the vanishing line method forces three differentials, including the longest $d_{23}$-differentials, just from the $E_2$-page information.

Along the way, we proved the following theorem for general heights.

\begin{thmx}[\cref{thm:sharpvanishingline}]\label{theorem:C}
The $Q_8$-$\mathrm{HFPSS}$ for $\mathbf{E}_{4k+2}$ admits a strong vanishing line of filtration $2^{4k+5}-9$.
\end{thmx}

The above theorem improves the number $N$ of the vanishing result in \cite[Theorem~6.1]{DLS2022} for the $Q_8$ case at all possible heights. This improvement makes the vanishing line sharp for all known cases.

We conclude this introduction with two questions related to the computation of $Q_8$-HFPSS for $\mathbf{E}_{4k+2}$ at higher heights. The equivariant methods and ``the vanishing line method'' work for higher heights. Nevertheless, a computable description of the $Q_8$-action on $\pi_*\mathbf{E}_{4k+2}$ like (\ref{eq:typeone}) for $k=0$ is not known for $k\geq 1$.

\begin{quest}
How to give a computable description of the $Q_8$-action on Morava $E$-theory $\mathbf{E}_{4k+2}$ for $k\geq 1$?
\end{quest}

\begin{quest}
Is the vanishing line result \Cref{theorem:C} of $Q_8$-HFPSS for Morava $E$-theory $\mathbf{E}_{4k+2}$ sharp for $k\geq 1$?
\end{quest}

\subsection{Summary of the contents}
This paper is organized as follows. \cref{section:preliminaries} provides a necessary background for the computational tools for the
$RO(G)$-graded homotopy fixed point spectral sequence, and the input for the computation of the $Q_8$-HFPSS for $\mathbf{E}_2$. In particular, we review the norm structure in $RO(G)$-graded homotopy fixed point spectral sequences (\cref{thm:normdiff}) and the interplay between the homotopy fixed point spectral sequences and the Tate spectral sequences in general (\cref{lemma:isorange}). We briefly review the $Q_8$-action on $\pi_* \mathbf{E}_2$ (\ref{eq:typeone}) and the computation of $RO(C_4)$-graded Mackey-functor-valued $C_4$-HFPSS for $\mathbf{E}_2$ (\cref{subsec:Mackeycomputation}). We take these as the input for the $Q_8$-HFPSS for $\mathbf{E}_2$. In \cref{sec:E2page} we compute the $E_2$-page of the integer-graded and $(*-\sigma_i)$-graded $Q_8$-HFPSS$(\mathbf{E}_2)$ by Bockstein spectral sequences. 

In \cref{section:differentials}, we derive all differentials in the integer-graded $Q_8$-HFPSS for $\mathbf{E}_2$ via equivariant methods and the method of \cref{thm:sharpvanishingline}. In \Cref{subsection:generalpropertyQ8HFPSS}, we prove the properties of the $Q_8$-HFPSS for $\mathbf{E}_2$ that we need for our computation. The vanishing line (\cref{thm:sharpvanishingline}) works for general heights and is of its own interest. In \Cref{subsection:differentials}, we give a complete computation of all differentials in the logical order. The vanishing line method gives some difficult differentials (for example \cref{prop:d13one}). In \cref{sub:2ext}, we solve all $2$ extensions. In \Cref{subsection:differentialsalternative}, we present alternative ways to compute some differentials.

In \cref{section:sigma}, we also apply equivariant methods and the vanishing line method to compute the  $(*-\sigma_i)$-graded $Q_8$-HFPSS for $\mathbf{E}_2$. In particular, this computation gives an alternative proof of a $d_9$-differential in the integer-graded part. In \cref{section:charts}, we list figures that present our computation. In \cref{sec:groupcoho}, we explain algebraic computations of the $Q_8$ group cohomology. In addition, we explain how the Hurewicz image of $\mathbf{E}_2^{hC_4}$ helps to compute the restriction map from $Q_8$-HFPSS to $C_4$-HFPSS.

\subsection{Acknowledgements}
The authors would like to thank Agn\`es Beaudry, Mark Behrens, Paul Goerss, Bert Guillou, XiaoLin Danny Shi, Zhouli Xu, and Mingcong Zeng for helpful conversations. We would also like to thank the anonymous referees for their careful reading and the many useful suggestions. The second author was supported
by the National Science Foundation under Grant No. DMS-1926686. The third author is grateful to the Max Planck Institute for Mathematics in Bonn for its hospitality and financial support. The fifth author is partially supported by grant NSFC-12226002, the Shanghai Rising--Star Program under Agreement No. 20QA1401600, and Shanghai
Pilot Program for Basic Research--Fudan University 21TQ1400100(21TQ002).


\section{Preliminaries}\label{section:preliminaries}
\subsection{\texorpdfstring{$RO(G)$-graded homotopy fixed point spectral sequences and Tate spectral sequences}{}}   \label{subsec:HFPSSTateSS} \hfill

Let $X$ be a (genuine) $G$-spectrum. Denote the ring of real orthogonal virtual representations of the group $G$ by $RO(G)$. The equivariant homotopy groups of $X$ can be organized into an $RO(G)$-graded Mackey functor $\underline{\pi}_{\bigstar}X$ as follows
\[
\underline{{\pi}}_{V}(X)(G/H) = \pi_{V}^H(X) =[S^V, X]^H
\]
where $V$ is a virtual $G$-representation and $[S^V, X]^H$ denotes the genuine $H$-equivariant homotopy classes of maps (see \cite[Section~2.1]{BBHS20}).

We briefly review the Tate diagram of a $G$-spectrum $X$ from \cite{GM95b}. Recall the cofree replacement map $X\rightarrow F(EG_+,X)$ and the isotropy separation sequence $EG_+\rightarrow S^0\rightarrow \tilde{E}G$. Smashing them  together gives the following diagram:
\begin{equation*}
    \begin{gathered}\xymatrix{EG_+\wedge X \ar[r]\ar[d]^{\simeq}& X\ar[r]\ar[d] & \tilde{E}G\wedge X \ar[d]\\
    EG_+\wedge F(EG_+,X)\ar[r]& F(EG_+,X)\ar[r]  & \tilde{E}G\wedge F(EG_+,X)
    }
    \end{gathered}
\end{equation*}
where the left vertical map is an underlying weak equivalence. Taking $G$-fixed points gives us the following Tate diagram: 
\begin{equation*}
    \begin{gathered}\xymatrix{X_{hG} \ar[r]\ar[d]^{\simeq}& X^G\ar[r]\ar[d] & \Phi^GX \ar[d]\\
   X_{hG}\ar[r]& X^{hG}\ar[r]  & X^{tG}.
    }
    \end{gathered}
\end{equation*}
The Adams isomorphism \cite[Theorem~5.4]{Ada84} shows that the fixed point $(EG_+\wedge X)^G$ is weak equivalent to the homotopy orbit $X_{hG}=(EG_+\wedge X)/G$. The left bottom map is the norm map and its cofiber $X^{tG}$ is the Tate construction. Moreover, the right square is a homotopy pullback square which relates the information of the homotopy fixed point $X^{hG}$ , the geometric fixed point $\Phi^GX$ and the Tate construction $X^{tG}$ with the actual fixed point $X^G$. This diagram is useful in both theoretical and computational applications (see for examples \cite{GM95b,Gre18}).

We shall consider an analog of the Tate diagram for the slice tower of $X$ as in \cite[Section~2]{DLS2022}. 
Let $P^{\bullet}X=\{P^n X\}_{n\in \mathbb{Z}}$ be the slice tower of $X$. The diagram of the towers
\begin{equation*}
    \begin{gathered}\xymatrix{EG_+\wedge P^{\bullet}X \ar[r]\ar[d]^{\simeq}& P^{\bullet}X\ar[r]\ar[d] & \tilde{E}G\wedge P^{\bullet}X \ar[d]\\
    EG_+\wedge F(EG_+,P^{\bullet}X)\ar[r]& F(EG_+,P^{\bullet}X)\ar[r]  & \tilde{E}G\wedge F(EG_+,P^{\bullet}X).
    }
    \end{gathered}
\end{equation*}
induces a Tate diagram of spectral sequences for computing each entry in the Tate diagram
\begin{equation} \label{diag:TateDiagramSS}
    \begin{gathered}\xymatrix{
    \text{HOSS}(X)\ar[r]\ar[d]^{=}&\text{SliceSS}(X)\ar[d]\ar[r]& \text{LSliceSS}(X)\ar[d]\\
    \text{HOSS}(X)\ar[r]& \text{HFPSS}(X)\ar[r] &\text{TateSS}(X).
    }
     \end{gathered}
\end{equation}
We now explain the above notations. We use $*$ to denote an integer and $\bigstar$ to denote an $RO(G)$-grading.  We denote the underling homotopy group  $\pi_0^e(X\wedge S^{-\bigstar})$ as a $G$-module by $\pi_{\bigstar}(X)$.
\begin{itemize}
    \item  The spectral sequence associated to the tower  $EG_{+}\wedge P^{\bullet}X$ is the $RO(G)$-graded homotopy orbit point spectral sequence (HOSS) of $X$ with the $E_2$-page as
    \[
    \underline{H}_*(G,\pi_{\bigstar}(X))
    \]
    which converges to $\underline{\pi}_{\bigstar-*}EG_+\wedge X$. In particular, in integer-graded part and $G/G$-level, this spectral sequence converges to $\pi_{*}^GEG_+\wedge X=\pi_*X_{hG}$. 
    
    \item The spectral sequence  associated to the tower $P^{\bullet}X$ is the slice spectral sequence (SliceSS) of $X$ with the $E_2$-page as 
    \[
    \underline{\pi}_{\bigstar-*}P^{|\bigstar|}_{|\bigstar|}X
    \]
    which converges to $\underline{\pi}_{\bigstar-*}X$. Here $P^{|\bigstar|}_{|\bigstar|}X$ is the fiber of $P^{|\bigstar|}X\rightarrow P^{|\bigstar|-1}X$ and $|\bigstar|$ is the underlying dimension of $\bigstar$. In particular, in integer-graded part and $G/G$-level, this spectral sequence converges to $\pi_*^G X=\pi_* X^G.$    
    \item The spectral sequence associated to the tower $F(EG_+,P^{\bullet}X)$ is the $RO(G)$-graded homotopy fixed point spectral sequence (HFPSS) of $X$ \cite[Section~2]{BM94} with the $E_2$-page as
    \[
    \underline{H}^*(G,\pi_{\bigstar}(X))
    \]
    which converges to $\underline{\pi}_{\bigstar-*} F(EG_+,X)$. In particular, in integer-graded part and $G/G$-level, this spectral sequence converges to $\pi_*^GF(EG_+,X)=\pi_* X^{hG}$    
    \item  The spectral sequence  associated to the tower $\tilde{E}G\wedge P^{\bullet}X$ is the $RO(G)$-graded localized slice spectral sequence (LSliceSS) of $X$ introduced in \cite{LSDZ23}.  In many cases, including $G=Q_8$, smashing with $\tilde{E}G$ is equivalent to inverting a certain Euler class $a_V$ (see \cref{defn:aclass} for the definition of Euler classes) for a specific $G$-representation $V$. Therefore, the $E_2$-page of this spectral sequence is 
    \[
    \underline{\pi}_{\bigstar-*}a_V^{-1}P^{|\bigstar|}_{|\bigstar|}X
    \]
    which converges to $\underline{\pi}_{\bigstar-*}a_V^{-1}X$. In particular, in integer-graded part and $G/G$-level, this spectral sequence converges to $\pi_*^G a_V^{-1}X=\pi_* \Phi^G X$. 
    \item The spectral sequence associated to the tower $\tilde{E}G\wedge F(EG_+,P^{\bullet}X)$ is the $RO(G)$-graded Tate spectral sequence (TateSS) of $X$ with the $E_2$-page as
    \[
    \underline{\widehat{H}}^*(G,\pi_{\bigstar}(X))
    \]
which converges to $\underline{\pi}_{\bigstar -*}\tilde{E}G\wedge F(EG_+,X)$.  In particular, in integer-graded part and $G/G$-level, this spectral sequence converges to $\pi_*^G \tilde{E}G\wedge F(EG_+,X)=\pi_* X^{tG}$
\end{itemize}

The natural map $P^{\bullet} X\rightarrow F(EG_+,P^{\bullet}X)$ induces a comparison map between spectral sequences SliceSS$(X)$ and HFPSS$(X)$. The following lemma states that this comparison map is an isomorphism in a certain range (\cite[Theorem~9.4]{Ull13} for the integer-graded part and \cite[Theorem~3.3]{DLS2022}]
for $RO(G)$-gradings).
\begin{lem}[\cite{Ull13},\cite{DLS2022}]
\label{lemma:isorange}
The map from the $RO(G)$-graded slice spectral sequence to the $RO(G)$-graded homotopy fixed point spectral sequence 
\[\begin{tikzcd}
\pi_{V-s}^G P^{|V|}_{|V|} X \ar[r] \ar[d, Rightarrow] & \pi_{V-s}^G F(EG_+, P^{|V|}_{|V|} X) \ar[d, Rightarrow ] \\ 
\pi_{V-s}^G X \ar[r] & \pi_{V-s}^G F(EG_+, X)
\end{tikzcd}\]
induces an isomorphism on the $E_2$-page in the region defined by the inequality 
\[\tau(V-s-1) > |V|,\hspace{8pt}
\tau(V): = \min_{\{e\} \subsetneq H \subset G} |H| \cdot \dim V^H. 
\]
Furthermore, the map induces a one-to-one correspondence between the differentials in this isomorphism region.  

\end{lem}

We recall two kinds of  distinguished classes in the $RO(G)$-graded homotopy groups that are useful for naming the relevant classes on the $E_2$-page of the slice spectral sequence (see \cite[Section~3.4]{HHR16} and \cite[Section~2.2]{HSWX2018}) and the homotopy fixed point spectral sequence.  
\begin{defn}\rm\label{defn:aclass}
Let $V$ be a $G$-representation.  We denote the inclusion of the fixed points $S^0\rightarrow S^{V}$ by $a_{V}$.  This is a class in $\pi_{-V}^GS^0$.  For a ring spectrum $X$ with $G$-action, we abuse notation to denote the image of $a_V$ by $a_V$ under the map $S^0\rightarrow X$. We will also denote the class on the $E_2$-page of the $G$-HFPSS$(S^0)$ or the $G$-HFPSS$(X)$ that detects the image of $a_V$ by $a_V$. 
\end{defn}
By construction, we have the following property.
\begin{prop}\label{prop:aVHFPSS}
With the above notation, the class $a_V$ on the $E_2$-page of the $G$-$\mathrm{HFPSS}(X)$ is a permanent cycle.
\end{prop}

If the representation $V$ has non-trivial fixed points (i.e. $V^{G}\neq \{0\}$), then $a_V=0$.  Moreover, for any two $G$-representations $V$ and $W$, we have the relation $a_{V\oplus W}=a_{V}a_{W}$ in $\pi^{G}_{-V-W}(S^0)$. When $X=H\underline{\mathbb{Z}}$, the $a_V$-class in $\pi_{-V}^{G}(H\underline{\mathbb{Z}})$ is always a torsion class, according to \cite[Lemma~3.6]{HHR17}
\[
|G/G_V|a_V=0
\]
where $G_V$ is the isotropy subgroup of $V$.

For an orientable $G$-representation $V$, a choice of orientation for $V$ gives an isomorphism $H_{|V|}^{G} (S^V; \underline{\mathbb{Z}}) \cong \mathbb{Z}$.  In particular, the restriction map 
\begin{equation}\label{eq:res}
H^{G}_{|V|}(S^V,\underline{\mathbb{Z}})\longrightarrow H_{|V|}(S^{|V|},\mathbb{Z})
\end{equation}
is an isomorphism.  
\begin{defn}\rm
Let $V$ be an orientable $G$-representation. We define the orientation class of $V$ $u_V\in H_{|V|}^{G}(S^V; \underline{\mathbb{Z}})$ to be the generator that maps to $1$ under the above restriction isomorphism ~\ref{eq:res}.
\end{defn}

The orientation class $u_V$ is stable in $V$ in the sense that if 1 is the trivial representation, then $u_{V\oplus 1}=u_{V}$.  If $V$ and $W$ are two orientable $G$-representations, then $V \oplus W$ is also orientable with the direct sum orientation, and $u_{V\oplus W}=u_{V}u_{W}$.  

Norms of $a_V$ classes and $u_V$ classes are given as follows. 
\begin{prop}{\rm({\cite[Lemma~3.13]{HHR16}}\rm)}
\label{prop:auclassnorm}
Let $H\subset G$ be a subgroup and $V$ is a $G$-representation
\begin{align*}
    N_H^G(a_V)&= a_{\Ind V};\\
   u_{\Ind |V|} N_H^G(u_V)&=u_{\Ind V}
\end{align*}
where $\Ind $ means $\Ind_H^G$.
\end{prop}
Given a $G$-oriented representation $V$ and a $G$-equivariant commutative ring spectrum $X$, by \cite[Corollary~4.54]{HHR16} and the unit map $S^0\rightarrow X$, Hill--Hopkins--Ravenel defines the $u_V$ classes on the $E_2$-page of the slice spectral sequence for $X$ via the following map on $0^\text{th}$ slices 
\[
H\underline{\mathbb{Z}}=P^0_0 S^0 \rightarrow P^0_0 X.
\]
With \cref{lemma:isorange}, we can define $u_V$ classes in the $RO(G)$-graded HFPSS for $X$.

The computation of the TateSS and the HFPSS are closely related. In any $RO(G)$-graded page the natural map from HFPSS$(X)$ to TateSS$(X)$ is an isomorphism in positive filtration (\cite[Theorem~3.6]{DLS2022}. See also \cite[Lemma~2.12]{BM94}).

\begin{lem}\label{lem:tateisorange}
The map from the $RO(G)$-graded homotopy fixed point spectral sequence to the $RO(G)$-graded Tate spectral sequence induces an isomorphism on the $E_2$-page for classes in filtration $s >0$, and a surjection for classes in filtration $s = 0$.  Furthermore, there is a one-to-one correspondence between differentials whose source is in non-negative filtrations.
\end{lem}

One advantage of considering Tate spectral sequences is that they are whole-plane spectral sequences with more invertible classes. This feature makes the calculations more accessible.
 
If $V$ is a $G$-representation such that its fixed point set $V^{H}$ is trivial for any non-trivial subgroup $H$ of $G$, then $S^{\infty V}$ is a geometric model for $\tilde{E}G$. If $X$ is a $G$-spectrum, we have
\[
\tilde{E}G\wedge X\simeq S^{\infty V}\wedge X=a_V^{-1}X
\]
This implies that for such representation $V$, the class $a_V$ is invertible in the Tate spectral sequence. 

\begin{met}\label{met:Tatemethod}

When $X$ is a commutative ring spectrum, its $G$-TateSS is multiplicative, which is extremely useful for proving permanent cycles and determining differentials in its  $G$-HFPSS.

Assume that we find a  non-trivial differential $d_r(a)=b$ in the $G$-HFPSS. Then there is a corresponding differential $d_r(a')=b'$ in the $G$-TateSS by \cref{lem:tateisorange}. We can move this differential by some $r$-cycle $c'$ in the $G$-TateSS such that $d_r(c'a')=c'b'$ is a differential with the source $c'a'$ in a negative filtration and the target $c'b'$ in a non-negative filtration. (One can choose $c'=a_V^{-k}$ for proper integer $k$ where $a_V$ is an invertible class as above.) Then $c'b'$ is a permanent cycle in the $G$-TateSS and hence the corresponding class of $c'b'$ in the $G$-HFPSS is also a permanent cycle by \cref{lem:tateisorange}. This method allows us to identify permanent cycles at $E_r$-page for $r< \infty$. Moreover, if $c'$ is an invertible permanent cycle in the $G$-TateSS for $X$, then the class $c'b'$  will survive to the $E_r$-page and the differential $d_r(c'a')=c'b'$ happens on the $E_r$-page in the $G$-TateSS. If it is not,  then there is a shorter differential, say $d_{r'}(d)=c'b'$, which kills the class $c'b'$ on the $E_r'$-page for $r'<r$. Then the Leibniz rule forces $d_{r'}((c')^{-1}d)=b'$, which is a contradiction. Furthermore, if this differential  $d_r(c'a')=c'b'$
completely locates in filtration greater than $0$, then \cref{lem:tateisorange} shows that there is a corresponding differential $d_r(ca)=cb$ in $G$-HFPSS.

\end{met}

Now we focus on $G=Q_8$ and its subgroups.  We will use the following notations for representations of $C_2,C_4$ and $Q_8$.
\begin{itemize}
\item When $G=C_2$, $RO(C_2)=\mathbb{Z}\{1,\sigma_2\}$ where $\sigma_2$ is the sign representation. 

\item When $G=C_4$, $RO(C_4)=\mathbb{Z}\{1,\sigma,\lambda\}$. The representation $\sigma$ is the sign representation and $\lambda$ is the $2$-dimensional representation by rotating the plane $\mathbb{R}^2$ by $\frac{\pi}{2}$.
\item When $G = Q_8$, $RO(Q_8) = \mathbb{Z}\{1, \sigma_i, \sigma_j, \sigma_k, \mathbb{H}\}$.  The representations $\sigma_i$, $\sigma_j$, and $\sigma_k$ are one-dimensional representations whose kernels are $C_4\langle i \rangle$, $C_4 \langle j \rangle$, and $C_4 \langle k \rangle$, i.e, the three $C_4$ subgroups generated by $i,j$ and $k$, respectively. The representation $\mathbb{H}$ is a four-dimensional irreducible representation, obtained by the action of $Q_8$ on the quaternion algebra $\mathbb{H} = \mathbb{R} \oplus \mathbb{R}i \oplus \mathbb{R}j \oplus \mathbb{R}k$ by left multiplication.
\end{itemize}

 By the above discussion, $S^{\infty \mathbb H}$ is a model of $\tilde{E}Q_8$. Therefore, the class $a_{\mathbb{H}}$ is invertible in any $Q_8$-Tate spectral sequence.

\subsection{Norm differentials and strong vanishing lines in spectral sequences}\label{subsec:normdiff}

The Hill--Hopkins--Ravenel norm structure holds in nice equivariant spectral sequences. Let $H\subset G$ be a subgroup. Consider the following diagram of $G$-spectra
\begin{align*}
\cdots\rightarrow P^{n+1}\rightarrow P^n \rightarrow P^{n-1}\rightarrow \cdots
\end{align*}
Recall that $P^m_n$ denotes the fiber of $P^m\rightarrow P^{n-1}$ and $P_n=P^{\infty}_n.$

We denote the spectral sequence associated to this tower by $\{E_r^{n,\bigstar},d_r\}$, where $n$ denotes the filtration and the second grading denotes the $RO(G)$-graded stem. 
We say the spectral sequence has a norm structure if there are two types of maps $N_H^G P_n\rightarrow P_{|G/H|n}$ and $N_H^G P_n^n\rightarrow P_{|G/H|n}^{|G/H|n}$ such that the following two diagrams commute up to homotopy.
\begin{equation*}
    \xymatrix{N_H^G P_n \ar[r]\ar[d] &P_{|G/H|n}\ar[d]\\
    N_H^GP_{n-1}\ar[r]& P_{|G/H|(n-1)}
    } \qquad    \xymatrix{N_H^G P_n \ar[r]\ar[d] &P_{|G/H|n}\ar[d]\\
    N_H^GP_{n}^n \ar[r]& P_{|G/H|n}^{|G/H|n}
    }   
\end{equation*}
The norm structure induces a map between the towers
 \begin{center}
\begin{tikzcd}[row sep=1em, column sep=1em]
				&\cdots \arrow[r] & N_H^G P_n\ar[d]  \arrow[rrrr] &&&& N_H^G P_{n-1} \ar[d]\arrow[r]&\cdots\\
				&\cdots \arrow[r] &P_{|G/H|n}\ar[r]&P_{|G/H|n-1}\ar[r]&\cdots \arrow[r]& P_{|G/H|(n-1)+1} \arrow[r] & P_{|G/H|(n-1)}\arrow[r] &\cdots
\end{tikzcd}
\end{center}
which induces a map from the $E_2$-page of the
$H$-level spectral sequence $H$-$E_2^{*,\bigstar}$ to the $E_2$-page of the $G$-level spectral sequence $G$-$E_2^{*,\bigstar}$ as follows

\[N_H^G: H\text{-}E_2^{n, V+n} \longrightarrow G\text{-}E_2^{|G/H|n, \Ind_H^GV +|G/H|n}.\]
It is proved in \cite{MSZ20} that if $X$ is a commutative $G$-ring spectrum then its slice spectral sequence, homotopy fixed point spectral sequence, and Tate spectral sequence (at least for $H\neq e$) have a norm structure.

One consequence of having a norm structure is that we can predict differentials in the $G$-level from differentials in the $H$-level.

\begin{thm}{\rm({\cite[Proposition~I.5.17]{Ull13}\cite[Theorem~4.7]{HHR17}}\rm )} \label{thm:normdiff}
In a spectral sequence with norm structures, if we have a differential $d_r(x)=y$ in the spectral sequence of a $H$-spectrum $X$. Then in the spectral sequence for $Y=N_H^G(X)$ there is a predicted differential
\[
d_{|G/H|(r-1)+1}(a_{\bar{\rho}}N_H^G(x))=N_H^G(y)
\]
where $\rho=\Ind_H^G(1)$ and $\bar{\rho}$ is the reduced representation of $\rho$.
\end{thm}

In \cite{DLS2022} the authors use the norm structures to show that every class in $G$-TateSS$(\mathbf{E}_n)$ is hit before a specific page depending on $n$ and $G$.
\begin{thm}
{\rm(\cite[Theorem~5.1]{DLS2022}\rm)}\label{thm:tatevanishing}
At the prime $2$, for any height $n$ and any $G \subset \mathbb{G}_n$ a finite subgroup, let $H$ be a Sylow 2-subgroup of $G$.  All the classes in the $RO(G)$-graded Tate spectral sequence of $\mathbf{E}_n$ vanish after the $E_{N_{n, H}}$-page.  Here  $N_{n, H}$ is a positive integer defined as follows: 
\begin{itemize}
\item when $(n, H) = (2^{m-1}\ell, C_{2^m})$, $N_{n, H} = 2^{n+m} - 2^m + 1$;
\item when $(n, H) = (4k+2, Q_8)$, $N_{n, H} = 2^{n+3}-7$.
\end{itemize}
\end{thm}

The isomorphism range of the natural map $G\text{-HFPSS}(\mathbf{E}_n)\rightarrow G\text{-TateSS}(\mathbf{E}_n)$ implies there is a strong horizontal vanishing line in $E_{\infty}$-page of $G$-HFPSS$(\mathbf{E}_n)$.

\begin{thm}{\rm(\cite[Theorem~6.1]{DLS2022}\rm)}\label{thm:vanishinglines}
At the prime $2$, for any height $n$ and any $G \subset \mathbb{G}_n$ a finite subgroup, let $H$ be a Sylow 2-subgroup of $G$.  There is a strong horizontal vanishing line of filtration $N_{n,H}$ in the $RO(G)$-graded homotopy fixed point spectral sequence of $\mathbf{E}_n$.

\end{thm}
It turns out that the existence of such horizontal vanishing lines is extremely helpful for determining higher differentials in homotopy fixed point spectral sequences. In particular,  for our computation in $Q_8$-HFPSS$(\mathbf{E}_2)$, the vanishing line gives an independent proof of several higher differentials in the integer gradings. Moreover, this vanishing line plays a crucial role in the computation of $(*-\sigma_i)$-graded $Q_8$-HFPSS for $\mathbf{E}_2$.

  \subsection{\texorpdfstring{Morava $E$-theory $\mathbf{E}_2$ with $G_{24}$-action}{}} \label{subsec:LubinTate}

We fix a pair $(\mathbb{F}_{p^n},\Gamma_n)$ where $\Gamma_n$ is the height-$n$ Honda formal group law over $\mathbb{F}_p$ extended to $\mathbb{F}_{p^n}$. Then Lubin and Tate \cite{LT65} show that there is a universal deformation $F_n$ defined over a complete local ring 
\[
\W(\mathbb{F}_{p^n})[\![u_1,\dots,u_{n-1}]\!]
\]
where $\W(\mathbb{F}_{p^n})$ is the $p$-typical ring of Witt vectors of $\mathbb{F}_{p^n}$. The Landweber exactness theorem shows that this ring can be realized by a complex-oriented ring spectrum $\mathbf{E}_n$.

Let $\mathbb{S}_n$ be the automorphism group of $\Gamma_n$, namely the small $n^\text{th}$ Morava stabilizer group. Let $\mathbb{G}_n=\mathbb{S}_n \rtimes \Gal(\mathbb{F}_{p^n}/\mathbb{F}_p)$ be the automorphism group of $(\mathbb{F}_{p^n},\Gamma_n)$, namely the (extended) $n^\text{th}$ Morava stabilizer group. By universality, $\pi_* \mathbf{E}_n$ admits a $\mathbb{G}_n$-action. The Goerss--Hopkins--Miller theorem \cite{Rez98,GH04,Lur18} lifts this action uniquely to an $\mathbb{E}_{\infty}$-action on $\mathbf{E}_n$.

 We are interested in computing $\pi_*\mathbf{E}_n^{hG}$ for $G$ a finite subgroup of $\mathbb{G}_n$ via $G$-homotopy fixed point spectral sequences. For these computations, the action of the Galois group $ \Gal(\mathbb{F}_{p^n} /\mathbb{F}_p)$ will not change the differential pattern. More precisely, we review the following result.

\begin{lem}{\rm(\cite[Lemma 1.32]{BG18}\cite[Lemma 2.2.6, Lemma 2.2.7]{BGH22})}
Let $F\subset \mathbb{G}_n$ be a closed subgroup and let $F_0=F\cap \mathbb{S}_n$. Suppose the following canonical map is an isomorphism
$$F/F_0 \rightarrow \mathbb{G}_n/\mathbb{S}_n\cong \Gal(\mathbb{F}_{p^n}/\mathbb{F}_p).$$
Then there is a commutative diagram of homotopy fixed point spectral sequences
\[\begin{tikzcd}
\mathbb{W}(\mathbb{F}_{p^n}) \otimes_{\mathbb{Z}_p} H^*(F,\pi_* \mathbf{E}_n) \ar[r,Rightarrow]\ar[d, rightarrow,"\cong"] & \mathbb{W}(\mathbb{F}_{p^n})\otimes_{\mathbb{Z}_p} \pi_*\mathbf{E}_n^{hF}\ar[d, rightarrow,"\cong"] \\ 
H^*(F_0,\pi_* \mathbf{E}_n) \ar[r,Rightarrow] & \pi_*\mathbf{E}_n^{hF_0}.
\end{tikzcd}\]
\end{lem}

In this paper, we will focus on the case $p=2$ and $n=2$. The Galois group $\Gal(\mathbb{F}_4/\mathbb{F}_2)$ is isomorphic to $C_2$ and we write $\mathbb{W}$ for the Witt vectors $\mathbb{W}(\mathbb{F}_4)$. There are finite subgroups $Q_8$ and $G_{24}\cong Q_8 \rtimes C_3$ in the small Morava stabilizer group $\mathbb{S}_2$ and $SD_8=Q_8\rtimes \Gal$ and $G_{48}\cong G_{24} \rtimes \Gal$ in the extended Morava stabilizer group $\mathbb{G}_2$. The subgroups $Q_8$, $G_{24}$ are unique up to conjugacy in $\mathbb{S}_2$ \cite{Buj12} (see also \cite[Remark 2.4.5]{BGH22}). Therefore, there is no ambiguity of the notation $\pi_*\mathbf{E}_2^{hQ_8}$ or $\pi_*\mathbf{E}_2^{hG_{24}}$. The subgroup $Q_8$ and complex orientation coordinates can be chosen specifically from the theory of elliptic curves at the prime $2$ so that the action has explicit formulas as follows (see \cite[Section 2]{Bea17} for more details).

We recall the action of $G_{24}$ on $\pi_* \mathbf{E}_2$ \cite[Lemma A.1]{Bea17}. The coefficient ring is a complete local ring $\pi_* \mathbf{E}_2=\W[\![u_1]\!][u^{\pm 1}]$ with $|u|=2$, $|u_i|=0$ and a maximal ideal $I=(2,u_1)$. Denote $u_1u^{-1}$ by $v_1$, the generator of the quaternion group $Q_8$ by $i,j,k$ and the generator of $C_3$ by $\omega$. We regard the third root of unity $\zeta$ as a class in the Witt vectors $\mathbb{W}$.
The $G_{24}$-actions on $u^{-1}$ and $v_1$ are
\begin{equation}\label{eq:typeone}
    \begin{aligned}
  \omega_*(u^{-1})&=\zeta^2 u^{-1},& \omega_*(v_1)&=v_1,\\
 i_*(u^{-1})&=\frac{v_1-u^{-1}}{\zeta^2-\zeta}, & i_*(v_1)&=\frac{v_1+2u^{-1}}{\zeta^2-\zeta},\\
    j_*(u^{-1})&=\frac{\zeta v_1-u^{-1}}{\zeta^2-\zeta}, & j_*(v_1)&=\frac{v_1+2\zeta^2 u^{-1}}{\zeta^2-\zeta},\\
    k_*(u^{-1})&=\frac{\zeta^2 v_1-u^{-1}}{\zeta^2-\zeta}, & k_*(v_1)&=\frac{v_1+2\zeta u^{-1}}{\zeta^2-\zeta}.
    \end{aligned}
\end{equation}

We define $D$ to be $\prod\limits_{g\in Q_8/C_2} g_*(u^{-1})$ which is $Q_8$-invariant. Then $(\mathbf{E}_2)_*$ could be expressed as  
  \[
 \pi_* \mathbf{E}_2\cong (\mathbb{W}[v_1,u^{-1}][D^{-1}])^{\wedge}_{I},
  \]
which is more convenient for the $Q_8$-cohomology computation.

\begin{lem}\label{lem:E2completion}
There is an isomorphism
\[
H^*(Q_8,\pi_* \E_2) \cong (H^*(Q_8,\mathbb{W}[v_1,u^{-1}])[D^{-1}])^\wedge_J
\]
where $J=(2,v_1^4)$.
\end{lem}

\begin{proof}
We observe that two ideals $I$ and $J$ share the same radical ideal, i.e., $\sqrt{I}=\sqrt{J}$.  Therefore, we have the isomorphism
\[
H^*(Q_8,(\mathbb{W}[v_1,u^{-1}][D^{-1}]^{\wedge}_I)\cong H^*(Q_8,\mathbb{W}[v_1,u^{-1}][D^{-1}]^{\wedge}_J).
\]
 Because $D$ is  $Q_8$-invariant, we have
\[
H^*(Q_8,\mathbb{W}[v_1,u^{-1}][D^{-1}]) \cong H^*(Q_8,\mathbb{W}[v_1,u^{-1}])[D^{-1}].
\]
Moreover, the ideal $J$ is actually $Q_8$-invariant, since 
\[
v_1^4\equiv N_{C_2}^{Q_8}(v_1) \mod {2}
\]
according to (\ref{eq:typeone}). It implies $ H^*(Q_8,\mathbb{W}[v_1,u^{-1}])[D^{-1}]$ is a $J$-module. Note that $\mathbb{W}[v_1,u^{-1}][D^{-1}]$ is finitely generated as a $\mathbb{W}$-algebra. Therefore, the completion is an exact functor \cite[Theorem ~10.12]{AM69} \cite[Theorem ~A.1]{HS99} and we have
\[
H^*(Q_8,\pi_* \mathbf{E}_2) \cong (H^*(Q_8,\mathbb{W}[v_1,u^{-1}])[D^{-1}])^\wedge_J.
\]
\end{proof}


\subsection{\texorpdfstring{Mackey functor $C_4$-homotopy fixed point spectral sequence for $\mathbf{E}_2$}{}}\label{subsec:Mackeycomputation}

In this subsection, we recall some results on the Mackey-functor-valued $C_4$-HFPSS for $E_2$  in \cite{BBHS20}. See also the slice spectral sequence computation of the truncated $C_4$-normed Real Brown--Petersen spectrum $\BPone$ \cite{HHR17}\cite{HSWX2018}. 
\begin{prop}{\rm({\cite[Proposition~5.6]{BBHS20}}\rm)}\label{prop:E2C2Slicename}
There is an isomorphism 
\[
H^{*}(C_{2},\pi_{\bigstar}\mathbf{E}_2)\cong \mathbb{W}[\![\mu_{0}]\!][\bar{r}_{1}^{\pm 1}, a_{\sigma_2}, u_{2\sigma_2}^{\pm 1}]/(2a_{\sigma_2}),
\]
where the $( \bigstar-*, *)$-degree of the classes is given by \(|\mu_0|=(0,0)\), \(|\bar{r}_{1}|=(\rho_{2},0)\), \(|a_{\sigma_2}|=(-\sigma_2,1)\), and \(|u_{2\sigma_2}|=(2-2\sigma_2,0)\).
\end{prop}

We partially rewrite the names of classes on the $E_2$-page of $C_4$-HFPSS$(\mathbf{E}_2)$ in \cite[Proposition~5.10] {BBHS20} with slice names. For slice names, see \cite{HHR17,HSWX2018} for details. One advantage of using slice names is that it is better to organize differentials by the slice differential theorem \cite[Theorem~9.9]{HHR16}.

\begin{prop}{\rm(\cite[Proposition~5.10]{BBHS20}\rm)}\label{prop:E2C4Slicename}
There is an isomorphism
\[
H^*(C_4, \pi_{\bigstar}\mathbf{E}_2)\cong\W[\![\mu]\!][T_2,\eta,\eta',a_{\lambda},a_{\sigma}][\done^{\pm 1},u_{\lambda}^{\pm 1},u_{2\sigma}^{\pm 1}]/\sim
\]
where $\mu=\Tr_{C_2}^{C_4}(\mu_0)$, $T_2=\sone^2 u_{2\sigma_2}=\Tr_{C_2}^{C_4}(\bar{r}_1^2 u_{2\sigma_2})$, $\eta=\sone a_{\sigma_2}=\Tr_{C_2}^{C_4}(\bar{r}_1 a_{\sigma_2})$ and $\eta'=\sone u_{\sigma}a_{\sigma_2}=\Tr_{C_2}^{C_4}(\bar{r}_1 a_{\sigma_2} u_{\sigma})$. Although $\sigma$ is not an oriented $C_4$-representation, we apply $u_{\sigma}$ here indicating that $\eta'$ is transfered from $\bar{r}_1 a_{\sigma_2}$ from integer-graded part in $C_2$-level to $(1-\sigma)$-page in $C_4$-level. The relation $\sim$ is the ideal generated by the following relations
 \begin{align*}
    2\eta &=2\eta'=2a_{\sigma}=4a_{\lambda}=0, &  T_2^2 & =  \Delta_1((\mu-2)^2 + 4),  \\
    \eta^2 u_{2\sigma}&=\eta'^2=T_2u_{\lambda}^{-1}u_{2\sigma}a_{\lambda},& T_2\eta'&=\done \mu \eta u_{\lambda}u_{2\sigma},\\
    T_2\eta&=\done \mu\eta' u_{\lambda}, &\eta \eta'&=\mu u_{2\sigma}a_{\lambda},\\
    u_{\lambda}a_{2\sigma}&=2a_{\lambda}u_{2\sigma}, &\mu a_{\sigma}&=\eta a_{\sigma}=\eta' a_{\sigma}=T_2 a_{\sigma}=0.
\end{align*}
Here $\Delta_1=\done^2 u_{2\lambda}u_{2\sigma}$ at $(8,0)$ is an invertible class in $\pi_*\mathbf{E}_2^{hC_4}$.
\end{prop}

\begin{rem}\rm

\cref{prop:E2C2Slicename} and \cref{prop:E2C4Slicename} give a full description of the Mackey functor  $\underline{H}^*(C_4,\pi_{\bigstar}\mathbf{E}_2)$ by the Frobenius relation \cite[Remark~5.17] {BBHS20} and the multiplicative property of restriction.
\end{rem}

\begin{rem}\rm
A warning is that one needs to be careful about the isomorphism range (see \cref{lemma:isorange}) to translate between the slice spectral sequence and the homotopy spectral sequence. For example, in the $C_4$-SliceSS$(\BPone)$, the class $u_{2\sigma}$ supports a non-trivial $d_5$-differential \cite[Theorem~3.4]{HSWX2018}, while in the corresponding $C_4$-HFPSS$(\mathbf{E}_2)$, the class $u_{2\sigma}$ actually supports a non-trivial $d_7$-differential \cite[Remark~5.23]{BBHS20}.
\end{rem}
The computation of the Mackey-functor-valued $C_4$-homotopy fixed point spectral sequence for $\mathbf{E}_2$ is explained in detail in \cite[Section~5]{BBHS20} and presented by \cite[Figure~5.8]{BBHS20} and \cite[Figure~5.14]{BBHS20}.

The $RO(G)$-graded Mackey functor computation is useful even if one only cares about the computation of the integer-graded part $\pi_*\mathbf{E}_n^{hG}$. The following discussion of hidden extensions is a good example.
We can use exotic operations (exotic transfers, exotic restrictions, and so on) in Mackey-functor-valued spectral sequences to deduce differentials and hidden extensions inside the spectral sequences. For more detailed definitions and properties of such phenomena, one could refer to \cite[Section~3.3]{MSZ20}.

In \cite[Lemma~4.2]{HHR17}, the authors introduce a useful trick to determine exotic restrictions and transfers on the $E_{\infty}$-page of Mackey-functor-valued $G$-HFPSS.

\begin{lem}{\rm(\cite[Lemma~4.2]{HHR17}\rm)}\label{lem:exoticrestr}
Let $G$ be a cyclic $2$-group and $G'$ be its index $2$ subgroup then in $\underline{\pi}_{\bigstar}(F(EG_+,X))$ we have
\begin{itemize}
 \item $\ker(\Res^G_{G'})=\text{im}(a_{\sigma})$
    \item $\text{im}(\Tr_{G'}^G)=\ker (a_{\sigma})$
\end{itemize}
where $\sigma$ is the sign representation of $G$.
\end{lem}

The following hidden $2$ extension in stem $22$ is a good example showing that equivariant structures provide extra integer-graded information (see a similar $2$ extension in stem $2$ in \cite[Remark ~5.15]{MSZ20}). In \cite[Figure~15]{HHR17} and \cite[Figure~5.6]{BBHS20}, they drew all exotic restrictions and transfers in the $E_{\infty}$-page of the Mackey functor valued $C_4$-HFPSS$(\mathbf{E}_2)$. The $2$ extension follows from an exotic transfer and an exotic restriction in the $22$ stem. We spell out the details in \cref{lem:hidden2C_4}.

\begin{lem}\label{lem:hidden2C_4}
In the Mackey-functor-valued $C_4$-$\mathrm{HFPSS}$ for $\mathbf{E}_2$, there is an exotic restriction in stem $22$ from $\done^6 u_{6\lambda}u_{4\sigma}a_{2\sigma}$ to $\done^{4}\bar{r}_1^6 u_{8\sigma_2}a_{6\sigma_2}$ and there is an exotic transfer in stem $22$ from $\done^{4}\bar{r}_1^6 u_{8\sigma_2}a_{6\sigma_2}$ to $\done^8 u_{4\lambda} u_{6\sigma}a_{4\lambda}a_{2\sigma}$. As a consequence, there is a hidden $2$ extension from $\done^6 u_{6\lambda}u_{4\sigma}a_{2\sigma}$ to $\done^8 u_{4\lambda} u_{6\sigma}a_{4\lambda}a_{2\sigma}$.

\end{lem}

\begin{proof}
According to the computations in \cite{HHR17}\cite{BBHS20}, in stem $22$   there are only three classes who survive: $\done^6 u_{6\lambda}u_{4\sigma}a_{2\sigma}$ and $\done^8 u_{4\lambda} u_{6\sigma}a_{4\lambda}a_{2\sigma}$ in $C_4$-level and $\done^{4}\bar{r}_1^6 u_{8\sigma_2}a_{6\sigma_2}$ in $C_2$-level. We first claim the class $\done^6 u_{6\lambda}u_{4\sigma}a_{2\sigma}$ is not in the image of multiplication by $a_{\sigma}$. If there is some $x$ such that $a_\sigma x$ is $\done^6 u_{6\lambda}u_{4\sigma}a_{2\sigma}$, then $x$ is detected by classes at $(22+\sigma,1)$ or $(22+\sigma,0)$. There is only one class at $(22+\sigma,1)$ which is  $\done^4u_{4\lambda}u_{4\sigma}a_{\sigma}$ on $E_2$-page. According to \cite[Theorem~3.11]{HSWX2018}, this class supports a $d_{13}$-differential
\[
d_{13}(\done^4u_{4\lambda}u_{4\sigma}a_{\sigma})=\done^4u_{4\sigma}d_{13}(u_{4\lambda}a_{\sigma})=\done^7 u_{8\sigma}a_{7\lambda}.
\]
There is no non-trivial class at $(22+\sigma,0)$. Therefore, in homotopy the class $\done^6 u_{6\lambda}u_{4\sigma}a_{2\sigma}$ is not in the image of multiplication by $a_{\sigma}$.
By \cref{lem:exoticrestr}, this class must have a non-trivial restriction in $C_4$-Mackey functor $\underline{\pi}_*(\E_2)$, and the desired exotic restriction follows from degree reasons. 

On the other hand by the gold relation $u_{\lambda}a_{2\sigma}=2u_{2\sigma}a_{\sigma}$ and $2a_{\sigma}=0$  we know on $E_2$-page
\[
\done^8 u_{4\lambda} u_{6\sigma}a_{4\lambda}a_{2\sigma}\cdot a_{\sigma}=0.
\]
Moreover, according to the computation on $(*-\sigma)$-page of $C_4$-HFPSS$(\mathbf{E}_2)$ \cite{BBHS20},  there is no hidden $a_{\sigma}$-extension from $\done^8 u_{4\lambda} u_{6\sigma}a_{4\lambda}a_{2\sigma}$ by degree reasons.  Since we have $\text{im}(\Tr_{G'}^G)=\ker (a_{\sigma})$, the class  $\done^8 u_{4\lambda} u_{6\sigma}a_{4\lambda}a_{2\sigma}$ must be a transfer of a class from $C_2$-level. Then the desired exotic transfer follows from degree reasons.

According to \cite{BBHS20}, $E_2$, as a $C_4$-spectrum, its Mackey functor valued homotopy groups $\underline{\pi}_*\mathbf{E}_2$ satisfy
\[
\Tr\circ \Res (1)=2. 
\]
Then the exotic transfer and restriction that we proved shows the existence of the hidden $2$ extension from $\done^6 u_{6\lambda}u_{4\sigma}a_{2\sigma}$ to $\done^8 u_{4\lambda} u_{6\sigma}a_{4\lambda}a_{2\sigma}$.
\end{proof}
\begin{rem}\rm\label{rem:exoticasigma}
For degree reasons, the class $\done^6 u_{6\lambda}u_{4\sigma}a_{2\sigma}$ cannot be in the image of the transfer from $C_2$. However, by the gold relation, the product of this class and $a_\sigma$ is zero on the $E_2$-page. Therefore, this class must have a hidden $a_{\sigma}$ extension.
\end{rem}
\begin{rem}\rm
The hidden $2$ extension in \cref{lem:hidden2C_4} will play a crucial role in deducing several higher differentials in $Q_8$-HFPSS$(\mathbf{E}_2)$ (see \cref{lem:hiiden2Q8}, \cref{prop:d13two}). A similar $2$ extension can also be seen in the homotopy groups of $tmf$ in stem $54$. The proof of this hidden $2$ extension in \cite[Proposition~8.5~(3)]{Bau08} uses shuffling arguments of $4$-fold Toda brackets. In our $Q_8$-HFPSS($\mathbf{E}_2$) computation, the corresponding hidden $2$ extension follows directly from the $C_4$-computation by restriction (see \cref{lem:hiiden2Q8}). 

\end{rem}

\subsection{\texorpdfstring{$RO(G)$-graded periodicity}{}}\label{subsec:ROGperiodicity}

When computing HFPSS, another advantage of expanding to $RO(G)$-gradings is having more periodicities. These periodicities have their own theoretical importance. They can also move integer-graded calculations to certain $RO(G)$-gradings where the calculations might be simpler. In either the slice spectral sequence for $\BPone$ \cite{HSWX2018} or the $C_4$-homotopy fixed point spectral sequence for $\E_2$ \cite{HHR17,BBHS20}, we have the following periodicities in the $RO(G)$-gradings.
\begin{lem}
The following permanent cycles in $C_4$-$\mathrm{HFPSS}(\E_2)$ \cite{HHR17,BBHS20} are periodic classes.
\begin{itemize}
\item The class $\done$ gives $ (1+\sigma+\lambda)$-periodicity.
    \item The class $u_{8\lambda} $ gives $( 16-8\lambda)$-periodicity.
    \item The class $u_{4\sigma}$ gives $(4-4\sigma)$-periodicity.
    \item The class $u_{4\lambda}u_{2\sigma}$ gives $(10-4\lambda-2\sigma)$-periodicity.
\end{itemize}
\end{lem}

Since the norm functor is symmetric monoidal, we can apply it to the above three invertible permanent cycles, which gives some $RO(Q_8)$-periodicities in $Q_8$-HFPSS$(\E_2)$. The quaternion group $Q_8$ has three $C_4$ subgroups $C_4\langle i\rangle$, $C_4\langle j\rangle$ and $C_4\langle k\rangle$ generated by $i,j$ and $k$ respectively. For each $C_4$ copy we have the associated $C_4$-periodicities and their norms give $RO(Q_8)$-periodicities as follows.

\begin{cor}\label{cor:periodicity}
We have the following $RO(Q_8)$-periodicities in $Q_8$-$\mathrm{HFPSS}(\E_2)$.
\begin{itemize}
\item  $N_{C_4}^{Q_8}(\done):$
\[
1+\sigma_i+\sigma_j+\sigma_k+\mathbb{H}
\]
\item $N_{C_4}^{Q_8}(u_{4\sigma}):$ 
\begin{equation*}
\begin{aligned}
&4+4\sigma_i-4\sigma_j-4\sigma_k\\
&4+4\sigma_j-4\sigma_i-4\sigma_k\\
&4+4\sigma_k-4\sigma_i-4\sigma_j
\end{aligned}
\end{equation*}
\item $N_{C_4}^{Q_8}(u_{4\lambda}u_{2\sigma}):$
\begin{equation*}
\begin{aligned}
& 10+10\sigma_i-2\sigma_j-2\sigma_k-4\mathbb{H}\\
&  10+10\sigma_j-2\sigma_i-2\sigma_k-4\mathbb{H}\\
& 10+10\sigma_k-2\sigma_j-2\sigma_i-4\mathbb{H}
\end{aligned}
\end{equation*}
\item $N_{C_4}^{Q_8}(u_{8\lambda}):$

\begin{equation*}
\begin{aligned}
& 16+16\sigma_i-8\mathbb{H}\\
& 16+16\sigma_j-8\mathbb{H}\\
& 16+16\sigma_k-8\mathbb{H}
\end{aligned}
\end{equation*}
\end{itemize}
\end{cor}

\begin{cor}\label{cor:u4sigma}
There are periodicities of $4-4\sigma_i, 4-4\sigma_j$ and $4-4\sigma_k$ in $Q_8$-$\mathrm{HFPSS}(\E_2)$.
\end{cor}
\begin{proof}
It suffices to show that $4-4\sigma_i$ is a periodicity. This periodicity is given by the following product: 
\[
N_{C_4\langle j\rangle}^{Q_8}(u_{4\lambda}u_{2\sigma})N_{C_4\langle k\rangle}^{Q_8}(u_{4\lambda}u_{2\sigma})N_{C_4\langle i\rangle}^{Q_8}(u_{8\lambda})^{-1}N_{C_4\langle i\rangle}^{Q_8}(u_{4\sigma})^2 N_{C_4\langle j\rangle}^{Q_8}(u_{4\sigma})^{-1}N_{C_4\langle k\rangle}^{Q_8}(u_{4\sigma})^{-1}.
\]
\end{proof}
\begin{rem}\rm \label{rem:4sigmapc}
    The above multiplication equals to $u_{4\sigma_i}$ by \cite[Lemma~3.13]{HHR16}, in other words, the classes $u_{4\sigma_i},u_{4\sigma_j}$ and $u_{\sigma_k}$ are permanent cycles in the $RO(Q_8)$-graded HFPSS for $\E_2$. 
\end{rem}

\section{\texorpdfstring{$E_2$-page of the $Q_8$-HFPSS$(\E_2)$}{}}
\label{sec:E2page}

In this section, we recollect the computation of the $E_2$-page of the integer-graded  $Q_8$-HFPSS for $\E_2$ by the $2$-Bockstein spectral sequence ($2$-BSS) from \cite{Bea17,Bau08}. Then we compute the $E_2$-page of the $(*-\sigma_i)$-graded part by the same method. By \Cref{lem:E2completion} we can compute $H^*(Q_8, \pi_* \E_2)$, the $E_2$-page of the $Q_8$-HFPSS for $\E_2$, by first computing $H^*(Q_8, \mathbb W[v_1,u^{-1}])[D^{-1}].$  
\subsection{$2$-BSS, integer-graded}

The integer-graded $2$-Bockstein spectral sequence for computing $H^*(Q_8, \mathbb W[v_1,u^{-1}])[D^{-1}]$ is
$$H^*(Q_8, \mathbb{F}_4[v_1,u^{-1}])[D^{-1}][h_0]\implies H^*(Q_8, \mathbb W[v_1,u^{-1}])[D^{-1}],$$
where $h_0$ detects $2$.
The computation of the $E_1$-page, $H^*(Q_8, \mathbb{F}_4[v_1,u^{-1}])[D^{-1}]$, is from \cite[Appendix A]{Bea17}. We follow the notation in \cite{Bea17}, except that we use $h_1$ for $\eta$ and $h_2$ for $\nu$. The differentials of this $2$-BSS are essentially from \cite[Section ~7]{Bau08} and we list them in \Cref{table:2BSS_integer}. 

More precisely, $H^*(G_{48}, \mathbb{F}_4[v_1,u^{-1}])$ is computed in explicit generators and relations in \cite[Theorem A.20]{Bea17} (the coefficient in \cite{Bea17} is denoted by $S_*(\rho)$). Further, we have $H^*(G_{24}, \mathbb{F}_4[v_1,u^{-1}]) \cong \mathbb{F}_4 \otimes_{\mathbb{F}_2}H^*(G_{48}, \mathbb{F}_4[v_1,u^{-1}])$. 

 We will show that $H^*(Q_8,\mathbb{F}_4[v_1,u^{-1}])[D^{-1}]$ is  $3$ copies of the $H^*(G_{24},\mathbb{F}_4[v_1,u^{-1}])[\Delta^{-1}]$ as follows, where $\Delta=D^3$.

\begin{lem}
 $H^*(Q_8,\mathbb{F}_4[v_1,u^{-1}])[D^{-1}]$ is a free $H^*(G_{24},\mathbb{F}_4[v_1,u^{-1}])[\Delta^{-1}]$-module of rank $3$ with generators $1, D$ and $D^2$.  
\end{lem}

\begin{proof}
We denote the coefficient $\mathbb{F}_4[v_1,u^{-1}]$ in the group cohomology by $A$ for simplicity. Because the coefficient $A$ is $2$-local, the Lyndon--Hochschild--Serre spectral sequence for the group extension
\[
1\rightarrow Q_8 \rightarrow G_{24}\rightarrow C_3\rightarrow 1
\]
collapses at the $E_2$-page. Therefore, we have 
\[
H^*(G_{24},A)=H^*(Q_8,A)^{C_3}.
\]
Note that $D$ is $Q_8$-invariant and hence in $H^0(Q_8,A)$. This gives an $H^*(G_{24},A)$-module map
$$H^*(G_{24},A)\oplus H^*(G_{24},A)\{D\}\oplus  H^*(G_{24},A)\{D^2\}\rightarrow H^*(Q_8,A).$$
We first prove the injectivity. Note that the $C_3$-action on $D$ is given by 
\[
\omega_*(D)=\zeta^2 D.
\]
Hence the three copies above have different eigenvalues so the images of these three copies must intersect trivially. We then prove the surjectivity after inverting $\Delta$. In $H^*(Q_8,A)$, we have $\Delta=D^3$, so $D$ is also inverted after inverting $\Delta$. We first show that $H^*(Q_8,A)$ is a direct sum of the above eigenspaces with respect to the $C_3$-action. Note that $\mathbb{F}_4[Q_8]$ as a $C_3$-module is a direct sum of eigenspaces with eigenvalues $1,\zeta, \zeta^2$, so are the entries of the bar resolution of $\mathbb{F}_4$ as $\mathbb{F}_4[Q_8]$-modules. Moreover, the coefficient $\mathbb{F}_4[v_1,u^{-1}]$ is also a direct sum of eigenspaces with eigenvalues $1, \zeta^2$ by (\ref{eq:typeone}). Therefore, every entry in the cochain complex for computing group cohomology $H^*(Q_8,A)$ is a direct sum of eigenspaces with eigenvalues $1,\zeta, \zeta^2$. So is $H^*(Q_8,A)$. After inverting $D$ and $\Delta$, $H^*(G_{24},A)$, $H^*(G_{24},A) \{D\}$, and $H^*(G_{24},A) \{D^2\}$ give the eigenspaces of eigenvalues $1, \zeta^2, \zeta$ respectively. Therefore, the map is also surjective.
\end{proof}

The lemma above and the computation of $H^*(G_{24}, \mathbb{F}_4[v_1,u^{-1}])$ in \cite[Theorem A.20]{Bea17} give the $Q_8$ case as follows.

\begin{prop}\label{prop:2BSSE1}
The bigradings of generators of $H^*(Q_8, \mathbb{F}_4[v_1,u^{-1}])[D^{-1}]$ are:
\[
\begin{tabular}{llll}
  $|v_1|=(2,0),$ &$|D|=(8,0)$, &$|k|=(-4,4)$, & $|h_1|=(1,1)$, \\
  $|h_2|=(3,1)$, &$|x|=(-1,1)$, & $|y|=(-1,1).$
\end{tabular}
\]
The relations $(\sim)$ are generated by
\begin{enumerate}
    \item in filtration $1$: 
    \[
    \begin{tabular}{lcr}
    $v_1h_2$, & $v_1^2 x,$& $v_1 y$;
    \end{tabular}
    \]
    \item in filtration $2$:
    \[
\begin{tabular}{lllll}
     $h_1 h_2,$ & $h_2 x-v_1h_1 x,$ &$h_1 y-v_1 x^2,$ & $xy,$ &$D y^2-h_2^2$;
\end{tabular}
    \]
    \item in filtration $3$: 
    \[
    \begin{tabular}{lcr}
      $h_1^2 D x-h_2^3,$&$D x^3-h_2^2 y;$
    \end{tabular}
    \]
    
    \item in filtration $4$:
    \[
    h_1^4-v_1^4 k.
    \]
\end{enumerate}
\end{prop}

The differentials in the integer-graded $2$-BSS for the cohomology $H^*(Q_8, \mathbb{W}[v_1,u^{-1}])[D^{-1}]$ are essentially from \cite[Section ~7]{Bau08} which are determined by the ones in \Cref{table:2BSS_integer} and the multiplicative structure.

The $2$-Bockstein computation gives the following result (see also \cite[Section~7]{Bau08}).
\begin{thm}
\label{thm:2BSSint}
\Cref{table:2BSS_integer_Einf} and \Cref{table:2BSS_integer_Einf_rel} present the $E_\infty$-page of the integer-graded $2$-Bockstein spectral sequence (also see \cref{fig:2BSSE1} and \cref{fig:integer2BSS}), which is the associated graded algebra of $H^*(Q_8, \mathbb{W}[v_1,u^{-1}])[D^{-1}]$.
\end{thm}

\begin{rem}\rm\label{rm:invertibleclass}
    According to the properties of Tate cohomology, we know the class $k$ is invertible in the associated Tate cohomology for $Q_8$ with the same coefficient. 
\end{rem}

\begin{rem}\rm \label{rmk:h2ext}
We note that in $H^*(Q_8, \mathbb{W}[v_1,u^{-1}])[D^{-1}]$, there is a hidden  $h_2$ extension 
$$h_2\cdot x^2 h_2=4 kD$$
by \cite[Equation~(7.13)]{Bau08} which is useful in later computations. See \cref{fig:2BSSE1} and \cref{fig:integer2BSS} for the information of $h_2$ extensions. 
\end{rem}
\begin{longtable}{llll}
\caption{$2$-BSS differentials, integer-graded
\label{table:2BSS_integer} 
} \\
\toprule
$(s,f)$ & $x$ & $r$ & $d_r(x)$ \\
\midrule \endhead
\bottomrule \endfoot
$(4k+2, 0)$ & $v_1^{2k+1}$ & 1 &  $2 v_1^{2k}h_1$\\

$(7,1)$ & $D x$ & 1 & $2 h_2^2$\\
$(-1,1)$ & $x$ & 1 & $2y^2$\\
$(-1,1)$ & $y$ & 1 & $2x^2$\\
\midrule
$(4,0)$ & $v_1^2$ & 2 & $4h_2$\\
\midrule
$(5,3)$ & $y h_2^2$ & 3 & $8kD$\\
\end{longtable}

\begin{longtable}{lll}
\caption{$E_\infty$-page, multiplicative generators, integer-graded
\label{table:2BSS_integer_Einf} 
} \\
\toprule
$(s,f)$ & $x$ &$2$-torsion \\
\midrule \endhead
\bottomrule \endfoot
$(-4,4)$ & $k$ & $\mathbb Z/8$ \\
$(-2,2)$ & $x^2$ & $\mathbb Z/2$ \\
$(-2,2)$ & $y^2$ & $\mathbb Z/2$ \\
$(0,2)$ & $xh_1$ & $\mathbb Z/2$\\
$(1,1)$ & $h_1$ & $\mathbb Z/2$\\
$(3,1)$ & $h_2$ & $\mathbb Z/4$ \\
$(5,1)$ & $v_1^2h_1$ & $\mathbb Z/2$ \\
$(8,0)$ & $D$ & $\mathbb Z$\\
$(8,0)$ & $v_1^4$ & $\mathbb Z$\\
\end{longtable}

\begin{longtable}{lc}
\caption{$E_\infty$-page, relations, integer-graded
\label{table:2BSS_integer_Einf_rel} 
} \\
\toprule
$f$ & relations \\
\midrule \endhead
\bottomrule \endfoot
$1$ & $v_1^4 h_2$  \\
$2$ & $h_1 h_2$, $v_1^2h_1\cdot h_2$, $D y^2-h_2^2$, $xh_1\cdot v_1^4$, $x^2\cdot v_1^4$, $y^2\cdot v_1^4$\\
$3$ & $xh_1\cdot h_2$, $xh_1\cdot v_1^2h_1$, $x^2\cdot  v_1^2h_1$, $y^2 h_1$, $y^2 \cdot v_1^2h_1$, $D \cdot xh_1\cdot h_1-h_2^3$\\
$4$ &$h_1^4-v_1^4 k$, $(xh_1)^2$, $(x^2)^2$,$(y^2)^2$, $h_2^4$,  $x^2 \cdot xh_1$, $y^2 \cdot xh_1$, $x^2 \cdot y^2$, \\
&$xh_1\cdot h_1^2$, $x_2\cdot h_1^2$, $y^2 \cdot h_1^2$,$h_2^2\cdot x^2-4kD$, $y^2 \cdot h_2^2$, $xh_1\cdot h_2^2$
\end{longtable}

We refer readers to \S \ref{section:charts} for charts of the $E_1$-page and the $E_\infty$-page.

\subsection{\texorpdfstring{2-BSS, $(*-\sigma_i)$-graded}{}}\hfill

We discuss the $RO(G)$-graded case and restrict it to the $(*-\sigma_i)$-graded case. 
A variation of \Cref{lem:E2completion} still holds in this case. Thus we can compute $H^*(Q_8, \pi_{*-\sigma_i}\E_2)$ by first computing the $(*-\sigma_i)$-graded 2-BSS and taking the completion. Note that after modulo $2$, the representation $\sigma_i$ is oriented and the orientation class $u_{\sigma_i}$ gives an isomorphism between $\pi_* \E_2/2$ and $\pi_{*+1-\sigma_i} \E_2/2$ as $Q_8$-modules. Therefore, the $E_1$-page of the $(*-\sigma_i)$-graded 2-BSS is abstractly isomorphic to that of the integer-graded part. We denote the $E_1$-page by 
$$H^*(Q_8, \mathbb{F}_4[v_1,u^{-1}])[D^{-1}]\{\usig\}[h_0]$$
where $\usig$ denote a generator of the class at $(1-\sigma_i,0)$.
\begin{prop}
\label{prop:2bss-diff-usig}
	In the $2$-BSS, there is a differential
	$$d_1(\usig)=2 x\usig+2y\usig.$$
\end{prop}

\begin{proof}
The group cohomology computation shows that $H^1(Q_8, \pi_{1-\sigma_i}\E_2)$ is $2$-torsion according to \cref{RestrictionList}. Hence in the $2$-BSS, there must be a $d_1$-differential hitting the bigrading $(-\sigma_i,1)$. Then $\usig$ in the $2$-BSS  must support a non-trivial $d_1$-differential by degree reasons. 
Assume that $d_1(\usig)=2a x\usig +2b y\usig$ where $a,b$ are either $0$ or $1$. 
By the Leibniz rule, we have $d_1(v_1\usig)=2h_1 \usig+2a x v_1 \usig$. Since $h_1$ is a permanent cycle, the Leibniz rule implies that $h_1\usig$ also supports a non-trivial $d_1$-differential. Therefore, the $d_1$-target of $v_1\usig$ cannot be $2h_1 \usig$. We deduce that $a=1.$

On the other hand, if $b=0$, then the Leibniz rule implies
\[
d_1(y\usi)=d_1(y)\usi+yd_1(\usi)=2x^2\usi,
\]
which means the class $x^2\usi$ is a $1$-cycle. However, $x^2$ is a $1$-cycle in the integral $2$-BSS. Then the Leibniz rule also implies that
\[
d_1(x^2\usi)=x^2d_1(\usi)=x^2(2x\usi)=2x^3\usi.
\]
This is a contradiction. Therefore $b$ must be $1$, and the claimed $d_1$-differential follows.

\end{proof}
\begin{rem}
The careful reader may observe that there are some Koszul sign rule-related concerns here; however, we opt to overlook them as they will have no bearing on our subsequent calculations.
\end{rem}
The remaining $(*-\sigma_i)$-graded $2$-BSS $d_1$-differentials can be determined by the Leibniz rule and the differential on $\usig$ in \cref{prop:2bss-diff-usig}.

\begin{prop}
\label{prop:2BSS_sigma_d2}
	There is a $2$-BSS differential
	$$d_2(xh_1^2 u_{\sigma_i})=4kv_1^2 u_{\sigma_i}.$$
\end{prop}
\begin{proof}
	By \Cref{prop:grouphomology_sigma}, the class at $(1-\sigma_i,4)$ is $4$-torsion in the $E_\infty$-page. This forces the desired $d_2$-differential.

\end{proof}
\begin{lem}\label{lem:hiddenh1}
    There is a hidden $h_1$ extension from $k^m x^2h_1 D^n u_{\sigma_i}$ to $2k^{m+1}v_1^2 D^n u_{\sigma_i}$, and a hidden $h_2$ extension from $k^m x^3 D^n u_{\sigma_i}$ to $2k^{m+1}v_1^2 D^n u_{\sigma_i}$ for any $m \in \mathbb{N}, n\in \mathbb{Z}$.     \end{lem}

\begin{proof}
It suffices to prove the case for $m=n=0$. 
    We have the following differentials by \cref{prop:2bss-diff-usig}, \cref{prop:2BSS_sigma_d2} and the Leibniz rule
        	$$d_1(xh_1 u_{\sigma_i})=2x^2h_1 u_{\sigma_i},$$
    	$$d_2(xh_1^2 u_{\sigma_i})=4kv_1^2 u_{\sigma_i}.$$
    This implies that in the associated $2$-inverted $2$-BSS, there are 
        	$$d_1(2^{-1}xh_1 u_{\sigma_i})=2x^2h_1 u_{\sigma_i},$$
    	$$d_2(2^{-1}xh_1^2 u_{\sigma_i})=4kv_1^2 u_{\sigma_i}.$$
 Denote 
     $\mathbb{W}[v_1,u^{-1}][D^{-1}]\{u_{\sigma_i}\}$  by $A$. Consider the long exact sequence on $Q_8$ group cohomology induced by the short exact sequence on the coefficients
     $$A \rightarrow A[1/2] \rightarrow A/2^\infty.$$ 
In $H^*(Q_8, A/2^{\infty})$ we have $h_1$ multiplication from $2^{-1}xh_1u_{\sigma_i}$ to $2^{-1}xh_1^2 u_{\sigma_i}$.  Then the differentials
        	$$d_1(2^{-1}xh_1 u_{\sigma_i})=x^2h_1 u_{\sigma_i},$$
    	$$d_2(2^{-1}xh_1^2 u_{\sigma_i})=2kv_1^2 u_{\sigma_i}$$ 
     imply that the boundary map in the long exact sequence sends the two classes $2^{-1}xh_1 u_{\sigma_i}$ and $2^{-1}xh_1^2 u_{\sigma_i}$ in $H^*(Q_8, A/2^{\infty})$    to the two classes $x^2h_1 u_{\sigma_i}$ and $2kv_1^2 u_{\sigma_i}$ in $H^*(Q_8,A)$ respectively. And there is an $h_1$ extension from $x^2h_1 u_{\sigma_i}$ to $2kv_1^2 u_{\sigma_i}$ . 

As for the $h_2$ extension, we apply the similar argument to the differentials 
\[
d_1(y^2u_{\sigma_i})=x^3 u_{\sigma_i},
\]
\[
d_2(xh_1^2 u_{\sigma_i)}=4kv_1^2u_{\sigma_i}.
\]

    \end{proof}

We list non-trivial differentials on classes of the form $\{\text{multiplicative generators}\}\usig$  in the table below.

\begin{longtable}{llll}
\caption{$2$-BSS differentials, $(*-\sigma_i)$-graded
\label{table:2BSS_sigma} 
} \\
\toprule
$(s,f)$ & $x$ & r & $d_r(x)$ \\
\midrule \endhead
\bottomrule \endfoot
$(1-\sigma_i, 0)$ & $\usig$ & 1 &  $2 x\usig+2y\usig$\\
$(-\sigma_i,1)$ & $xu_{\sigma_i}$& 1 & $2x^2u_{\sigma_i}+2y^2u_{\sigma_i} $\\
$(3-\sigma_i,0)$ & $v_1\usig$ & 1 & $2h_1 \usig + 2xv_1\usig$\\
$(4-\sigma_i,1)$ & $h_2 \usig$ & 1 & $2xh_2 \usig+2yh_2\usig$\\
\midrule
$(2-\sigma_i,3)$ & $xh_1^2\usig$ & 2 & $4kv_1^2\usig$\\
\end{longtable}

\begin{thm}
\label{thm:2BSSsig}
\Cref{table:2BSS_sigma_Einf} and \Cref{table:2BSS_sigma_Einf_rel} present the $E_{\infty}$-page the $(*-\sigma_i)$-graded $2$-Bockstein spectral sequence (also see \cref{fig:sigma2BSS} and \cref{fig:sigmaE2} with hidden $h_1$ and $h_2$ extensions), which is the associated graded algebra of $H^*(Q_8, \mathbb{W}[v_1,u^{-1}]\otimes \sigma_i)[D^{-1}]$.
\end{thm}
\begin{proof}
The result follows from the 2-BSS computation.
\end{proof}

\begin{longtable}{lll}
\caption{$E_\infty$-page, module generators, $(*-\sigma_i)$-graded
\label{table:2BSS_sigma_Einf} 
} \\
\toprule
$(s,f)$ & $x$ &$2$-torsion \\
\midrule \endhead
\bottomrule \endfoot
$(-2,2)$ & $\{x^2+y^2\}\usig$ & $\mathbb Z/2$ \\
$(-1,1)$ & $\{x+y\}\usig$ & $\mathbb Z/2$\\
$(1,1)$ & $\{h_1+xv_1\}\usig$ & $\mathbb Z/2$\\
$(0,2)$ & $v_1^2\usig$ & $\mathbb Z$
\end{longtable}

\vspace{0.5cm}
\begin{longtable}{lc}
\caption{$E_\infty$-page, relations, $(*-\sigma_i)$-graded
\label{table:2BSS_sigma_Einf_rel} 
} \\
\toprule
$f$ & relation generators \\
\midrule \endhead
\bottomrule \endfoot
$1$ & $\{h_1+xv_1\}\usig \cdot v_1^4-v_1^2\usig\cdot v_1^2h_1$, $v_1^2\usig\cdot h_2$, $\{x+y\}\usig\cdot v_1^4$\\
$2$ & $\{h_1+xv_1\}\usig \cdot h_2$, $\{h_1+xv_1\}\usig \cdot v_1^2 h_1- v_1^2\usig\cdot h_1^2$, \\
&
$v_1^2\usig\cdot x^2$, $v_1^2\usig\cdot y^2$, $v_1^2\usig\cdot xh_1$, $v_1^2\usig\cdot h_2^2$, $\{x+y\}\usig \cdot v_1^2h_1$, $\{x^2+y^2\}\usig \cdot v_1^4$\\
$3$ & $\{h_1+xv_1\}\usig \cdot x^2 - \{x^2+y^2\}\usig\cdot h_1$, $\{h_1+xv_1\}\usig \cdot y^2$, $\{x^2+y^2\}\usig \cdot v_1^2h_1$\\
& $\{x+y\}\usig\cdot h_1^2 - \{h_1+xv_1\}u_{\sigma_i}\cdot xh_1$,
$\{x+y\}\usig\cdot x^2-\{x+y\}\usig\cdot y^2$\\
$4$ & $\{x^2+y^2\}\usig \cdot h_1^2$, $\{x^2+y^2\}\usig \cdot h_2^2$, $\{x^2+y^2\}\usig \cdot x^2$, $\{x^2+y^2\}\usig \cdot y^2$, $\{x^2+y^2\}\usig \cdot xh_1$ \\
& $\{h_1+xv_1\}\usig\cdot h_1^3-v_1^2\usig $, $4v_1^2\usig \cdot k$\\
\end{longtable}

We refer the readers to \S \ref{section:charts} for charts of the $E_1$-page and the $E_\infty$-page.

By \cref{lem:E2completion}, in both the  integer-graded and the $(*-\sigma_i)$-graded case, the $E_2$-page of $Q_8$-HFPSS($E_2$)  follows from \cref{thm:2BSSint} and \cref{thm:2BSSsig}.
\begin{rem}
\label{rem:tateE2}
The $E_2$-page of TateSS($\E_2$) follows by further inverting the class $k$ from that of $Q_8$-HFPSS$(\E_2)$, and replacing the $0$-line with the cokernel of the norm map.
\end{rem}

\section{\texorpdfstring{Computation of the integer-graded $Q_8$-HFPSS$(\E_2)$}{}}\label{section:differentials}

In this section, we derive all differentials in the integer-graded $Q_8$-HFPSS for $\E_2$ via the following two methods.
\begin{enumerate}
    \item Equivariant methods: apply the restrictions, transfers, and norms to deduce differentials in the $Q_8$-HFPSS for $\E_2$ from the $C_4$-HFPSS for $\E_2$;
    \item The vanishing line method: use the fact that the $Q_8$-HFPSS for $\E_2$ admits a strong vanishing line of filtration $23$ (\cref{thm:sharpvanishingline}, for general cases, see \cite[Theorem 6.1]{DLS2022}) to force differentials.
\end{enumerate}
We also solve all hidden $2$ extensions via equivariant methods and investigation of the Tate spectral sequence.

We will rename several classes on the $E_2$-page of the $Q_8$-HFPSS for $\E_2$ as follows. The advantage is that these names are compatible with the $tmf$ computation and the Hurewicz images in $\E_2^{hQ_8}$ (see \cite{Bau08}, also compare to \cite{Isa18}).
For example, we rename the class $kD^3$ by $g$, which is compatible with \cite{Bau08} and suggests that this class detects the Hurewicz image of $\bar{\kappa}$ (see \ref{lemma:gPC}).

\begin{longtable}{lcc}
\caption{Distinguished classes}\label{table:classnames}\\
\toprule
Classes & Bauer's notation& Bigrading\\
\midrule \endhead
\bottomrule \endfoot
$Dxh_1$ & $c$&$(8,2)$\\
$D^2 x^2$& $d$&$(14,2)$\\
$kD^3$& $g$ &$(20,4)$\\
\end{longtable}
When we talk about the restriction map from $Q_8$ to $C_4$, the subgroup $C_4$ usually indicates the subgroup $C_4\langle i\rangle $ generated by $i$ if there is no further specification. Some of the arguments in the proofs of this section are easier to see when accompanied by charts in \S \ref{section:charts}.

\subsection{General properties of the $Q_8$-HFPSS for $\E_{4k+2}$}\label{subsection:generalpropertyQ8HFPSS}

It is a result of Shi--Wang--Xu, using the Slice Differential Theorem and the norm functor of Hill--Hopkins--Ravenel \cite{HHR16}, that the homotopy fixed point spectrum $\E_{4k+2}^{hQ_8}$ is $2^{4k+6}$-periodic. 

The periodicity of $\E_2^{hQ_8}$ is known by computation to be $64$ classically. Here we give a proof that $\E_2^{hQ_8}$ is $64$-periodic before computing it using $Q_8$-HFPSS.

\begin{prop}\label{prop:periodicity}
The homotopy groups of the spectrum $\E_2^{hQ_8}$ is $64$-periodic and the periodicity class can be given by the class $D^8$.
\end{prop}
\begin{proof}
The product
\[
N_{C_4}^{Q_8}(\done)^8 N_{C_4\langle i\rangle}^{Q_8}(u_{4\sigma})^2 N_{C_4\langle i\rangle}^{Q_8}(u_{8\lambda})N_{C_4\langle j\rangle}^{Q_8}(u_{4\sigma})^4 N_{C_4\langle k\rangle}^{Q_8}(u_{4\sigma})^4
\]
gives the $64$ periodicity of $\E_2^{hQ_8}$.  This product is in bigrading $(64,0)$ and is invertible. On the other hand, the generator $D^8$ of $\pi_{64}\E_2$ is $Q_8$-invariant and invertible.
Therefore, this periodicity class is $D^8$ up to a unit.
\end{proof}

From now on we can simply view $D^8$ as a periodicity class of $\E_2^{hQ_8}$.
In the following property, we show that the $Q_8$-HFPSS for $\E_2$ splits into three parts such that there are no differentials across different parts.

Note that the universal space $EG_{24}$ can be viewed as a model for $EQ_8$. The transfer and the restriction of the  
spectrum $F(EG_{24},\E_2)$ give a sequence $\E_2^{hG_{24}} \xrightarrow{\Res} \E_2^{hQ_8} \xrightarrow{\Tr} \E_2^{hG_{24}}$, which is compatible with the filtration of the HFPSS.

\begin{prop}\label{prop:G24Q8}
The composition
$$\E_2^{hG_{24}} \xrightarrow{\Res} \E_2^{hQ_8} \xrightarrow{\Tr} \E_2^{hG_{24}}$$ is an equivalence. In particular, the $G_{24}$-$\mathrm{HFPSS}$ for $\E_2$ splits as a summand of the $Q_8$-$\mathrm{HFPSS}$ for $\E_2$.
\end{prop}
\begin{proof}
The composition $\Tr \circ \Res$ is multiplication by $|G_{24}|/|Q_8|=3$. All spectra are $2$-local and $3$ is coprime to $2$ so this composition is an equivalence.
\end{proof}

We identify the $E_2$-page of $Q_8$-HFPSS$(\E_2)$ as a free module over the $E_2$-page of $G_{24}$-HFPSS$(\E_2)$ generated by $\{1,D,D^2\}$.
\begin{cor}\label{cor:split}
Let $a,b$ be two classes on the $E_2$-page of $G_{24}$-$\mathrm{HFPSS}(\E_2)$.  View $a,b$ as classes in $Q_{8}$-$\mathrm{HFPSS}(\E_2)$ and consider classes $aD^{k_a}, bD^{k_b}$ where  $k_a, k_b\in \{0,1,2\}$. Then there is a differential $d_r (aD^{k_a}) = bD^{k_b}$ in the $Q_8$-$\mathrm{HFPSS}(\E_2)$ iff there is a differential $d_r (a) =b$ in the $G_{24}$-$\mathrm{HFPSS}(\E_2)$ and $k_a=k_b$.
\end{cor}
\begin{proof}
When $k_a=0$, this follows from \cref{prop:G24Q8}. For $k_a=1$, note that the $Q_8$-HFPSS for $\E_2$ is $D^8$-periodic by \cref{prop:periodicity}. The two differentials 
$$(1)~d_r (aD) = b D^{k_b} \text{~and~} (2)~d_r (aD^9)= b D^{k_b+8}$$ imply each other. We observe that the class $aD^9$ is a class in $G_{24}$-HFPSS$(\E_2)$. Then by the case $k_a=0$, the differential (2) happens in $G_{24}$-HFPSS$(\E_2)$. This implies the desired result. The case $k_a=2$ is similar.
\end{proof}
As a consequence, the computation of the $Q_8$-HFPSS for $\E_2$  splits into three copies with the same differential patterns and there are no differentials across different copies. In particular, the $G_{24}$-HFPSS for $\E_2$ is $192$-periodic.

\begin{rem}\rm
A similar statement holds for general height $4k+2$. A maximal finite subgroup in $\mathbb{S}_{4k+2}$ is $Q_8 \rtimes C_{3(2^{2k+1}-1)}\cong G_{24}\times C_{(2^{2k+1}-1)}$ \cite{Hew95}\cite[Section~4.3]{Buj12}. The computation of the $Q_8$-HFPSS for $E_{4k+2}$ also splits into copies of the computation of the $Q_8 \rtimes C_{3(2^{2k+1}-1)}$-HFPSS for $\E_{4k+2}$.
\end{rem}

\begin{rem}\rm
The $G_{24}$-HFPSS for $\E_2$ computation is essentially the same as the $2$-local $tmf$ computation \cite{Bau08}. However, our computation only relies on the $C_4$ computation of $\E_2$ and hence is an independent computation of the classical $tmf$ computations.
\end{rem}

In \cref{thm:sharpvanishingline}, we will improve the horizontal vanishing line result of the $Q_8$-HFPSS for $\E_{4k+2}$ in \cref{thm:vanishinglines}. In the case of the $Q_8$-HFPSS for $\E_2$, the improved vanishing line of filtration $23$ turns out to be sharp by computation. We start with the following fact.

\begin{prop}\label{prop:aclasstrivial}
Let $H\underline{\mathbb{Z}}$ be the Eilenberg-Mac Lane spectrum with trivial $Q_8$-action. Then on the $E_2$-page of $Q_8$-$\mathrm{HFPSS}(H\underline{\mathbb{Z}})$, the product $a_{\sigma_i}a_{\sigma_j}a_{\sigma_k}$ is trivial.
\end{prop}
\begin{proof}

We prove a stronger statement that the whole group $H^3(Q_8,\pi_{3-\sigma_i-\sigma_j-\sigma_k}(H\underline{\mathbb{Z}}))$, where the class $a_{\sigma_i}a_{\sigma_j}a_{\sigma_k}$ lies in, is trivial. According to \cref{RestrictionList},  the group $H^3(Q_8,\mathbb{Z})$ is trivial. We observe that the homotopy group $\pi_{3-\sigma_i-\sigma_j-\sigma_k}(H\underline{\mathbb{Z}})$ as a $Q_8$-module is a copy of $\mathbb{Z}$ with trivial $Q_8$-action ($\sigma_i\otimes \sigma_j\otimes \sigma_k$ is a trivial $Q_8$-representation). Then we have $$H^0(Q_8,\pi_{3-\sigma_i-\sigma_j-\sigma_k}(H\underline{\mathbb{Z}}))=(\pi_{3-\sigma_i-\sigma_j-\sigma_k}(H\underline{\mathbb{Z}})))^{Q_8}\cong \mathbb{Z}.$$
Similarly we also have $$H^0(Q_8,\pi_{-3+\sigma_i+\sigma_i+\sigma_k}(H\underline{\mathbb{Z}}))= (\pi_{-3+\sigma_i+\sigma_i+\sigma_k}(H\underline{\mathbb{Z}}))^{Q_8}\cong \mathbb{Z}.$$
Let $u$ be a generator of $H^0(Q_8,\pi_{3-\sigma_i-\sigma_j-\sigma_k}(H\underline{\mathbb{Z}}))$. Then the class $u$ is invertible on the $E_2$-page of HFPSS for $H\underline{\mathbb{Z}}$ by the following paring
\[
\pi_{3-\sigma_i-\sigma_j-\sigma_k}(H\underline{\mathbb{Z}})\otimes \pi_{-3+\sigma_i+\sigma_i+\sigma_k}(H\underline{\mathbb{Z}}) \cong \mathbb{Z}.
\]
Therefore, the class $u$ induces an isomorphism $H^3(Q_8,\pi_{3-\sigma_i-\sigma_j-\sigma_k}(H\underline{\mathbb{Z}}))\simeq H^3(Q_8,\mathbb{Z})$, the latter of which is trivial.

\end{proof}

\begin{rem}
We thank Guillou for confirming and explaining  \cref{prop:aclasstrivial}.  This proposition also follows from Guillou and Slone's computation of quaternionic Eilenberg--Mac Lane spectra \cite{BC22}.
\end{rem}

\begin{thm}\label{thm:sharpvanishingline}Let $k$ denote a non-negative integer.
\begin{enumerate}
\item
The $RO(Q_8)$-graded $Q_8$-$\mathrm{TateSS}$ for $\E_{4k+2}$ vanishes after $E_{2^{4k+5}-9}$-page. 
\item The $RO(Q_8)$-graded $Q_8$-$\mathrm{HFPSS}$ for $\E_{4k+2}$ admits a strong vanishing line of filtration $2^{4k+5}-9$.
\end{enumerate}
\end{thm}
\begin{proof}
\hfill

$(1)$~
Denote the height $4k+2$ by $h$. 
We briefly review the proof of the vanishing line of filtration $2^{h+3}-7$ in \cite[Theorem~6.1]{DLS2022} and explain the filtration improvement by $2$. By \cref{thm:normdiff}, in the $Q_8$-TateSS$(\E_h)$, there is a predicted differential
\begin{equation}
\label{eq:predicted}
	d_{2^{h+3}-7}(N_{C_2}^{Q_8}(\bar{v}_h^{-1}u_{2\sigma_2}^{2^h-1}a_{\sigma_2}^{1-2^{h+1}})a_{\bar{\rho}} ) = 1.
\end{equation}
By naturality, the unit $1$ has to be hit by a differential $d_r$ with $r\leq 2^{h+3}-7$. Note that since $1$ is hit, the spectral sequence vanishes at $E_r$-page.

The ring map $\mathbb{Z}\rightarrow \pi_*\E_{h}$ induces a map between $E_2$-pages of the $Q_8$-HFPSS for $H\underline{\mathbb{Z}}$ and $\E_{h}$. Then the naturality forces the source of (\ref{eq:predicted}) to be trivial since $a_{\bar{\rho}}=a_{\sigma_i}a_{\sigma_j}a_{\sigma_k}=0$  by \cref{prop:aclasstrivial}. For degree reasons, we conclude $r\leq 2^{h+3}-9$. So every class in the $Q_8$-TateSS$(E_h)$ will disappear on or before the $E_{2^{h+3}-9}$-page. 

$(2)$~ It follows from \cref{lem:tateisorange}.
\end{proof}

\begin{lem}\label{lemma:gPC}
In the $Q_8$-$\mathrm{HFPSS}$ for $\E_2$, the class $h_1$, $h_2$, $g$ are permanent cycles.
\end{lem}
\begin{proof}
Consider the following maps
$$S^0\xrightarrow{\mathrm{unit}} E_2^{hQ_8} \xrightarrow{\mathrm{res}} E_2^{hC_2}.$$
By \cite[Theorem 1.8]{LSWX19}, the class $\bar{\kappa} \in \pi_{20}S^0$ maps to a non-trivial class in $\E_2^{hC_2}$ in filtration $4$ in the $C_2$-HFPSS for $E_2$. Thus the image of $\bar{\kappa}$ in $\pi_*\E_2^{hQ_8}$ is non-trivial. For degree reasons, it is detected by the class $g$ in $Q_8$-HFPSS($E_2$). The proofs for $h_1$ and $h_2$ are similar.
\end{proof}

We only use the Hurewicz image of $\E_2^{hC_2}$ as the input. This has been systematically studied in \cite{LSWX19}. Our method does not assume the knowledge of the Hurewicz image of $\E_2^{hC_4}$.

\subsection{Differentials in the integer-graded pages}\label{subsection:differentials}

We suggest readers refer to the charts while reading the proofs in this section.

All statements about differentials in this subsection are differentials in integer-graded $Q_8$-HFPSS($\E_2$) if there is no specification.


\begin{prop}\label{prop:d3}
The class $v_1^6$ in $(12,0)$ supports a $d_3$-differential
\[
d_3(v_1^6)=v_1^4 h_1^3.
\]
\end{prop}
\begin{proof}
By construction, we have $\Res^{Q_8}_{C_4}(v_1^6)=T_2^3,  \Res^{Q_8}_{C_4}(h_1)=\eta$. In $C_4$-HFPSS$(\E_2)$, \cite[Proposition~5.21]{BBHS20} implies that we have
\[
d_3(T_2^3)=T_2^2 \eta^3.
\]
The result now follows by naturality.

\end{proof}

\begin{cor}\label{cor:d3}
The class $v_1^2 h_1$ at $(5,1)$ supports a $d_3$-differential
\[
d_3(v_1^2h_1)= h_1^4.
\]
\end{cor}
\begin{proof}
By \cref{prop:d3}, we have $d_3(v_1^6 h_1) = v_1^4 h_1^4.$ Note that $v_1^4$ is a $3$-cycle. This forces the desired $d_3$-differential.
\end{proof}

\cref{prop:d3} produces a family of $d_3$-differentials by the Leibniz rule:
$$d_3(D^m g^s v_1^{4l+2}h_1^n)=D^m g^sv_1^{4l}h_1^{n+3} \text{, and }
   d_3(D^m g^s v_1^2h_1^{n})= D^mg^sh_1^{n+3}$$
for any $(m,s,l,n)\in \mathbb{Z}\times \mathbb{Z}_{\geq 0}\times \mathbb{Z}_{\geq 1}\times \mathbb{Z}_{\geq 1}$.

For degree reasons (and the following proposition), these are all the non-trivial $d_3$-differentials.

\begin{prop}\label{prop:inftybo}
The following classes survive to the $E_{\infty}$-page. 
\[
2D^m v_1^{4l+2}, D^mv_1^{4l} , D^mv_1^{4l}h_1,D^mv_1^{4l}h_1^2 
~~, ~~(m,l)\in \mathbb{Z}\times \mathbb Z_{\geq 1}.\]
\end{prop}
\begin{proof}
The classes $D^mv_1^{4l} , D^mv_1^{4l}h_1,D^mv_1^{4l}h_1^2 $ cannot be hit by degree reasons. They are permanent cycles by \cref{lem:tateisorange} and the $Q_8$-TateSS($\E_2$) $d_3$-differentials 
\[
d_3(D^{m+3}g^{-1}v_1^{4l-2}h_1^{n+4})=D^m v_1^{4l} h_1^{n+3}, ~~ m,l\in \mathbb{Z}, ~l\neq0, n\geq -3.
\]
As for the classes $2D^m v_1^{4l+2}$, we consider the additive norm map
\[
H_0(Q_8,(\E_2)_*)\xrightarrow{N} H^0(Q_8, (\E_2)_*)
\]
where $N(x)=\sum\limits_{g\in Q_8} g(x)$.
By the $Q_8$-action formulas (\ref{eq:typeone}), we have 
\begin{align*}
&N(v_1^{2l+1}(u^{-1})^{2l+1})=\sum\limits_{g\in Q_8} g(v_1^{2l+1}(u^{-1})^{2l+1})\\
&=2v_1^{2l+1}(u^{-1})^{2l+1}+2\left(\frac{v_1+2u^{-1}}{\zeta^2-\zeta}\right)^{2l+1}\left(\frac{v_1-u^{-1}}{\zeta^2-\zeta}\right)^{2l+1}\\
&+2\left(\frac{v_1+2\zeta^2u^{-1}}{\zeta^2-\zeta}\right)^{2l+1}\left(\frac{\zeta v_1-u^{-1}}{\zeta^2-\zeta}\right)^{2l+1}+2\left(\frac{v_1+2\zeta u^{-1}}{\zeta^2-\zeta}\right)^{2l+1}\left(\frac{\zeta^2 v_1-u^{-1}}{\zeta^2-\zeta}\right)^{2l+1}.
\end{align*}
The leading term of the above formula on the $E_\infty$-page of the $2$-BSS for $H^*(Q_8, \mathbb{W}[v_1,u^{-1}])$ is $2v_1^{4l+2}$ for $l\geq 1$. Then we have $$N(D^mv_1^{2l+1}(u^{-1})^{2l+1})=D^m\sum\limits_{g\in Q_8} g(v_1^{2l+1}(u^{-1})^{2l+1})=2D^mv_1^{4l+2}$$
since $D$ is $Q_8$-invariant. As the additive norm map is the $d_1$-differential on $E_1$-page of the $Q_8$-TateSS for $\E_2$, we have the classes $2D^m v_1^{4l+2}$ are permanent cycles who survive to the $E_{\infty}$-page by \cref{lem:tateisorange}.
\end{proof}

\begin{rem}\rm
All the classes supporting or receiving non-trivial $d_3$-differentials and all classes in \cref{prop:inftybo} are sometimes referred to as the $bo$-pattern. They match the pattern of (many copies of) $\pi_*KO$, the homotopy groups of the real $K$-theory. See \cite[Definition ~2.1]{BG18} for more details. 
\end{rem}


The following result is the first example of the strong vanishing line method (\cref{thm:sharpvanishingline}). The method gives differentials of three lengths (including the longest $d_{23}$-differential) all at once (see \cref{fig:Vanishing line method}). 

\begin{prop}\label{prop:d13one}
There are differentials
\begin{enumerate}
    \item $d_5(D^{-13} g^5 d h_2 )=4D^{-16}g^7;$
   \item $d_{13}(D^{-7}  g^3 ch_1 )=2D^{-16}g^7;$
\item $d_{23}(D^{-1} gh_1 )=D^{-16}g^7.$
\end{enumerate}
\end{prop}
\begin{proof}
We suggest readers compare the arguments with \cref{fig:Vanishing line method}.
The class $D^{-16}g^7$ is a permanent cycle in filtration $28$, which is above the vanishing line (\cref{thm:sharpvanishingline}). Therefore, the classes $D^{-16}g^7$, $2D^{-16}g^7$ and $4D^{-16}g^7$ must receive differentials. According to \cref{cor:split}, $Q_8$-HFPSS$(\E_2)$ splits into three parts. On the $E_2$-page, these three parts are modules over the $E_2$-page of $G_{24}$-HFPSS$(\E_2)$, and all differentials do not cross different copies. In \cref{fig:Vanishing line method}, we highlight the relevant copy. By inspection, we obtain the desired $d_5$, $d_{13}$ and $d_{23}$-differentials.\end{proof}

\begin{figure}[!htb]
\begin{center}
{\includegraphics[trim={0cm, 1cm, 0cm, 3cm},clip,page=1,scale=1]{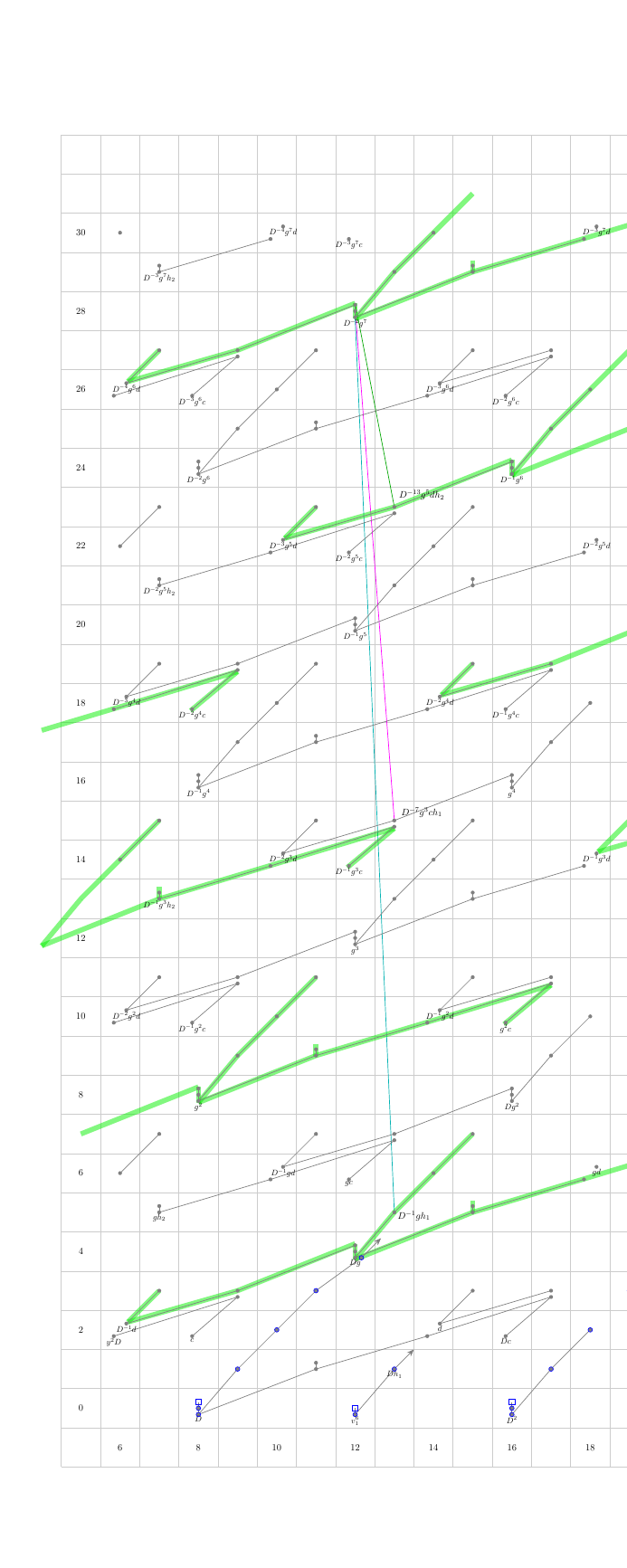}}
\caption{$d_5$, $d_{13}$, $d_{23}$-differentials}
\label{fig:Vanishing line method}
\hfill
\end{center}
\end{figure}

\begin{cor}\label{prop:d5}
The class $D$ at $(8,0)$ supports a $d_5$-differential
$$d_5(D)=D^{-2}gh_2.$$
\end{cor}
\begin{proof}
Note that $D^{8}$ is an invertible permanent cycle (\cref{prop:periodicity}), and $g^5$ is a permanent cycle (\cref{lemma:gPC}). By \cref{prop:d13one}(1) and the Leibniz rule, there is a $d_5$-differential 
\begin{equation}
\label{eq:d51}
d_5(D^3 d h_2 )=4g^2.
\end{equation}
The relation $dh_2^2=4g$ (see \cref{rmk:h2ext} under $2$BSS names) forces the following $d_5$-differential 
\begin{equation}
\label{eq:d5}
	d_5(D^3 d)=gdh_2.
\end{equation}

With (\ref{eq:d5}), it suffices to show that $D^2 d$ is a $5$-cycle. In fact, the only possible $d_5$ target of $D^2 d$ supports a differential
\[
d_5(D^{-1}gdh_2)=4D^{-4}g^3.
\]
by multiplying $D^{-4}gh_2$ with (\ref{eq:d5}).
Note that $D^{-4}$ is a $5$-cycle since $D$ is a $3$-cycle.

\end{proof}

All the remaining $d_5$-differentials follow from the Leibniz rule. There are no more $d_5$-differentials by degree reasons and \cref{cor:split}.
\\

We also get a $d_9$-differential from the $d_{13}$-differential in \cref{prop:d13one}(2).
\begin{cor}\label{prop:d9one}
The class $Dc$ at $(16,2)$ supports a $d_9$-differential
\[
d_9(D c)=D^{-5}g^2dh_1.
\]
\end{cor}
\begin{proof}
We observe that in $Q_{8}$-HFPSS($\E_2$) there is an $h_1$ extension from $Dc$ to $Dch_1$. We prove this by contradiction.
Suppose that $Dc$ does not support the claimed $d_9$-differential.
Then for degree reasons, $Dc$ becomes a $13$-cycle. However, this contradicts \cref{prop:d13one} since $Dch_1$ supports a non-trivial $d_{13}$-differential. 
\end{proof}

\begin{prop}\label{prop:d7one}
The classes $4D$ and $2D^2$ at $(16,0)$ support the following $d_7$-differentials
\begin{enumerate}
\item $d_7(4D)=D^{-2}gh_1^3;$

\item $d_7(2D^2)=D^{-1}gh_1^3.$
\end{enumerate}
\end{prop}
\begin{proof}
By \cref{prop:d5} and the hidden $2$ extension from $2h_2$ to $h_1^3$ (see \cite{Tod62}), $D^{-2}gh_1^3$ has to be hit by a differential. For degree reasons and \cref{cor:split}, the only possible source is $4D$. The second $d_7$-differential follows similarly from $d_5(D^2)=2D^{-1}gh_2$. 
\end{proof}

The $d_7$-differential on $D^4$ (which we prove in \cref{prop:d7two}) turns out to be a hard one, as it does not follow from primary relations like the Leibniz rule or (hidden) extensions.  We will first prove several $d_9,d_{13}$-differentials, and the $d_7$-differential follows from the vanishing line method.

\begin{prop}\label{prop:diff9}
The class $D^5 ch_1$ at $(49,3)$ and the class $D^5 c$ at $(48,2)$ support the following differentials. 
\begin{enumerate}
\item $d_{13}(D^5 ch_1)=2D^{-4}g^4;$
\item $d_9(D^5 c)=D^{-1}g^2 dh_1.$
\end{enumerate}
\end{prop}
\begin{proof}
By a similar argument as in \cref{prop:d9one}, it is enough to show (1). 
We first observe that the class $2D^{-4}g^4$ is in the image of the transfer map from $C_4$-HFPSS$(\E_2)$ since 
\[
\Tr\circ\Res(D^{-4}g^4)=[Q_8:C_4]D^{-4}g^4=2D^{-4}g^4.
\]
Since the class $D^{-4}g^4$ is order $8$, its restriction to $C_4$-HFPSS for $\E_2$ is non-trivial and of order 4. Then according to the computations in \cite{BBHS20} (See \cite[Figure~5.3.]{BBHS20}), on bigrading $(48,16)$, the class $\Res(D^{-4}g^4)$ receives a $d_{13}$-differential in $C_4$-HFPSS$(\E_2)$. The naturality forces that $2D^{-4}g^4$ dies on or before the the $E_{13}$-page in $Q_8$-HFPSS$(\E_2)$. The only possibility is the desired $d_{13}$-differential by \cref{cor:split} and  degree reasons.\end{proof}

\begin{rem}\rm
Since $C_4$-HFPSS$(\E_2)$ is $32$-periodic with the periodicity class $\Delta_1^4=\done^8 u_{8\lambda} u_{8\sigma}$ \cite{HHR17}\cite{BBHS20}, the same argument in the proof of \cref{prop:diff9} gives an alternative proof of \cref{prop:d13one}(2) and \cref{prop:d9one}.
\end{rem}

\begin{lem}\label{lem:pcford11}
The class $D^3 h_1$ is a permanent cycle.
\end{lem}
\begin{proof}
By \cref{cor:split}, it suffices to show that $D^3 h_1$ is a permanent cycle in $G_{24}$-HFPSS$(\E_2)$.
For degree reasons, $D^3 h_1$ can only possibly hit $D^{-3}gc$ or $2D^{-12}g^6$ in $G_{24}$-HFPSS$(\E_2)$. Because $D^{-8}$, $g$ are permanent cycles, \cref{prop:diff9} implies
\[
d_{13}(D^{-3}g^2 ch_1)= 2D^{-12}g^6
~\text{ and }~
d_9(D^{-3}gc)=D^{-9}g^3 dh_1.
\]
Therefore, the class $D^3h_1$ has to be a permanent cycle. \end{proof}
\begin{rem}\rm
It turns out that $D^3 h_1$ is hit by a $d_{23}$-differential in the Tate spectral sequence by \cref{cor:d23two}. 
\end{rem}
\begin{cor}\label{cor:d23two}
There are non-trivial $d_{23}$-differentials
\begin{enumerate}
    \item $d_{23}(D^2  h_1^2)=D^{-13}g^6 h_1;$
    \item $d_{23}(D^5 h_1^3)=D^{-10}g^6 h_1^2.$
\end{enumerate}
\end{cor}
\begin{proof}
The claimed $d_{23}$-differentials follow from \cref{prop:d13one}(3)
and \cref{lem:pcford11}
\end{proof}

\begin{lem}\label{lem:hiiden2Q8}
There is a hidden $2$ extension from $D^6 h_2^2$ to $g^2 d$.
\end{lem}
\begin{proof}
According to \cref{lem:hidden2C_4}, there is a hidden $2$ extension in stem $54$ from $\Delta_1^4\done^6 u_{6\lambda}u_{4\sigma}a_{2\sigma}$ to $\Delta_1^4 \done^8 u_{4\lambda}u_{6\sigma}a_{4\lambda}a_{2\sigma}$ in the $C_4$-HFPSS$(\E_2)$ since it is $\Delta_1^4$-periodic. Note that the restriction of $D$ to the $E_2$-page of the $C_4$-HFPSS$(\E_2)$ is invertible then it  equals  $\Delta_1$ up to a unit in $\W(\F_4)$, in other words up to a unit we have $\Res^{Q_8}_{C_4}(D)= \Delta_1$.  In \cref{sec:groupcoho} we show that the restriction of the classes $h_2$, $d$ and $g$ are non-trivial. Then in stem $54$ of $Q_{8}$-HFPSS$(\E_2)$, we have the following two restrictions up to units
\begin{align*}
\Res^{Q_8}_{C_4}(D^6 h_2^2)&= \Delta_1^4\done^6 u_{6\lambda}u_{4\sigma}a_{2\sigma} \\ \Res^{Q_8}_{C_4}(g^2d)&= \Delta_1^4\done^8 u_{4\lambda}u_{6\sigma}a_{4\lambda}a_{2\sigma}.
\end{align*}
Note that in $G_{24}$-HFPSS$(\E_2)$, there are no other classes between these two filtrations. Then the naturality forces a hidden $2$ extension from $D^6 h_2^2$ to  $g^2 d$ in  $G_{24}$-HFPSS$(\E_2)$. This hidden $2$ extension also happens in $Q_8$-HFPSS$(\E_2)$ by \cref{cor:split}.
\end{proof}
As the $C_4$-HFPSS for $\E_2$ is $32$-periodic, a similar proof gives the following hidden $2$ extension in stem $22$ in the $Q_8$-HFPSS for $\E_2$. In the rest of this paper, when we refer to restricting a class, we always mean the restriction up to a unit in $\W(\F_4)$.

\begin{cor}\label{cor:hidden2Q8}
There is a hidden $2$ extension from $D^2 h_2^2$ to $D^{-4}g^2 d$.
\end{cor}

\begin{prop}\label{prop:d13two}
The classes $2Dh_2$ at $(11,1)$ and $2D^5 h_2$ at $(43,1)$ support  $d_{13}$-differentials
\begin{enumerate}
\item $d_{13}(2D h_2)=D^{-8}g^3d;$
\item $d_{13}(2D^5 h_2)=D^{-4}g^3 d.$
\end{enumerate}
\end{prop}
\begin{proof}
(1)~ By \cref{lem:hiiden2Q8} and the $E_{\infty}$-page class $g$, there is a hidden $2$ extension from $D^{-2}gh_2^2$ to $ D^{-8}g^3d$ in stem $10$ of the $Q_8$-HFPSS for $\E_2$. By \cref{prop:d5}, we have
\[
d_5(D h_2)=D^{-2}gh_2^2.
\] 
Then the hidden $2$ extension forces $D^{-8}g^3d$ to be hit by a differential of length at most $13$. Note that there is a $2$ extension $2 (2D h_2) = D h_1^3$. Then $D h_1^3$ cannot support a shorter differential than $2D h_2$. In particular, $D h_1^3$ cannot support a $d_{11}$-differential to $D^{-8}g^3d$. This rules out the only possibility that the target $D^{-8}g^3d$ is hit by a shorter differential. Therefore, we proved the desired non-trivial differential.\\
(2)~ 
It follows similarly from the hidden $2$ extension from $D^2g^2h_2^2$ to $D^{-4}g^3d$ by \cref{cor:hidden2Q8}.
\end{proof}

\begin{rem}\rm
In Bauer's computation for $tmf$ \cite{Bau08}, the hidden $2$ extension in \cref{lem:hiiden2Q8} is proved using four-fold Toda brackets. In our approach, the hidden $2$ extension follows from the restriction and the $C_4$-HFPSS hidden $2$ extension, which again is forced by the exotic restrictions and transfers in \cref{lem:hidden2C_4}.
\end{rem}

\begin{lem}\label{lem:pcford7}
The class $Dh_1^3$ is a permanent cycle.
\end{lem}
\begin{proof}
The class $Dh_1^3$ is a $5$-cycle. By \cref{cor:split} and degree reasons, $Dh_1^3$ can only possibly hit $D^{-8}g^3 d$ and $D^{-14}g^6h_1^2$. According to \cref{prop:d13two} the former class is hit by a non-trivial $d_{13}$-differential.  Moreover, according to \cref{cor:d23two} and \cref{met:Tatemethod}, the class $D^{-14}g^6h_1^2$ supports a non-trivial $d_{23}$-differential ($D^{-14}g^6 h_1^2=D^{-16}g^6 D^2h_1^2$). Therefore, these two potential targets cannot receive differentials from $Dh_1^3$. The result thereby follows.
\end{proof}

\begin{prop}\label{prop:d7two}
The class $D^4$ at $(32,0)$ supports a $d_7$-differential
\[
d_7(D^{4})=Dg h_1^{3}.
\]
\end{prop}
\begin{proof}

Note that $g$ and $D^{-8}$ are permanent cycles. Then by \cref{lem:pcford7} the class $D^{-15} h_1^{3} g^6$ at $(3,27)$ is also a permanent cycle. This class has to be hit by a differential via the vanishing line method (\Cref{thm:sharpvanishingline}). By \Cref{cor:split}, the potential source is either $D^{-3} g c$ or $D^{-12}g^5$. The former supports a $d_9$ by \Cref{prop:diff9}.
Therefore, the only possibility is the $d_7$-differential
$$d_7(D^{-12}g^5) = D^{-15} h_1^{3} g^6.$$
Since $D^{-8} g^5$ is a permanent cycle, the result follows.

\end{proof}
All  $d_7$-differentials follow from \cref{prop:d7one}, \cref{prop:d7two} and the Leibniz rule.

Before proving the next two $d_9$-differentials in \cref{prop:d9two}, we need to first prove a permanent cycle in \cref{cor:pcford9} and two $d_{11}$-differentials in \cref{prop:d11}.

\begin{lem}\label{cor:pcford9}
The class $D^3 dh_1$ is a permanent cycle.
\end{lem}
\begin{proof}
By \cref{prop:d9one} in the $Q_8$-TateSS for $\E_2$, we have a $d_9$-differential
\[
d_9(D^9 g^{-2} c)= D^3 d h_1.
\]
Then $D^3 dh_1$ is a permanent cycle in the $Q_8$-TateSS. By \cref{lem:tateisorange} it is also a permanent cycle in $Q_8$-HFPSS$(E_2)$.
\end{proof}

\begin{prop}\label{prop:d11}
The classes $D^{2}d$ at $(30,2)$ and $D^6 d$ at $(62,2)$ support $d_{11}$-differentials
\begin{enumerate}
\item $d_{11}(D^{2}d)=D^{-4}g^3 h_1;$
\item $d_{11}(D^6 d)=g^3h_1.$
\end{enumerate}
\end{prop}
\begin{proof}

According to \cref{RestrictionList}, the restriction of the class $d$ from $Q_8$-HFPSS$(\E_2)$ to $C_4$-HFPSS$(\E_2)$ is non-trivial, Since the class $d$ is  order $2$, the class $\Res(D^2 d)$ must be order $2$. Then according to the computations in \cite{BBHS20} (See \cite[Figure 5.3]{BBHS20}), on bigrading $(30,2)$, the class $\Res(D^2 d)$ supports a non-trivial $d_{13}$-differential. This implies the class $D^2 d$ supports a non-trivial differential with a length at most $13$. The desired differential in (1) follows by degree reasons. The proof for (2) is similar since $C_4$-HFPSS$(\E_2)$ is $32$-periodic. 

\end{proof}
\begin{cor}
The classes $D^{2}dh_1$ at $(31,3)$ and $D^6 dh_1$ at $(63,3)$ support  $d_{11}$-differentials
\begin{enumerate}
\item $d_{11}(D^{2}dh_1)=D^{-4}g^3 h_1^2;$
\item $d_{11}(D^6 dh_1)=g^3 h_1^2$.
\end{enumerate}
\end{cor}

\begin{cor}\label{prop:d9two}
The classes $Dh_1$ at $(9,1)$ and  $D^5 h_1$ at $(41,1)$ support $d_9$-differentials
\begin{enumerate}
\item  $d_9(D h_1)=D^{-5}g^2 c;$
\item $d_9(D^5 h_1)=D^{-1}g^2 c.$
\end{enumerate}
\end{cor}
\begin{proof}

By \cref{cor:split} and degree reasons, the class $Dh_1$ either supports a non-trivial $d_9$-differential or is an $11$-cycle. We show that it is the first case.

If $Dh_1$ were a $11$-cycle then by \cref{prop:d11} and the Leibniz rule, there would be a $d_{11}$-differential
\[
d_{11}(D^3 dh_1)=D^{-3}g^3 h_1^2.
\]
This contradicts \cref{cor:pcford9}. Therefore, we have the desired $d_9$-differential in (1). The proof for $(2)$ is similar. 

\end{proof}

\begin{prop}\label{prop:9cycle}
The class $D^{-1} h_1$ is a $13$-cycle. 
\end{prop}

\begin{proof}
Since $D^8$ is the periodic class, it suffices to prove that $D^7 h_1$ is a $13$-cycle. The $D^{7} h_1$ is a $7$-cycle from our computation of $E_9$-page. According to \cref{prop:d9one}, the class $Dc$ supports a $d_9$-differential. Then the class $g^2 Dc$ supports a non-trivial $d_9$-differential by \cref{met:Tatemethod} since the class $g=kD^3$ is invertible.

Therefore, the class $D^{7}h_1$  does not support a $d_9$-differential since the possible target $g^2 Dc$ already supports a $d_9$-differential. Then for degree reasons, $D^7 h_1$ is a $13$-cycle. So is the class $D^{-1}h_1$. 
\end{proof}
\begin{cor}\label{prop:d9three}
The classes $D^2c$ at $(24,2)$ and $D^6c$ at $(56,2)$ support $d_9$-differentials
 \begin{enumerate}
\item $d_9(D^2c)=D^{-4}g^2 dh_1;$
\item $d_9(D^6 c)=g^2 dh_1.$
\end{enumerate}
\end{cor}
\begin{proof}
Suppose $D^2 c$ does not support a non-trivial $d_9$-differential. Then for degree reasons, it is a $13$-cycle. However, since $D^{-1}h_1$ is also a $13$-cycle, the Leibniz rule show that $Dh_1c$ is also a $13$-cycle. This contradicts \cref{prop:d13one} and proves the $d_9$-differential in $(1)$. The $d_9$-differential in $(2)$ follows similarly by \cref{prop:diff9}.
\end{proof}

\begin{cor}\label{cor:d11}
The classes $Ddh_1$ at $(23,3)$ and $D^5dh_1$ at $(55,3)$ support $d_{11}$-differentials
\begin{enumerate}
    \item $d_{11}(Ddh_1)=D^{-5}g^3 h_1^2;$
    \item $d_{11}(D^5 dh_1)=D^{-1}g^3 h_1^2.$
\end{enumerate}
\end{cor}
\begin{proof}
According to \cref{prop:9cycle}, the class $D^{-1}h_1$ is a $13$-cycle. Then these two $d_{11}$-differentials follow by \cref{prop:d11} and the Leibniz rule.
\end{proof}

\begin{lem}\label{lem:dpc}
The class $d$ is a permanent cycle.
\end{lem}
\begin{proof}
\cref{prop:d13two} shows that $d$ is hit by a $d_{13}$-differential from $2D^9 g^{-3}h_2$ in $Q_8$-TateSS$(\E_2)$.
By \cref{lem:tateisorange} $d$ is a permanent cycle.
\end{proof}
\begin{rem}\rm
The class $d$ is in the image of the Hurewicz map $S^0\rightarrow \E_2^{hQ_8}$. This follows from the Hurewicz image of $\E_2^{hC_4}$ \cite[Figure 12]{HSWX2018} (see \cref{prop:hurewiczC4}).
\end{rem}
\begin{prop}\label{prop:d9five}
The classes $D^2h_1$ at $(17,1)$ and $D^6h_1$ at $(49,1)$ support  $d_9$-differentials
\begin{enumerate}
\item $d_9(D^2h_1)=D^{-4}g^2c;$
\item $d_9(D^6 h_1)=g^2 c.$
\end{enumerate}
\end{prop}
\begin{proof}
We prove this by contradiction.
Assume $D^2h_1$ does not support the desired differential. Then it is a $11$-cycle by degree reasons. The Leibniz rule forces the class $Dh_1$ to support a non-trivial $d_{11}$-differential but this contradicts \cref{lem:dpc}. The proof of (2) is similar.
\end{proof}

\cref{table:HPFSS_integer_diff} lists the differentials we have computed so far. They generate differentials via the Leibniz rule.
By inspection, these are all non-trivial differentials since the remaining classes are permanent cycles by \cref{met:Tatemethod}. 


\subsection{Extension problem}\label{sub:2ext}
Now we solve all the $2$-extensions on the $E_\infty$-page.
\begin{thm}\label{thm:hidden24}
All the hidden $2$ extensions in the integer-graded  $G_{24}$-$\mathrm{HFPSS}(\E_2)$ are displayed in \cref{fig:G24integerEinf1} by grey vertical lines. 
\end{thm}

\begin{proof}
Since the $G_{24}$-HFPSS for $\E_2$ is $192$-periodic, it suffices to consider the stem range from $0$ to $192$.
 We divide these $2$ extensions into three types by their proofs. The first type follows from the fact that in homotopy groups of spheres, $4\nu = \eta^3$ and $h_1$ detects $\eta$, $h_2$ detects $\nu$ (\cref{lemma:gPC}). This type of hidden $2$ extensions happens in stem $3,27,51,99,123$ and $147$ in the period from $0$ to $192$. 

The second type consists of the $2$ extensions in stem $54$ and $150$. The proof of the first is in \cref{lem:hiiden2Q8}, and the proof of the second is similar using the $32$-periodicity of $C_4$-HFPSS$(\E_2)$ and \cref{lem:hidden2C_4}.

The third type consists of three hidden $2$ extensions in the first period. The first one is in stem $110$ from $D^{12}d$ to $D^6g^3 h_1^2$. The other two in stem $130$ and $150$ (from filtration $10$ to $22$) follow from the first one by multiplying $g$ and $g^2$ respectively. So it suffices to show that there is a $2$ extension from $D^{12}g^2d$ to $D^6 g^5 h_1^2$.  To derive this $2$ extension, we claim there are two hidden $h_1$ extensions from $D^{18}h_2$ to $D^{15}gc$ and from $D^{15}gch_1$ to $D^6 g^5 h_1^2$.  As for the first hidden $h_1$ extension, In $G_{24}$-TateSS$(\E_2)$, we have the following two differentials by \cref{prop:d5} and \cref{prop:d9two}:
\begin{align*}
    &d_{5}(g^{-1}D^{21})=D^{18}h_2,\\
    &d_{9}(g^{-1}D^{21}h_1)=D^{15} g c.
    \end{align*}
 Now consider the cofibration 
\[
 {\E_2}_{hG_{24}}\rightarrow \E_2^{hG_{24}}\rightarrow \E_2^{tG_{24}}.
 \]
In the negative filtrations in $G_{24}$-TateSS$(\E_2)$, there is an $h_1$ extension from $g^{-1}D^{21}$ to $g^{-1}D^{21}h_1$, then this  $h_1$ extension under the additive norm map gives an $h_1$ extension relation in $\pi_*\E_2^{hG_{24}}$ from an element detected by $D^{18}h_2$ to some element detected by $D^{15} g c$. This forces a hidden $h_1$ extension from $D^{18}h_2$ to $D^{15} g c$ in $G_{24}$-HFPSS$(\E_2)$. The similar hidden $h_1$ extension from $D^{15}gch_1$ to $D^6 g^5 h_1^2$ follow from the following two differentials in $G_{24}$-TateSS$(\E_2)$ by \cref{prop:d9two} and \cref{cor:d23two}.
\begin{align*}
    &d_{9}(g^{-1}D^{21}h_1^2)=D^{15} g ch_1,\\
    & d_{23}(g^{-1}D^{21}h_1^3)=D^6 g^5 h_1^2
    \end{align*}
Therefore, in $\pi_* \E_2^{hG_{24}}$ there is an $h_1^3$-extension from $D^{18}h_2$ to $D^6 g^5 h_1^2$. On the other hand we know $h_1^3=4h_2$, which implies $D^6 g^5 h_1^2$ must be $4a$ for some class $a\in \pi_{150}\E_2^{hG_{24}}$. Then the degree reasons forces the $2$ extension $D^{12}g^2d$ to $D^6 g^5 h_1^2$.

We claim there are no further $2$ extensions in $G_{24}$-HFPSS$(\E_2)$. By degree reasons, the other possible hidden $2$ extensions either have sources that are $h_1$ divisible or have targets that support $h_1$ extensions. Therefore, the hidden $2$ extensions cannot happen in these cases. 
\end{proof}

\begin{cor}\label{cor:hiddenQ8}
All the hidden $2$ extensions in the integer-graded  $Q_8$-$\mathrm{HFPSS}(\E_2)$ are displayed in \cref{fig:integerEinf} by gray vertical lines. 
\end{cor}
\begin{proof}
This follows from \cref{thm:hidden24} and \cref{prop:G24Q8}.
\end{proof}

Our result of $2$ extensions via the equivariant and the Tate methods matches the $tmf$ computation in \cite{Bau08}. In \cite{Bau08}, because the arguments for proving differentials rely on (hidden) $\eta$ and $\nu$ extensions, almost all these hidden extension are also computed (there is another $\nu$ extension from $D^{15}h_1^2$ at $(122,2)$ and its $\bar\kappa$ multiples \cite[Lemma~5.3]{Isa09}). Here our new methods only use hidden $2$ extensions and the $h_1$, $h_2$ multiplications on the $E_2$-page. Therefore, we do not need to work out hidden $\eta$ and $\nu$ extensions and in our figures, we only draw $h_1$, $h_2$ multiplications.

\subsection{Differentials: alternative methods}\label{subsection:differentialsalternative}

In this subsection, we revisit several differentials in the integer-graded part via different approaches.
\begin{prop}\label{prop:d5res}
The class $D$ at $(8,0)$ supports a $d_5$-differential
\[
d_5(D)=D^{-2}gh_2.
\]
\end{prop}
\begin{proof}
The restriction of $D$ to the $C_4$-HFPSS for $\E_2$ is $\Delta_1$, which supports a non-trivial $d_5$-differential according to \cite[Proposition~5.24]{BBHS20}. By naturality, $D$ must support a non-trivial differential with length $\leq 5$. Then by \cref{cor:split} and degree reasons, it has to be
  $
d_5(D)=D^{-2}gh_2.
$
\end{proof}



Moreover, given all $d_5,d_7$-differentials, then the vanishing line forces the $d_{11}$-differential in \cref{prop:d11}.
\begin{prop}\label{prop:d11vanishing}
The class $D^6 d$ at $(62,2)$ supports a $d_{11}$-differential

$$d_{11}(D^6 d)= g^3h_1.$$
\end{prop}
\begin{proof}
It is enough to prove the $d_{11}$-differential
$$d_{11}(D^6 g^5 d h_1)= g^8 h_1^2$$ since $g$ is invertible in the $Q_8$-TateSS for $\E_2$.  The target $g^8h_1^2$ is a permanent cycle in filtration $34\geqslant 23$. By  \cref{thm:sharpvanishingline} and \cref{thm:tatevanishing} it has to be hit by a differential. Since $D^6 g^5 d h_1$ is a $7$-cycle, the only possibility is the desired $d_{11}$-differential. 
\end{proof}

We here present another proof of the $d_9$-differential in \cref{prop:d9five} which combines the partial calculations in $(*-\sigma_i)$-gradings by the norm method (see \cref{prop:d9norm}).
\begin{prop}\label{prop:d9fivenorm}
The class $D^2h_1$ at $(17,1)$ supports a $d_9$-differential
\[
d_9(D^2h_1)=D^{-4}g^2c.
\]
\end{prop}
\begin{proof}
Suppose the claimed $d_9$-differential does not happen, then $D^2 h_1$ is a $9$-cycle. According to \cref{lem:sigma11cycle}, the class $\{x+y\}D^4 u_{\sigma_i}$ is a $9$-cycle. Then the Leibniz rule implies that $\{x+y\}D^6 h_1 u_{\sigma_i}$ is also a $9$-cycle. This contradicts the fact that $\{x+y\}D^6 h_1 u_{\sigma_i}$ supports a non-trivial $d_9$-differential in \cref{prop:d9norm}. 
\end{proof}

\subsection{Summary of differentials}
We summarize differentials in \cref{table:HPFSS_integer_diff}. All differentials follow from this list by the Leibniz rule. 
\newline
\begin{longtable}{llllll}
\caption{HPFSS differentials, integer page
\label{table:HPFSS_integer_diff} 
} \\
\toprule
$(s,f)$ & $x$ & $r$ & $d_r(x)$ &Proof &\\
\midrule \endhead
\bottomrule \endfoot
$(12,0)$&$v_1^6$&3&$v_1^4h_1^3$& \cref{prop:d3} (restriction)\\
\midrule
 $(8,0)$ &  $D$&  5&  $D^{-2} g h_2$   &\cref{prop:d5} (vanishing line)\\
 &&&& or \cref{prop:d5res} (restriction)\\
 \midrule
         $(8,0)$&   $4D$ &  7 &  $D^{-2} g h_1^3$&\cref{prop:d7one} ($8\nu =\eta^3$)\\ 
         $(16,0)$&  $2D^2$ &  7 &
         $D^{-1} g h_1^3$&\cref{prop:d7one}\\
         $(32,0)$&   $D^4$ &  7 
         &$D g h_1^3$& \cref{prop:d7two} (vanishing line)\\
         \midrule
         $(9,1)$ & $Dh_1$ & 9 & $ D^{-5} g^2 c$& \cref{prop:d9two}    \\
           $(41,1)$ & $D^5h_1$ & 9 & $ D^{-1} g^2 c$&\cref{prop:d9two}    \\
           
         $(16,2)$ & $Dc$ & 9 &$D^{-5}g^2dh_1$& \cref{prop:d9one}\\
                  $(48,2)$ & $D^5 c$ & 9 &$D^{-1}g^2dh_1$& \cref{prop:diff9}\\
         
         $(17,1)$ & $D^2 h_1$ & 9 &$ D^{-4}g^2c$& \cref{prop:d9five}\\
          $(49,1)$ & $D^6 h_1$ & 9 &$ g^2c$&\cref{prop:d9five} \\
         $(24,2)$ & $D^2c$ & 9 &$D^{-4} g^2 dh_1$& \cref{prop:d9three}  \\
         
            $(56,2)$ & $D^6c$ & 9 &$ g^2 dh_1$& \cref{prop:d9three} \\

         \midrule
         $(30,2)$ & $D^{2}d$ & 11 &$D^{-4}g^3 h_1$&\cref{prop:d11} (restriction) \\ 
        $(62,2)$ & $D^{6}d$ & 11 &$g^3 h_1$&\cref{prop:d11} (restriction) \\ 
        &&&& or \cref{prop:d11vanishing} (vanishing line)\\
        $(23,3)$ & $Ddh_1$ &11& $D^{-5}g^3h_1^2$ &\cref{cor:d11} \\
       $(55,3)$ & $D^5 dh_1$ &11& $D^{-1}g^3h_1^2$ & \cref{cor:d11}\\
          \midrule
         $(17,3)$ & $Dch_1$ & 13 &$2D^{-8}g^4$&\cref{prop:d13one} (vanishing line) \\
         $(49,3)$ & $D^5 ch_1$ &13 & $2D^{-4}g^4$& \cref{prop:diff9} (transfer) \\
         $(11,1)$ & $2D h_2$ & 13 &$ D^{-8}g^3 d$&\cref{prop:d13two} (hidden $2$ extension)\\
        $(43,1)$ &$2D^5h_2$ &13& $D^{-4}g^3 d$ &\cref{prop:d13two}\\
         \midrule
          $(-7,1)$ & $D^{-1}  h_1$ & 23 &$ D^{-16}g^6$&
     \cref{prop:d13one} (vanishing line) \\ 
$(18,2)$ &$D^2h_1^2$&23&$D^{-13}g^6h_1$&\cref{cor:d23two}\\
$(43,3)$&$D^5h_1^3$&23&$D^{-10}g^6h_1^2$&\cref{cor:d23two}\\
     
     \midrule
\end{longtable}


\section{\texorpdfstring{The $(*-\sigma_i)$-graded computation} {}}\label{section:sigma}
In this section, we compute the $(*-\sigma_i)$-graded $Q_8$-HFPSS for $\E_2$. We use the following convention: a class at $(n-\sigma_i,m)$ will be denoted as in degree $(n-1,m)$. Since the $Q_8$-representation $\sigma_i$ cannot be lifted to $G_{24}$, in this section, we only consider the groups $Q_8$ and $SD_{16}$. We name classes by their names in the 2-BSS in \cref{table:2BSS_sigma_Einf}, and also use 2-BSS names for the integer-graded classes as it makes the multiplication relation clearer.

\begin{prop}\label{prop:sigmad3one}
The class $v_1^2u_{\sigma_i}$ at $(4,0)$ supports a $d_3$-differential
\[
d_3(v_1^2u_{\sigma_i})=h_1^3 u_{\sigma_i}.
\]
\end{prop}
\begin{proof}
We consider the restriction map from $(*-\sigma_i)$-graded  $Q_8$-HFPSS$(\E_2)$ to the integer-graded $C_4$-HFPSS$(\E_2)$. Note that the $C_4$-invariant element $T_2 \in H^0(C_4, \pi_{4}\E_2)$ equals $v_1^2$ modulo $2$. This implies 
$\Res ^{Q_8}_{C_4}(v_1^2  u_{\sigma_i}) = T_2$ modulo $2$. Recall that in the $C_4$-HFPSS for $\E_2$, the class $T_2$ supports a non-trivial $d_3$-differential (\cite[Proposition~5.21]{BBHS20}). Then the class $v_1^2 u_{\sigma_i}$ must support a non-trivial differential of length $\leq 3$. For degree reasons, we have 
\[
d_3(v_1^2  u_{\sigma_i})=h_1^3 u_{\sigma_i}.
\]
\end{proof}
Since the $(*-\sigma_i)$-graded part is a module over the integer-graded part, this $d_3$-differential implies a family of $d_3$-differentials as follows:
\[
d_3(k^sD^m v_1^{4l+2}h_1^n u_{\sigma_i})=k^sD^m v_1^{4l}h_1^{n+3}u_{\sigma_i}
\]
where $k,m,l,n\in \mathbb{Z}$ and $l,n \geq 0$. By taking out these $d_3$-differentials, an argument similar to the proof in \cref{prop:inftybo} shows that the following classes are permanent cycles
\[
2D^m v_1^{4l-2}, D^m v_1^{4l}, D^m v_1^{4l}h_1, D^m v_1^{4l}h_1^2
\]
where $l\geq 1$ and $m\geq 0$. 
All the classes above either support non-trivial $d_3$-differentials or are permanent cycles. Similar to the $bo$-pattern in the integer graded part, we do not need to consider this part in later computations of higher differentials. 

However, this is not the only kind of $d_3$-differentials in $(*-\sigma_i)$-graded part. To derive the second kind of $d_3$-differentials, we first need to show the $d_5$-differential pattern and several other facts.

\begin{lem}\label{lem: asigmaiP.C.}
The class $\{x+y\}u_{\sigma_i}$ is a permanent cycle.
\end{lem}
\begin{proof}
For degree reasons, this class is $a_{\sigma_i}$ on the $E_2$-page defined in \cref{defn:aclass}. By \cref{prop:aVHFPSS}, this class is a permanent cycle.
\end{proof}

\begin{cor}\label{cor: sigmad5one}
The class $\{x+y\}Du_{\sigma_i}$ at $(7,1)$ supports a $d_5$-differential
\[
d_5(\{x+y\}Du_{\sigma_i})=k\{yh_2+xh_1 v_1\}Du_{\sigma_i}.
\]
\end{cor}
\begin{proof}
Since the $(*-\sigma_i)$-graded part is a module over the integer-graded part, the claimed differential follows from \cref{lem: asigmaiP.C.},  \cref{prop:d5} and the Leibniz rule.
\end{proof}

\cref{cor: sigmad5one} generates
the first kind of $d_5$-differentials via the Leibniz rule.

\begin{lem}\label{lem:x2y2P.C.}
The class $\{x^2+y^2\}Du_{\sigma_i}$ is a permanent cycle.
\end{lem}
\begin{proof}
According to \cref{prop:d9five}, there is a $d_9$-differential 
$
d_9(D^6h_1)=k^2D^7 xh_1.
$
Then the Leibniz rule implies that the class $k^2x^2h_1D^7 u_{\sigma_i}=k^2D^7xh_1 \cdot \{x+y\}u_{\sigma_i}$ is hit by a differential of length $\leq 9$. For degree reasons, it is hit by either a $d_9$-differential or a $d_7$-differential.  In either case, the $h_1$ extensions force $k^2\{x^2+y^2\}D^7 u_{\sigma_i}$ to be hit on or before the $E_9$-page. Then the class  $\{x^2+y^2\}Du_{\sigma_i}$  must be a permanent cycle; otherwise 
the class $k^2\{x^2+y^2\}D^7 u_{\sigma_i}$ would support a non-trivial differential since 
\[
k^2\{x^2+y^2\}D^7 u_{\sigma_i}=\{x^2+y^2\}Du_{\sigma_i} \cdot k^2D^6
\]
where  $k^2D^6=g^2$ is a permanent cycle that survives to $E_{\infty}$-page in the integer-graded part. 
\end{proof}
\begin{cor}\label{cor:sigmad5two}
The class $\{x^2+y^2\}u_{\sigma_i}$ at $(-2,2)$ supports a $d_5$-differential
\[
d_5(\{x^2+y^2\}u_{\sigma_i})=k\{x+y\}h_1^2 u_{\sigma_i}.
\]
\end{cor}
 By inspection, all $d_5$-differentials in $(*-\sigma_i)$-graded part follows from \cref{cor: sigmad5one} and \cref{cor:sigmad5two} by the Leibniz rule. 
\begin{lem}\label{lem:sigma11cycle}
The class $\{x+y\}D^4 u_{\sigma_i}$ is a $11$-cycle.
\end{lem}
\begin{proof}
According to \cite[Remark~5.23]{BBHS20}, the class  $u_{2\sigma}$ is a $5$-cycle in the $C_4$-HFPSS for $\E_2$. Therefore, \cref{thm:normdiff} implies that  $N_{C_4}^{Q_8}(u_{2\sigma})a_{\sigma_i}$ is a $9$-cycle. According to \cref{prop:auclassnorm}
\[
N_{C_4}^{Q_8}(u_{2\sigma})a_{\sigma_i}=\frac{u_{2\sigma_j}{u_{2\sigma_k}}}{u_{2\sigma_i}}a_{\sigma_i}.
\]
If we restrict this class to the subgroup $C_4\langle j\rangle $, then we get the class $a_{\sigma}$ which is non-trivial. Hence $N_{C_4}^{Q_8}(u_{2\sigma})a_{\sigma_i}$ is non-trivial on the $E_2$-page. By multiplying $N_{C_4}^{Q_8}(u_{2\sigma}) a_{\sigma_i}$ with the following periodicity classes 
$$ N_{C_4 \langle i\rangle}^{Q_8}(u_{4\sigma}) N_{C_4 \langle i\rangle}^{Q_8}(\done)^4 N_{C_4 \langle i\rangle}^{Q_8}(u_{4\lambda}u_{2\sigma})u_{4\sigma_i}^5u_{4\sigma_j}^{-1}u_{4\sigma_k}^{-1}$$
where $u_{4\sigma_i},u_{4\sigma_j},u_{4\sigma_k}$ are permanent cycles by \cref{rem:4sigmapc}, we get a non-trivial class at $(31,1)$. For degree reasons, this class must be $\{x+y\}D^4 u_{\sigma_i}$ (up to a unit). This implies that $\{x+y\}D^4 u_{\sigma_i}$ is also a $9$-cycle. For degree reasons, $\{x+y\}D^4 u_{\sigma_i}$ is a $11$-cycle .
\end{proof}
\begin{rem}\rm
We will show in \cref{prop:sigmad13} that the above class supports a non-trivial $d_{13}$-differential.
\end{rem}


\begin{prop}\label{prop:sigmad3two}
The class $\{h_1+xv_1\}u_{\sigma_i}$ at $(1,1)$ supports a $d_3$-differential
\[
d_3(\{h_1+xv_1\}u_{\sigma_i})=2kv_1^2 u_{\sigma_i}.
\]
\end{prop}
\begin{proof}
 We argue by contradiction. Suppose this differential does not happen.  Then the class $2kv_1^2 u_{\sigma_i}$ will survive to the $E_5$-page and the Leibniz rule implies that there is a $d_5$-differential 
 \[
d_5(x^3D^4\usi)=2k^2v_1^2D^4\usi,
 \]
since there is an $h_2$ extension from $kx^3D^4\usi$ to $2k^2v_1^2D^4\usi$ by \cref{lem:hiddenh1}. 

On the other hand, according to \cref{lem: asigmaiP.C.}, \cref{prop:d7two} and the Leibniz rule we know the class  $\{x+y\}D^4$ is a $5$-cycle. Moreover, the class $x^2$ is a also $5$-cycle by \cref{prop:d11}. so the product  $x^3D^4=\{x+y\}D^4\usi\cdot x^2$ is a $5$-cycle. This is a contradiction. And the claimed $d_3$ follows immediately.

\end{proof}
\begin{rem}\rm \cref{prop:sigmad3two} shows that $2kv_1^2 u_{\sigma_i}$ is hit by a $d_3$-differential. Recall that  the class $kv_1^2 u_{\sigma_i}$ itself supports a non-trivial $d_3$-differential by \cref{prop:sigmad3one}.
\end{rem}
By the above discussion and by inspection, all $d_3$-differentials in the  $(*-\sigma_i)$-graded part  follows from \cref{prop:sigmad3one}, \cref{prop:sigmad3two} and the Leibniz rule.

\begin{cor}\label{cor:sigmad11one}
The classes $x^3 u_{\sigma_i}$ at $(-3,3)$ and $x^3D^4 u_{\sigma_i}$ at $(29,3)$ support $d_{11}$-differentials
\begin{enumerate}
\item $d_{11}(x^3 u_{\sigma_i})=k^3\{x+y\}Dh_1 u_{\sigma_i};$

\item $d_{11}(x^3D^4 u_{\sigma_i})=k^3\{x+y\}D^5 h_1 u_{\sigma_i}.$
\end{enumerate}
\end{cor}
\begin{proof}
According to \cref{prop:d11}, there is a $d_{11}$-differential in the integer-gradings
\[
d_{11}(x^2)=k^3D h_1.
\]
Note that $\{x+y\}u_{\sigma_i}$ and $\{x+y\}D^4 u_{\sigma_i}$ are both $11$-cycles. By the Leibniz rule, we have
\[
d_{11}(x^3 u_{\sigma_i})=\{x+y\}u_{\sigma_i}d_{11}(x^2)=k^3\{x+y\}Dh_1 u_{\sigma_i}. 
\]
The proof of the second $d_{11}$-differential is similar. 
\end{proof}

\begin{prop}\label{prop:sigmad9two}
The classes $\{h_1^2+xh_1v_1\}Du_{\sigma_i}$ at $(10,2)$ and  $\{h_1^2+xh_1v_1\}D^5 u_{\sigma_i}$ at $(42,2)$ support  $d_9$-differentials
\begin{enumerate}
\item $d_9(\{h_1^2+xh_1v_1\}Du_{\sigma_i})=k^2 \{x+y\}h_1^2D^2u_{\sigma_i};$
\item $d_9(\{h_1^2+xh_1v_1\}D^5 u_{\sigma_i})=k^2 \{x+y\} h_1^2 D^6u_{\sigma_i}.$
\end{enumerate}
\end{prop}
\begin{proof}
Because $kD^3=g$ is an invertible permanent cycle in  $Q_8$-TateSS$(\E_2)$, the $d_{11}$-differential in \cref{cor:sigmad11one} 
$$d_{11}(x^3D^4 u_{\sigma_i}) = k^3 \{x+y\}  h_1 D^5 u_{\sigma_i}$$
implies that in the $(*-\sigma_i)$-graded   $Q_8$-TateSS$(\E_2)$ we have
\[
d_{11}(k^{-1}x^3 Du_{\sigma_i})=(kD^3)^{-1}d_{11}(x^3D^4u_{\sigma_i})=k^2\{x+y\}h_1D^2 u_{\sigma_i}.
\]
Since $k^2\{x+y\}h_1D^2u_{\sigma_i}$ is hit by a $d_{11}$-differential in $Q_8$-TateSS$(\E_2)$, its $h_1$ extension,  
$k^2\{x+y\}h_1^2D^2u_{\sigma_i}$, has to be hit on or before the $E_{11}$-page. For degree reasons, this class $k^2\{x+y\}D^2h_1^2u_{\sigma_i}$ must be hit by the claimed $d_9$-differential in $Q_8$-TateSS$(\E_2)$.
By \cref{lem:tateisorange}, the first claimed $d_9$-differential also happens in $Q_8$-HFPSS$(\E_2)$. The second $d_9$-differential follows similarly. 
\end{proof}
We have the following $d_9$-differentials by the Leibniz rule and integer-graded $d_9$-differentials.

\begin{prop}\label{prop:sigmad9three}
We have the following $d_9$-differentials
\begin{enumerate}
    \item $d_9(\{x+y\}h_1D u_{\sigma_i})=k^2x^2 h_1 D^2u_{\sigma_i};$
    \item $d_9(\{x+y\}h_1D^2 u_{\sigma_i})=k^2x^2 h_1D^3u_{\sigma_i};$
    \item$d_9(\{x+y\} h_1 D^5 u_{\sigma_i})=k^2x^2 h_1 D^6 u_{\sigma_i};$
    
    \item $d_9(\{x+y\}h_1D^6 u_{\sigma_i})=k^2x^2 h_1D^7 u_{\sigma_i}.$
\end{enumerate}
\end{prop}
\begin{proof}
We prove the first differential, and the proofs of the rest three differentials are similar. According to \cref{prop:d9two}, in the integer-graded  $Q_8$-HFPSS$(\E_2)$ we have
\[
d_9(Dh_1)=k^2xh_1D^2.
\]
Note that the class $\{x+y\}u_{\sigma_i}$ is a permanent cycle by \cref{lem: asigmaiP.C.}. Then the Leibniz rule implies
\[
d_9(\{x+y\}h_1Du_{\sigma_i})=\{x+y\}u_{\sigma_i}d_9(Dh_1)=k^2x^2 h_1 D^2u_{\sigma_i}.
\]
\end{proof}

\begin{cor}\label{cor:sigmad9three}
The classes $\{x+y\}D^2 u_{\sigma_i}$ at $(15,1)$ and $\{x+y\}D^6 u_{\sigma_i}$ at $(47,1)$ support $d_9$-differentials
\begin{enumerate}
    \item $d_9(\{x+y\}D^2 u_{\sigma_i})=k^2\{x^2+y^2\} D^3u_{\sigma_i};$
    \item $d_9(\{x+y\}D^6 u_{\sigma_i})=k^2\{x^2+y^2\} D^7u_{\sigma_i}$.
\end{enumerate}
\end{cor}
\begin{proof}
By \cref{prop:9cycle}, the class $D^{-1}h_1$ is a $9$-cycle. These two $d_9$-differentials hold since otherwise the classes $\{x+y\}Dh_1 u_{\sigma_i}$ and $\{x+y\}D^5 h_1 u_{\sigma_i}$ would be $9$-cycles by the Leibniz rule, which contradicts \cref{prop:sigmad9three}.
\end{proof}

To derive the last type of $d_9$-differential, we first need to show the following $d_{17}$-differential in the $(*-\sigma_i)$-graded part.

\begin{prop}\label{prop:sigmad17}
The class $\{h_1^2+xh_1v_1\}u_{\sigma_i}$ at $(2,2)$ supports a $d_{17}$-differential
\[
d_{17}(\{h_1^2+xh_1v_1\}u_{\sigma_i})=k^4\{ x+y\} h_1^2 D^2 u_{\sigma_i}.
\]
\end{prop}
\begin{proof}
Consider the class $k^6\{h_1^2+xh_1v_1\} D^{10}u_{\sigma_i}$ in filtration $26$. By \cref{thm:sharpvanishingline} this class cannot survive to the $E_\infty$-page.

 After the $E_5$-page, all the potential sources that could support a differential hitting the class $k^6\{h_1^2+xh_1v_1\} D^{10}u_{\sigma_i}$ are $k^3x^2 h_1 D^9 u_{\sigma_i}$, $k^3\{x+y\}D^9 u_{\sigma_i}$ and $kx^2h_1D^8 u_{\sigma_i}$. We rule out all three possibilities one by one. The class $k^3x^2 h_1 D^9 u_{\sigma_i}$ is hit by the following $d_9$-differential in \cref{prop:sigmad9three}
\[
d_9(kxh_1 D^9 u_{\sigma_i})=kD^3 d_9(\{x+y\}h_1 D^6 u_{\sigma_i})=k^3x^2 h_1 D^9 u_{\sigma_i}.
\]
 The class $k^3\{x+y\}D^9 u_{\sigma_i}$ is a permanent cycle since $\{x+y\} u_{\sigma_i}$ is a permanent cycle by \cref{lem: asigmaiP.C.}. The class $kx^2h_1D^8$ is also a permanent cycle since it is hit by a known $d_9$-differential in the  $Q_8$-TateSS for $\E_2$ according to \cref{prop:sigmad9three}
\[
d_9(k^{-1}\{x+y\}D^7 h_1 u_{\sigma_i})=kx^2h_1D^8 u_{\sigma_i}.
\]
Therefore, the class $k^6\{h_1^2+xh_1v_1\} D^{10}u_{\sigma_i}$ must support a non-trivial differential. Since $kD^3=g$ is an  invertible permanent cycle in TateSS, the class $\{h_1^2+xh_1v_1\}u_{\sigma_i}=D^8 (kD^3)^{-6}k^6\{h_1^2+xh_1v_1\} D^{10}u_{\sigma_i}$ also has to support a non-trivial differential. 

Therefore, the class  $k^6\{h_1^2+xh_1v_1\} D^{10}u_{\sigma_i}$ has to support a non-trivial differential, so is the class $\{h_1^2+xh_1v_1\}u_{\sigma_i}$. For degree reasons, there are three possible targets which are $kx^3D$,
 $k^4xh_1^2 D^2 u_{\sigma_i}$ and $k^5x^3D^3 u_{\sigma_i}$. The class $\{h_1^2+xh_1v_1\}D^{-1}u_{\sigma_i}$ is a $5$-cycle for degree reasons, and the Leibniz rule implies
 \begin{equation*}
 \begin{aligned}
 d_5(\{h_1^2+xh_1v_1\}u_{\sigma_i})&=\{h_1^2+xh_1v_1\}D^{-1}u_{\sigma_i}d_5(D)\\
  &=\{h_1^2+xh_1v_1\}D^{-1}u_{\sigma_i}\cdot kDh_2=0
 \end{aligned}
 \end{equation*}
 So the class  $\{h_1^2+xh_1v_1\}u_{\sigma_i}$ is also a $5$-cycle, in other words, the class $kx^3D$ cannot receive a differential from the class $kx^3D$. On the other hand, the class $k^5x^3D^3 u_{\sigma_i}$ supports the  following $d_{11}$-differential by \cref{cor:sigmad11one}
\[
d_{11}(k^5x^3D^3 u_{\sigma_i})= (kD^3)^5 D^{-16} d_{11}(x^3 D^4 u_{\sigma_i})=k^8 D^4 \{x+y\}h_1 u_{\sigma_i}.
\]
Therefore, the class $\{h_1^2+xh_1v_1\}u_{\sigma_i}$ supports the desired $d_{17}$-differentials 
\[
d_{17}(\{h_1^2+xh_1v_1\}u_{\sigma_i})=k^4 \{x+y\} h_1^2 D^2u_{\sigma_i}.
\]
\end{proof}

It turns out that this is the only $d_{17}$-differential in one period of the $(*-\sigma_i)$-graded part of $Q_8$-HFPSS$(E_2)$.


\begin{prop}\label{cor:sigmad9one}
The classes $\{x^2+y^2\}D^3 u_{\sigma_i}$ at $(22,2)$ and $\{x^2+y^2\}D^7 u_{\sigma_i}$ at $(54,2)$ support $d_9$-differentials
\begin{enumerate}
    \item $d_9(\{x^2+y^2\}D^3 u_{\sigma_i})=k^2 x^3 D^4 u_{\sigma_i};$
    \item $d_9(\{x^2+y^2\}D^7 u_{\sigma_i})=k^2 x^3 D^8 u_{\sigma_i}.$
\end{enumerate}
\end{prop}

\begin{proof}
According to \cref{RestrictionList}, the restriction of $\{x^2+y^2\}u_{\sigma_i}$ to the integer-graded   $C_4$-HFPSS for $\E_2$ is non-trivial. It implies the following  restriction by degree reasons
\begin{align*}
    \Res^{Q_8}_{C_4}(\{x^2+y^2\}D^3 u_{\sigma_i})&=\done^6 u_{6\lambda}u_{4\sigma}a_{2\sigma}.
\end{align*}
We now prove that the class $\{x^2+y^2\}D^3 u_{\sigma_i}$ supports a non-trivial differential by contradiction. Suppose that the class $\{x^2+y^2\}D^3 u_{\sigma_i}$  is a permanent cycle that survives to the $E_{\infty}$-page. Note that its $C_4$-restriction $\done^6 u_{6\lambda}u_{4\sigma}a_{2\sigma}$ has a hidden $2$ extension in $\pi_*\E_2^{hC_4}$ by \cref{lem:hidden2C_4}. Then $\{x^2+y^2\}D^3 u_{\sigma_i}$ also has a hidden $2$ extension in the  $E_{\infty}$-page. However, since hidden extensions and natural maps between spectral sequences will not decrease filtration, the potential target of the hidden $2$ extension from the class $\{x^2+y^2\}D^3 u_{\sigma_i}$ can only be $k^2\{x^2+y^2\}D^4 u_{\sigma_i}$, $k\{yh_2+xh_1v_1\}D^3 u_{\sigma_i}$ and $k\{h_1^2+xh_1v_1\}D^3$   by degree reasons. However,  the first class $k^2\{x^2+y^2\}D^4 u_{\sigma_i}$ supports a non-trivial $d_5$-differential by \cref{cor:sigmad5two}
\[
d_5(k^2\{x^2+y^2\}D^4 u_{\sigma_i})=k^3 xh_1^2D^4 u_{\sigma_i}.
\]
The second class $k\{yh_2+xh_1v_1\}D^3 u_{\sigma_i}$ is hit by a $d_5$-differential by \cref{cor: sigmad5one} 
\[
d_5(\{x+y\}D^3u_{\sigma_i})=k\{yh_2+xh_1v_1\}D^3 u_{\sigma_i}.
\]
The third class $k\{h_1^2+xh_1v_1\}D^3 u_{\sigma_i}$ supports a $d_{17}$-differential by \cref{prop:sigmad17}
\[
d_{17}(k\{h_1^2+xh_1v_1\}D^3 u_{\sigma_i})= k^5xh_1^2 D^5u_{\sigma_i}.
\]
Therefore, all the potential targets of the hidden $2$ extension from the class $\{x^2+y^2\}D^3 u_{\sigma_i}$ will not survive to the $E_{\infty}$-page. This is a contradiction. Hence the class $\{x^2+y^2\}D^3 u_{\sigma_i}$ must support a non-trivial differential.

After the $E_5$-page,  the only two potential targets are $k^2 x^3 D^4 u_{\sigma_i}$ and  $k^5\{x+y\}h_1^2 D^5u_{\sigma_i}$  by degree reasons. However, the class $k^5\{x+y\}h_1^2 D^5u_{\sigma_i}$ is hit by the following $d_{17}$-differential by \cref{prop:sigmad17} and the Leibniz rule
\[
d_{17}(k\{h_1^2+xh_1v_1\}D^3u_{\sigma_i})=kD^3d_{17}(\{h_1^2+xh_1v_1\}u_{\sigma_i}) =k^5\{x+y\}h_1^2 D^5u_{\sigma_i}.
\]
Then the first desired $d_9$-differential follows. The proof of the second $d_9$-differential in the statement is similar since the $C_4$-HFPSS for $\E_2$ is $32$-periodic.
\end{proof}

We can apply the norm method to get a $d_9$-differential directly (after the calculation of $E_3$-page) which is independent of the $d_9$ information in the integer-graded part.
\begin{prop}\label{prop:d9norm}

There is a normed  $d_9$-differential in $(*-\sigma_i)$-page
\[
d_9(\{x+y\}D^6 u_{\sigma_i})=k^2\{x^2+y^2\} D^7u_{\sigma_i}.
\]
\end{prop}
\begin{proof}
According to \cite[Theorem~11.13]{HHR17},  the class $u_{2\lambda}$ supports a non-trivial $d_5$-differential in $C_4$-HFPSS$(\E_2)$
\[
d_5(u_{2\lambda})=\done u_{\lambda} a_{2\lambda} a_{\sigma}.
\]
Then \cref{thm:normdiff} implies there is a predicted $d_9$-differential in $Q_8$-HFPSS$(E_2)$
\[
d_9(N_{C_4}^{Q_8}(u_{2\lambda})a_{\sigma_i})=N_{C_4}^{Q_8}(\done) N_{C_4}^{Q_8}(u_{\lambda})a_{2\mathbb{H}} a_{\sigma_j}a_{\sigma_k}.
\]
We claim the target of this predicted $d_9$-differential is non-trivial on the $E_2$-page. It suffices to show that the class $a_{\sigma_j}a_{\sigma_k}$ is non-trivial since $ N_{C_4}^{Q_8}(u_{\lambda})a_{2\mathbb{H}}$ is invertible in TateSS$(\E_2)$. We observe that 
\[
\Res^{Q_8}_{C_4}(a_{\sigma_j}a_{\sigma_k})=a_{2\sigma}
\]
where $a_{2\sigma}$ is non-trivial in $C_4$-HFPSS$(\E_2)$. This implies that $a_{\sigma_j}a_{\sigma_k}$ is also non-trivial. Therefore, the non-trivial class on the $E_2$-page $N_{C_4}^{Q_8}(\done) N_{C_4}^{Q_8}(u_{\lambda})a_{2\mathbb{H}} a_{\sigma_j}a_{\sigma_j}$ must be hit on or before the $E_9$-page. By multiplying this class with the following periodicity classes 
$$N_{C_4 \langle i\rangle}^{Q_8}(u_{4\lambda}u_{2\sigma})N_{C_4 \langle i\rangle}^{Q_8}(\done)^6 u_{4\sigma_i}^5 u_{4\sigma_j} u_{4\sigma_k}$$
we get a non-trivial class at $(46,10)$, which has to be the class $k^2\{x^2+y^2\} D^7u_{\sigma_i}$ (up to a unit) by degree reasons. Therefore, the class $k^2\{x^2+y^2\} D^7u_{\sigma_i}$ has to be hit on or before the $E_9$-page too. If this class is hit by a $d_7$-differential from the class $x^2h_1D^6\usi$, then the class $k^2x^2h_1D^7\usi$ has to be killed on or before the $E_7$-page. However, this is a contradiction by degree reasons. Therefore, the claimed $d_9$-differential follows.
\end{proof}


All $d_9$-differentials follow from \cref{prop:sigmad9two}, \cref{prop:sigmad9three}, \cref{cor:sigmad9three}, \cref{cor:sigmad9one} and the Leibniz rule.

\begin{prop}\label{prop:sigmad23}
The  classes $\{x+y\}h_1^2 D^2 u_{\sigma_i}$ at $(17,3)$, $\{x+y\}h_1  D^7 u_{\sigma_i}$ at $(56,2)$ support $d_{23}$-differentials
\begin{enumerate}
\item $d_{23}(\{x+y\}h_1^2 D^2 u_{\sigma_i})= k^6 \{x+y\} h_1D^5 u_{\sigma_i};$
\item $d_{23} (\{x+y\}h_1  D^7 u_{\sigma_i}) = k^6 \{x+y\}D^{10} u_{\sigma_i}.$
\end{enumerate}
\end{prop}
\begin{proof}
By \cref{cor:d23two} we have the following two $d_{23}$-differentials in the integer-graded part
\begin{align*}
d_{23}(D^2 h_1^2)=k^6 h_1 D^5, \text{~and~}
d_{23}(D^7 h_1)=k^6 D^{10}.
\end{align*}
Note that the class $\{x+y\}u_{\sigma_i}$ is a permanent cycle by \cref{lem: asigmaiP.C.}. Then the desired two differentials follow from these  $d_{23}$-differentials and the Leibniz rule.
\end{proof}

All $d_{23}$-differentials follow from \cref{prop:sigmad23} and the Leibniz rule.

\begin{lem}\label{lem:sigmapctwo}
The four classes $\{h_1^2+xh_1v_1\} D^2 u_{\sigma_i},\{h_1^2+xh_1v_1\}D^3 u_{\sigma_i}, \{h_1^2+xh_1v_1\} D^6 u_{\sigma_i}$ and $\{h_1^2+xh_1v_1\} D^7 u_{\sigma_i}$ are all permanent cycles.
\end{lem}
\begin{proof}
After the $E_5$-page, the potential targets of $\{h_1^2+xh_1v_1\} D^7 u_{\sigma_i}$ are the classes $k^2\{x+y\}h_1^2D^8u_{\sigma_i}$ and $k^3x^3D^9u_{\sigma_i}$, since lengths of differentials in the $RO(Q_8)$-graded $Q_8$-HFPSS$(\E_2)$ are less than or equal to $23$ by \cref{thm:sharpvanishingline}. However, the class $k^2\{x+y\}h_1^2D^8u_{\sigma_i}$ supports a non-trivial $d_{23}$-differential by \cref{prop:sigmad23} and the class $k^3x^3D^9u_{\sigma_i}$ supports a non-trivial $d_{11}$-differential by \cref{cor:sigmad11one}. For similar reasons, the rest three classes are permanent cycles. 
\end{proof}

\begin{prop}\label{prop:sigmad11two}
There are four non-trivial $d_{11}$-differentials
\begin{enumerate}
  \item $d_{11}(x^2 h_1 D^2 u_{\sigma_i})=k^3\{h_1^2+xh_1v_1\}D^3 u_{\sigma_i};$
    \item $d_{11}(x^2h_1D^3 u_{\sigma_i})=k^3\{h_1^2+xh_1v_1\} D^4 u_{\sigma_i};$
    \item $d_{11}(x^2 h_1 D^6 u_{\sigma_i})=k^3\{h_1^2+xh_1v_1\}D^7 u_{\sigma_i};$
  
\item $d_{11}(x^2 h_1 D^7 u_{\sigma_i})=k^3\{h_1^2+xh_1v_1\}D^8 u_{\sigma_i}$.
\end{enumerate}
\end{prop}

\begin{proof}

We first prove $(2)$. Consider the class $k^6\{h_1^2+xh_1v_1\}D^5 u_{\sigma_i}$ which is a permanent cycle by \cref{lem:sigmapctwo}. Since its filtration is greater than $23$, the horizontal vanishing line forces it to be killed. For degree reasons, after $E_5$-page, there are two potential sources: $k^3x^2 h_1 D^4u_{\sigma_i}$ and $kx^2h_1 D^3u_{\sigma_i}$. However, the class $kx^2h_1 D^3u_{\sigma_i}$ is a permanent cycle since it is killed by a $d_9$-differential in the associated TateSS according to \cref{prop:sigmad9three}. Therefore, we have
\[
d_{11}(k^3x^2h_1 D^4u_{\sigma_i})=k^6\{h_1^2+xh_1v_1\}D^5u_{\sigma_i}.
\]

Next, consider the class $k^6\{h_1^2+xh_1v_1\}D^8 u_{\sigma_i}$ which is also a permanent cycle by \cref{lem:sigmapctwo}. Similarly, the horizontal vanishing line forces it to be killed eventually. After $E_5$-page, for degree reasons, there are three potential sources: $k^3 x^2h_1D^7 u_{\sigma_i}, k^3\{x+y\}D^7 u_{\sigma_i}$ and $kx^2h_1D^6u_{\sigma_i}$. The second class $k^3\{x+y\}D^7 u_{\sigma_i}$ supports a $d_9$-differential by \cref{prop:d9norm}. The third class $kx^2h_1D^6u_{\sigma_i}$ supports a $d_{11}$-differential we just proved
\[
d_{11}(kx^2h_1D^6u_{\sigma_i})=kD^3d_{11}(x^2h_1D^3 u_{\sigma_i})=k^4\{h_1^2+xh_1v_1\} D^7 u_{\sigma_i}
\]
Therefore, we have
\[
d_{11}(k^3 x^2h_1D^7u_{\sigma_i})=k^3\{h_1^2+xh_1v_1\}D^8 u_{\sigma_i}.
\]

The rest two claimed $d_{11}$-differentials follow by similar arguments. 

\end{proof}

All $d_{11}$-differentials follow from \cref{cor:sigmad11one}, \cref{prop:sigmad11two} and the Leibniz rule.

\begin{prop}\label{prop:sigmad13}
The class $\{x+y\}D^4 u_{\sigma_i}$ at $(31,1)$ supports a $d_{13}$-differential
\[
d_{13}(\{x+y\}D^4 u_{\sigma_i})=k^3\{h_1^2+xh_1v_1\}D^5 u_{\sigma_i}.
\]
\end{prop}
\begin{proof}
We first claim the class $k^3\{h_1^2+xh_1v_1\}D^5 u_{\sigma_i}$ is a permanent cycle. In the $Q_8$-TateSS for $\E_2$, by multiplying it with $k^{-3}D^{-9}\cdot D^8$, we obtain $\{h_1^2+xh_1v_1\}D^4 u_{\sigma_i}$, which is a permanent cycle by \cref{lem:sigmapctwo}. So $k^3\{h_1^2+xh_1v_1\}D^5 u_{\sigma_i}$ is also a permanent cycle in the $Q_8$-HFPSS for $E_2$. 

Next we consider the class $k^6\{h_1^2+xh_1v_1\}D^6 u_{\sigma_i}=k^3\{h_1^2+xh_1v_1\}D^5 u_{\sigma_i}\cdot k^3 D^9\cdot D^{-8}$  above the vanishing line. By \cref{thm:sharpvanishingline} it must be hit by a differential since it is a permanent cycle. Then for degree reasons, the only two possible sources are $\{x+y\}D^4 u_{\sigma_i}$ and $x^2h_1 D^4u_{\sigma_i}$. Note that the class $x^2h_1 D^4$ is a permanent cycle since it is hit by a $d_9$-differential in $Q_8$-TateSS$(\E_2)$
\[
d_9(k^{-2}xh_1 D^3 u_{\sigma_i})=x^2h_1 D^4 u_{\sigma_i}.
\]
Therefore, the claimed $d_{13}$-differential must happen.
\end{proof}
This $d_{13}$-differential can also be deduced via the norm method.
\begin{proof}[Second proof of \cref{prop:sigmad13}]
According to \cite[Theorem~11.13] {HHR17}\cite[Corollary~3.14]{HSWX2018}, there is a $d_7$-differential in the $C_4$-HFPSS for $\E_2$
\[
d_7(u_{4\lambda})=\done \eta' u_{2\lambda}a_{3\lambda}.
\]
Then \cref{thm:normdiff} shows that there is a predicted $d_{13}$-differential
\[
d_{13}(N_{C_4}^{Q_8}(u_{4\lambda})a_{\sigma_i})=N_{C_4}^{Q_8}(\done)N_{C_4}^{Q_8}(\eta')N_{C_4}^{Q_8}(u_{2\lambda})a_{3\mathbb{H}}.
\]
According to \cite[Proposition~10.4 ~(viii)]{Schwede},
$
\Res^{Q_8}_{C_4}N_{C_4}^{Q_8}(\eta')=\eta'^2
$
is non-trivial. Then $N_{C_4}^{Q_8}(\eta')$ is non-trivial on the $E_2$-page and so is the class $N_{C_4}^{Q_8}(\done)N_{C_4}^{Q_8}(\eta')N_{C_4}^{Q_8}(u_{2\lambda})a_{3\mathbb{H}}$. By multiplying the non-trivial class $N_{C_4}^{Q_8}(\done)N_{C_4}^{Q_8}(\eta')N_{C_4}^{Q_8}(u_{2\lambda})a_{3\mathbb{H}}$ with the periodicity classes in \cref{cor:periodicity}, we get a non-trivial class at $(30,14)$ on the $E_2$-page, which has to be the class $k^3\{h_1^2+xh_1v_1\}D^5 u_{\sigma_i}$ by degree reasons. Therefore,  the class $k^3\{h_1^2+xh_1v_1\}D^5 u_{\sigma_i}$ must be hit on or before  the $E_{13}$-page. For degree reasons, the desired $d_{13}$-differential follows.
\end{proof}

All $d_{13}$-differentials follow from \cref{prop:sigmad13} and the Leibniz rule.

\cref{table:HPFSS_sigmai_diff} lists the differentials we have computed so far. They generate differentials via the
Leibniz rule. By inspection, these are all non-trivial differentials since the remaining classes are permanent cycles by \cref{met:Tatemethod}.

The result is presented in \cref{fig:sigmaEinf}. 

\subsection{Summary of differentials}
Differentials in $(*-\sigma)$-graded part are given by \cref{table:HPFSS_sigmai_diff}. All differentials follow from this list by multiplying permanent cycles and the Leibniz rule.

\begin{longtable}{llllll}
\caption{HPFSS differentials, $(*-\sigma_i)$-page
\label{table:HPFSS_sigmai_diff} 
} \\
\toprule
$(s,f)$ & $x$ & $r$ & $d_r(x)$ &Proof &\\
\midrule 
\endfirsthead
\multicolumn{5}{c}
{{Table \thetable{}. -- HPFSS differentials, $(*-\sigma_i)$-page (continued)}} \\
\toprule
$(s,f)$ & $x$ & $r$ & $d_r(x)$ &Proof &\\
\midrule 
\endhead

\hline
\multicolumn{5}{r}{{Continued on next page}} \\
\hline
\endfoot

\bottomrule \endlastfoot

 $(1,1)$& $\{h_1+xv_1\}u_{\sigma_i}$& 3& $2kv_1^2u_{\sigma_i}$ & \cref{prop:sigmad3two}\\
 $(4,0)$& $v_1^2u_{\sigma_i}$ &3& $h_1^3u_{\sigma_i}$&\cref{prop:sigmad3one} (restriction)\\
 \midrule
         $(7,1)$&   $\{x+y\}Du_{\sigma_i}$ &  5 &  $k\{yh_2+xh_1v_1\}Du_{\sigma_i}$&\cref{cor: sigmad5one} (module structure)\\
         $(14,2)$&   $\{x^2+y^2\}D^2 u_{\sigma_i}$ &  5&
         $kxh_1^2D^2u_{\sigma_i}$&\cref{cor:sigmad5two} (module structure)\\
     
         \midrule
          $(10,2)$ & $\{h_1^2+xh_1v_1\}Du_{\sigma_i}$ & 9 &$k^2\{x+y\}h_1^2D^2u_{\sigma_i}$& \cref{prop:sigmad9two}\\
         
          $(42,2)$ & $\{h_1^2+xh_1v_1\}D^5u_{\sigma_i}$ & 9 &$k^2\{x+y\}h_1^2D^6u_{\sigma_i}$& \cref{prop:sigmad9two}\\
         $(8,2)$ & $\{x+y\}h_1Du_{\sigma_i}$ & 9 & $k^2x^2h_1D^2u_{\sigma_i}$& \cref{prop:sigmad9three} (module structure) \\
         
                  $(40,2)$ & $\{x+y\}h_1D^5u_{\sigma_i}$ & 9 & $k^2x^2h_1D^6u_{\sigma_i}$& \cref{prop:sigmad9three}\\
         
            $(15,1)$ & $\{x+y\}D^2u_{\sigma_i}$ & 9 &$ k^2\{x^2+y^2\}D^3u_{\sigma_i}$&\cref{cor:sigmad9three} \\
         $(47,1)$ & $\{x+y\}D^6u_{\sigma_i}$ & 9 &$ k^2\{x^2+y^2\}D^7u_{\sigma_i}$& \cref{cor:sigmad9three}\\
        $(22,2)$&$\{x^2+y^2\}D^3u_{\sigma_i}$ &9& $k^2x^3D^4u_{\sigma_i}$&\cref{cor:sigmad9one} (hidden $2$ extension)\\
         $(54,2)$&$\{x^2+y^2\}D^7u_{\sigma_i}$ &9& $k^2x^3D^8u_{\sigma_i}$&\cref{cor:sigmad9one}   \\
         \midrule
         $(15,3)$ & $x^2h_1D^2 u_{\sigma_i}$ & 11 &$k^3\{h_1^2+xh_1v_1\}D^3u_{\sigma_i}$& \cref{prop:sigmad11two} (vanishing line) \\ 
            $(47,3)$ & $x^2h_1D^6 u_{\sigma_i}$ & 11 &$k^3\{h_1^2+xh_1v_1\}D^7u_{\sigma_i}$& \cref{prop:sigmad11two}\\ 
         $(23,3)$&$x^2h_1D^3u_{\sigma_i}$& 11& $k^3\{h_1^2+xh_1v_1\}D^4u_{\sigma_i}$ & \cref{prop:sigmad11two}\\
   $(55,3)$&$x^2h_1D^7u_{\sigma_i}$& 11& $k^3\{h_1^2+xh_1v_1\}D^8u_{\sigma_i}$ & \cref{prop:sigmad11two}\\
         $(29,3)$& $x^3D^4 u_{\sigma_i}$&11&$k^3xh_1D^5 u_{\sigma_i}$&\cref{cor:sigmad11one} (module structure)\\
          $(61,3)$& $x^3D^8 u_{\sigma_i}$&11&$k^3xh_1D^9 u_{\sigma_i}$&\cref{cor:sigmad11one}\\
    \midrule
        $(31,1)$ &$\{x+y\}D^4u_{\sigma_i}$ &13& $k^3\{h_1^2+xh_1v_1\}D^5u_{\sigma_i}$ &\cref{prop:sigmad13} (vanishing line\\
         &&&& or norm differential)\\
         \midrule
         $(2,2)$ & $\{h_1^2+xh_1v_1\}u_{\sigma_i}$ & 17 &$k^4\{x+y\}h_1^2D^2u_{\sigma_i}$& \cref{prop:sigmad17} (vanishing line)\\
      
        \midrule
          $(17,3)$ & $\{x+y\}h_1^2D^2 u_{\sigma_i}$ & 23 &$ k^6\{ x+y\}h_1 D^5$&\cref{prop:sigmad23} (module structure)
  \\ 
    $(56,2)$ &$\{x+y\}h_1D^7 u_{\sigma_i}$&23&$k^6\{x+y\}D^{10}u_{\sigma_i}$ &\cref{prop:sigmad23}\\

     \midrule
\end{longtable}

\newpage
\section{Charts and Tables}
\label{section:charts}
\subsection{Keys for the charts}

In all charts, a gray line denotes a multiplication. See the following table for the keys.

\begin{longtable}{ll}
\caption{keys for multiplications} \\
\toprule
line & meanings  \\
\midrule \endhead
\bottomrule \endfoot
vertical & $2$ multiplication\\
slope $1$ & $h_1$ multiplication \\
slope $1/3$ & $h_2$ multiplication \\
dashed (only in $2$BSS) & hidden extension \\
\midrule
\end{longtable}

The colored lines denote the differentials. We use different colors to distinguish different lengths.

\begin{longtable}{ll}
\caption{keys for classes} \\
\toprule
class & meaning  \\
\midrule \endhead
\bottomrule \endfoot
dot & $k$\\
blue dot & $k[\![j]\!]$\\
red dot & $k[\![j]\!]\{j\}$\\
square & $\mathbb W(k)$\\
\end{longtable}
Here $k$ is $\mathbb F_2$ for $G= SD_{16} \text{ or }G_{48}$, and is $\mathbb F_4$ for $G= Q_{8} \text{ or }G_{24}$; $j$ is $v_1^{12} D^{-3}$ for $G_{24}$ or $G_{48}$, and $v_1^{4} D^{-1}$ otherwise.

\begin{rem}
We elaborate more on boxes and dots connected by vertical lines in the same bidegree. Such a pattern denotes a $2$-adic presentation of a class. Namely, the bottom dot is generated by the generator and represents a $2$-torsion copy, the dot or box just above is generated by twice the generator, and so on.

For example, on the $E_\infty$-page of the integral degrees (\cref{fig:integerEinf}), in bigrading $(32,0)$ the bottom red dot represents the class $\mathbb W/2[\![v_1^{4} D^{-1}]\!]\{v_1^4D^3\}$ and the blue box above represents $\mathbb W[\![v_1^{4} D^{-1}]\!]\{2 D^4\}$; Note that there is a $2$ extension. Thus the
class at $(32,0)$ is $\mathbb W[\![v_1^{4} D^{-1}]\!]\{v_1^4D^3\}\oplus \mathbb W\{2D^4\}$.

Such presentations help to demonstrate where the differentials or extensions come from. For example, in \cref{fig:integerE2} in bigrading $(12,0)$, only the generator $v_1^6$ supports a non-trivial $d_3$-differential and $2v_1^6$ survives. This convention is due to Dan Isaksen. 

\end{rem}

\begin{rem}
We comment on the extensions between dots of different colors.
For example, in the bidegree $(24,0)$ and $(25,1)$ in \cref{fig:integerEinf}, there is an $h_1$ multiplication connecting a red and a blue dot. The red dot represents the class $\mathbb W/2[\![v_1^{4} D^{-1}]\!]\{v_1^4 D^2\}$ and the blue dot represents the class $\mathbb W/2[\![v_1^{4} D^{-1}]\!]\{h_1D^3\}$. The $h_1$ multiplication happens whenever it is indicated by the class names. Note that the class $\mathbb W/2\{h_1D^3\}$ is not $h_1$-divisible in this case since the source is missing.
\end{rem}

\subsubsection{2-BSS}\hfill

\cref{fig:2BSSE1} -- \cref{fig:sigma2BSS} are charts for the $2$-Bockstein spectral sequences. All three charts have $(8,0)$ periodicity by multiplying $D$ and $(-4,4)$ periodicity by multiplying $k$ (except the $v_1$ local classes in low filtration). We only depict part of the spectral sequence here, which contains a full periodic range.

In \cref{fig:2BSSE1}, a blue line indicates the multiplication by $x$, while an orange line indicates the multiplication by $y$.

Recall the $(*-\sigma_i)$-graded part and the integer-graded part have isomorphic $E_1$-pages. When interpret the chart as the $(*-\sigma_i)$-graded part, the name of a class at $(s,f)$ is its label multiplied by $u_{\sigma_i}$, and its degree is $(s+1-\sigma_i,f)$. For example, the class $1$ at $(0,0)$, when interpreted as an $(*-\sigma_i)$-graded part class, denotes $u_{\sigma_i}$ at $(1-\sigma_i,0)$ in the 2BSS.

\cref{fig:integer2BSS} and \cref{fig:sigma2BSS} show the $E_\infty$-page of 2BSS, for the integer-graded part and $(*-\sigma_i)$-graded part respectively.

\subsubsection{HFPSS}\hfill

\cref{fig:integerE2}--\cref{fig:integerEinf} depict the integer degree calculation of the integer-graded  $G$-HFPSS$(\E_2)$ for $G=Q_8$ or $SD_{16}$, and \cref{fig:sigmaE2}--\cref{fig:sigmaEinf} depict the $(*-\sigma_i)$-graded calculation. Both $E_2$-pages are $(8,0)$ periodic by multiplying $D$, and other pages are $(64,0)$ periodic by multiplying $D^8$. All charts are $(20,4)$ periodic by multiplying $kD^3$ (except the $v_1$ local classes in low filtration). 
The differentials are denoted by the colored lines with their length classified by the color. When the target or the source of the differential is out of range, we replace the line with an arrow. There are horizontal vanishing lines in filtration $23$ on $E_\infty$-pages.

\newpage
\begin{figure}[H]
\begin{center}
\makebox[0.95\textwidth][c]{
\includegraphics[trim={0cm, 0.8cm, 0cm, 0.8cm},clip,page=1,scale=1.4]{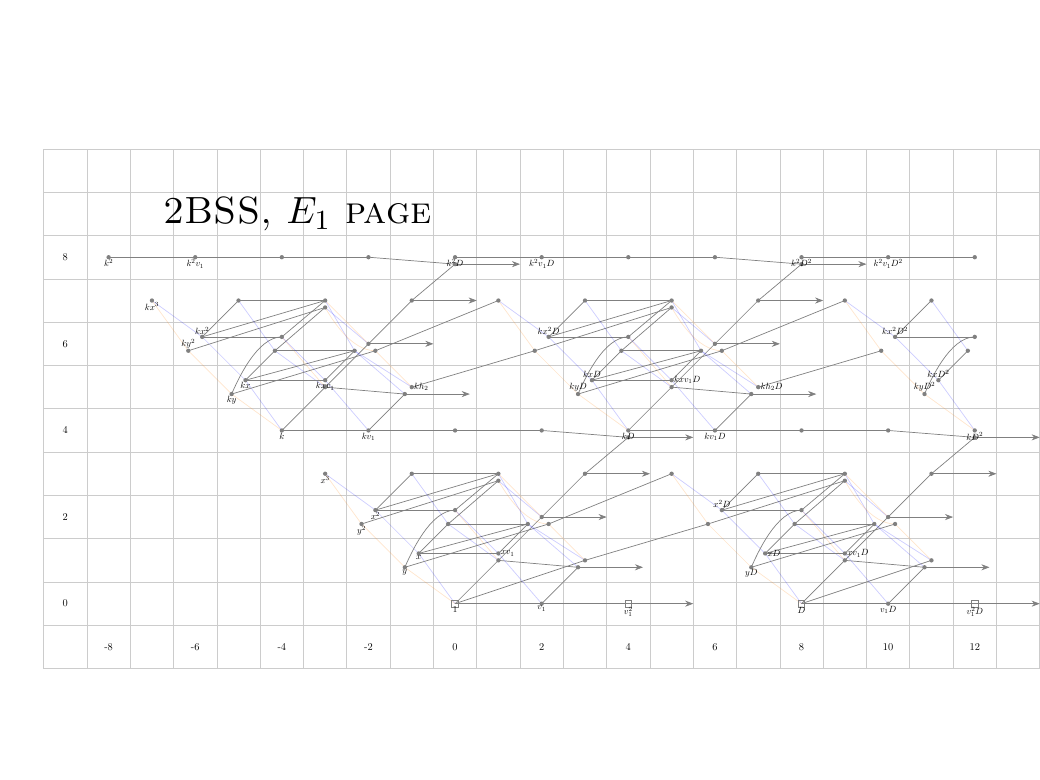}}
\caption{The $E_1$-page of the integer/$(*-\sigma_i)$-graded 2BSS.}
\label{fig:2BSSE1}
\end{center}
\end{figure}

\begin{figure}[H]
\begin{center}
\makebox[0.95\textwidth]{\includegraphics[trim={0cm, 0.8cm, 0cm, 0.8cm},clip,page=1,scale=1.4]{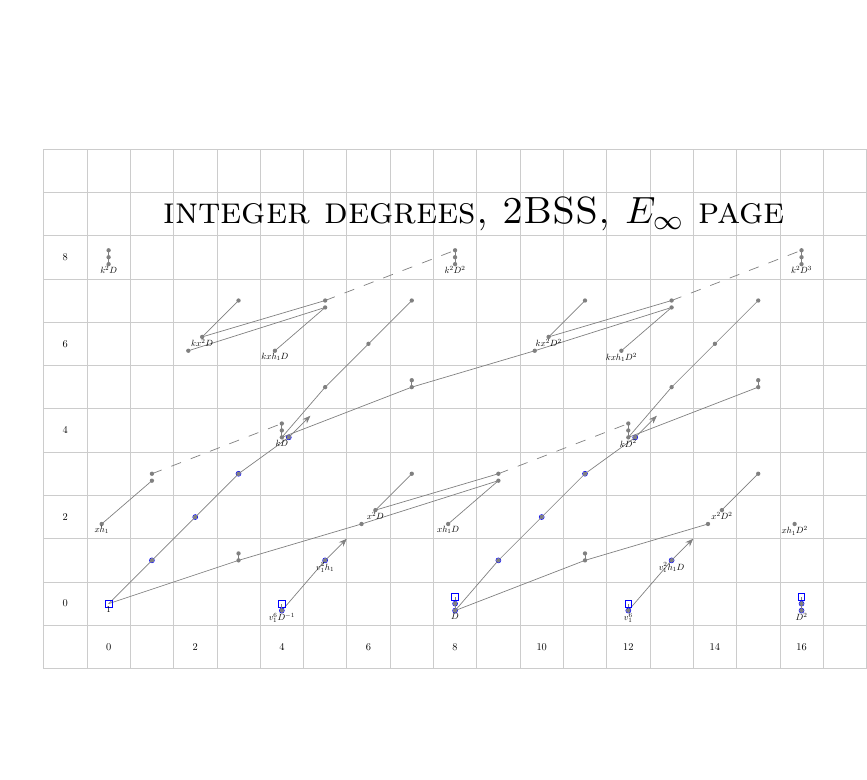}}
\caption{The $E_\infty$-page of the integer-graded 2BSS.The dotted lines are hidden  $h_2$ extensions.}
\label{fig:integer2BSS}
\end{center}
\end{figure}

\begin{figure}[H]
\begin{center}
\makebox[0.95\textwidth]{\includegraphics[trim={0cm, 0.8cm, 0cm, 0.8cm},clip,page=1,scale=1.4]{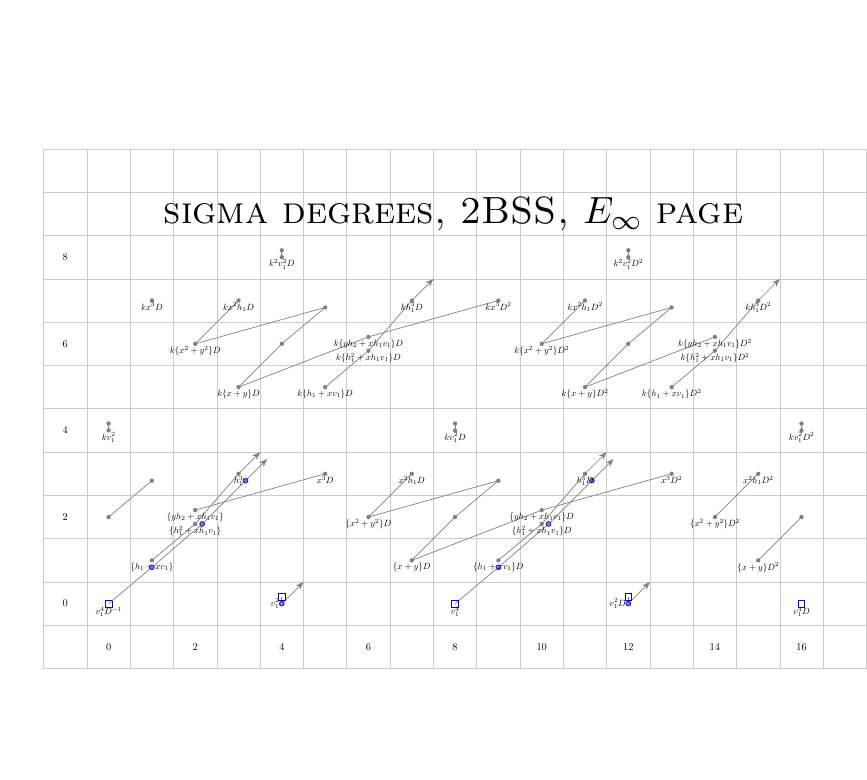}}
\caption{The $E_\infty$-page of the $(*-\sigma_i)$-graded 2BSS. The dotted lines are hidden $h_1$ and $h_2$ extensions.}
\label{fig:sigma2BSS}
\end{center}
\end{figure}


\begin{figure}[H]
\begin{center}
\makebox[0.95\textwidth]{\includegraphics[trim={0cm, 0.8cm, 0cm, 0.8cm},clip,page=1,scale=1.4]{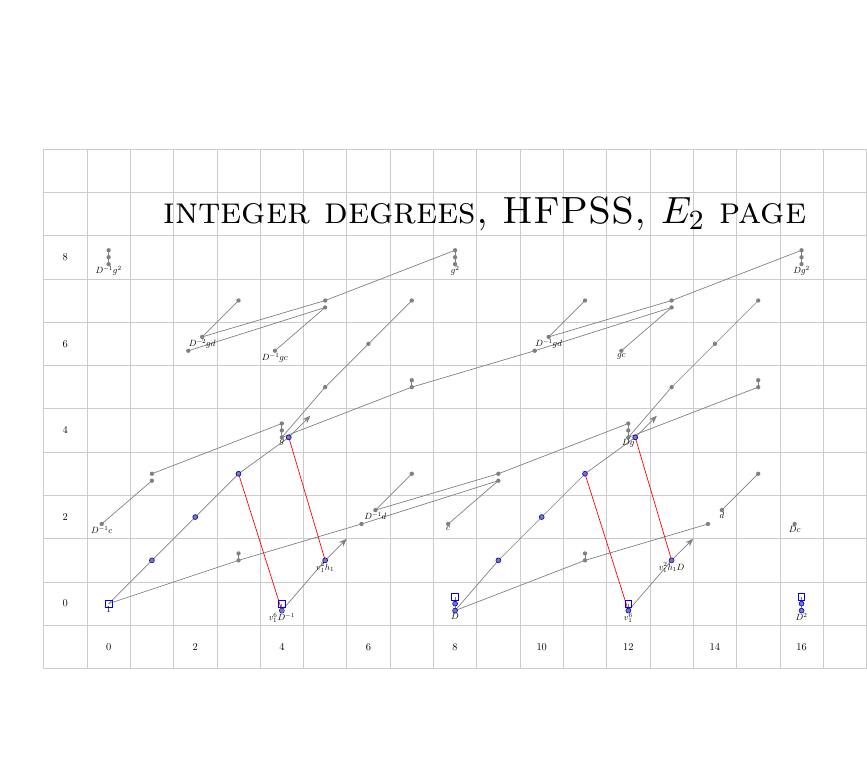}}
\caption{The $E_3$-page of the integer-graded $Q_8$-HFPSS($\E_2$). The {\color{red} red lines} are $d_3$-differentials.}
\label{fig:integerE2}
\hfill
\end{center}
\end{figure}

\begin{figure}[H]
\begin{center}
\makebox[0.95\textwidth]{\includegraphics[angle=90,trim={0cm, 0cm, 0cm, 0cm},clip,page=1,scale=0.7]{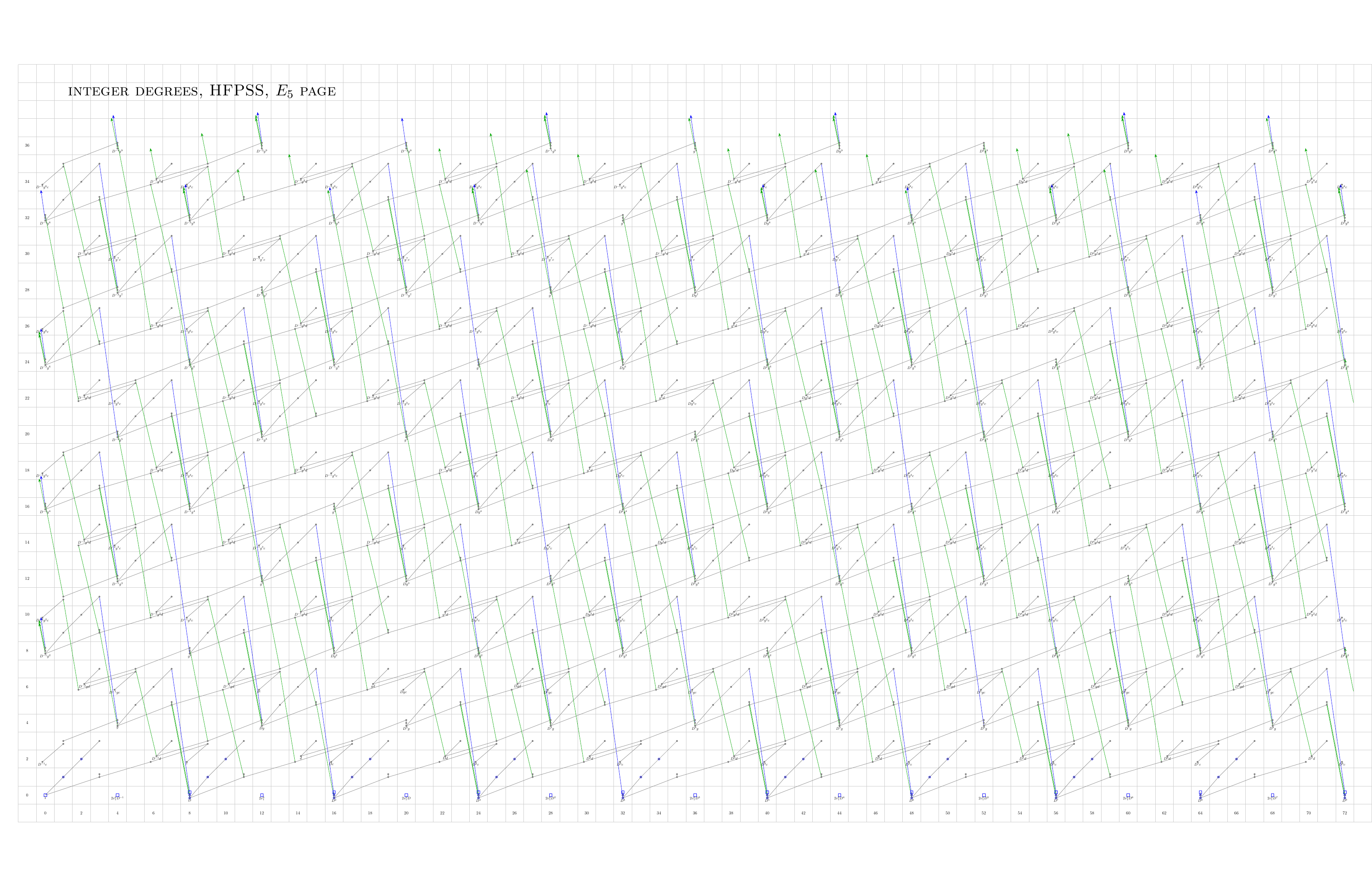}}
\caption{The $E_5$-page of the integer-graded $Q_8$-HFPSS($\E_2$). The {\color{darkgreen}green lines} are $d_5$-differentials. The {\color{blue} blue lines} are $d_7$-differentials.}
\label{fig:integerE5}
\hfill
\end{center}
\end{figure}
\newpage

\begin{figure}[H]
\begin{center}
\makebox[0.95\textwidth]{\includegraphics[angle=90,trim={0cm, 0cm, 0cm, 0cm},clip,page=1,scale=0.7]{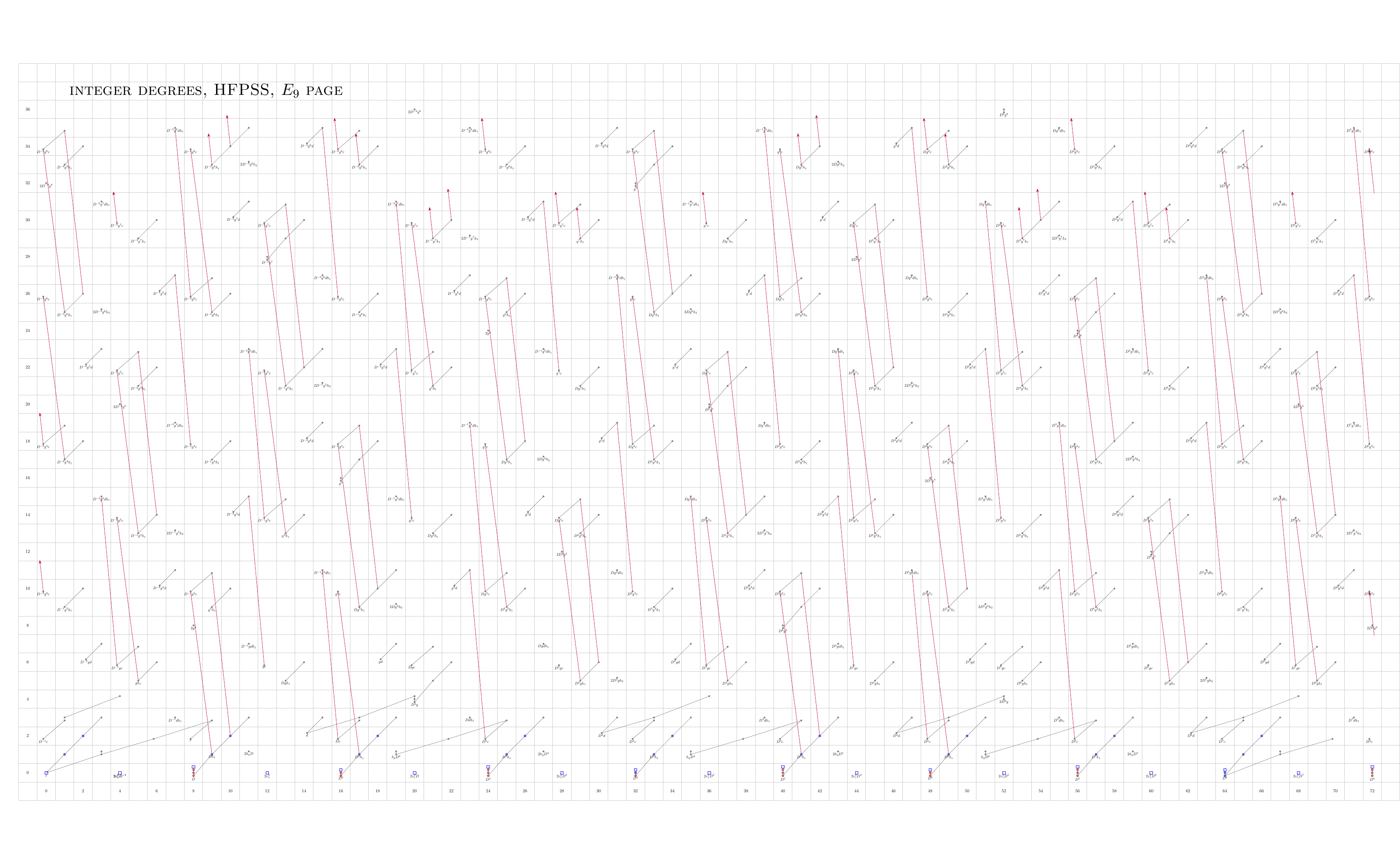}}
\caption{The $E_9$-page of the integer-graded $Q_8$-HFPSS($\E_2$). The {\color{purple}purple lines} are $d_9$-differentials.}
\label{fig:integerE9}
\hfill
\end{center}
\end{figure}

\begin{figure}[H]
\begin{center}
\makebox[0.95\textwidth]{\includegraphics[angle=90,trim={0cm, 0cm, 0cm, 0cm},clip,page=1,scale=0.7]{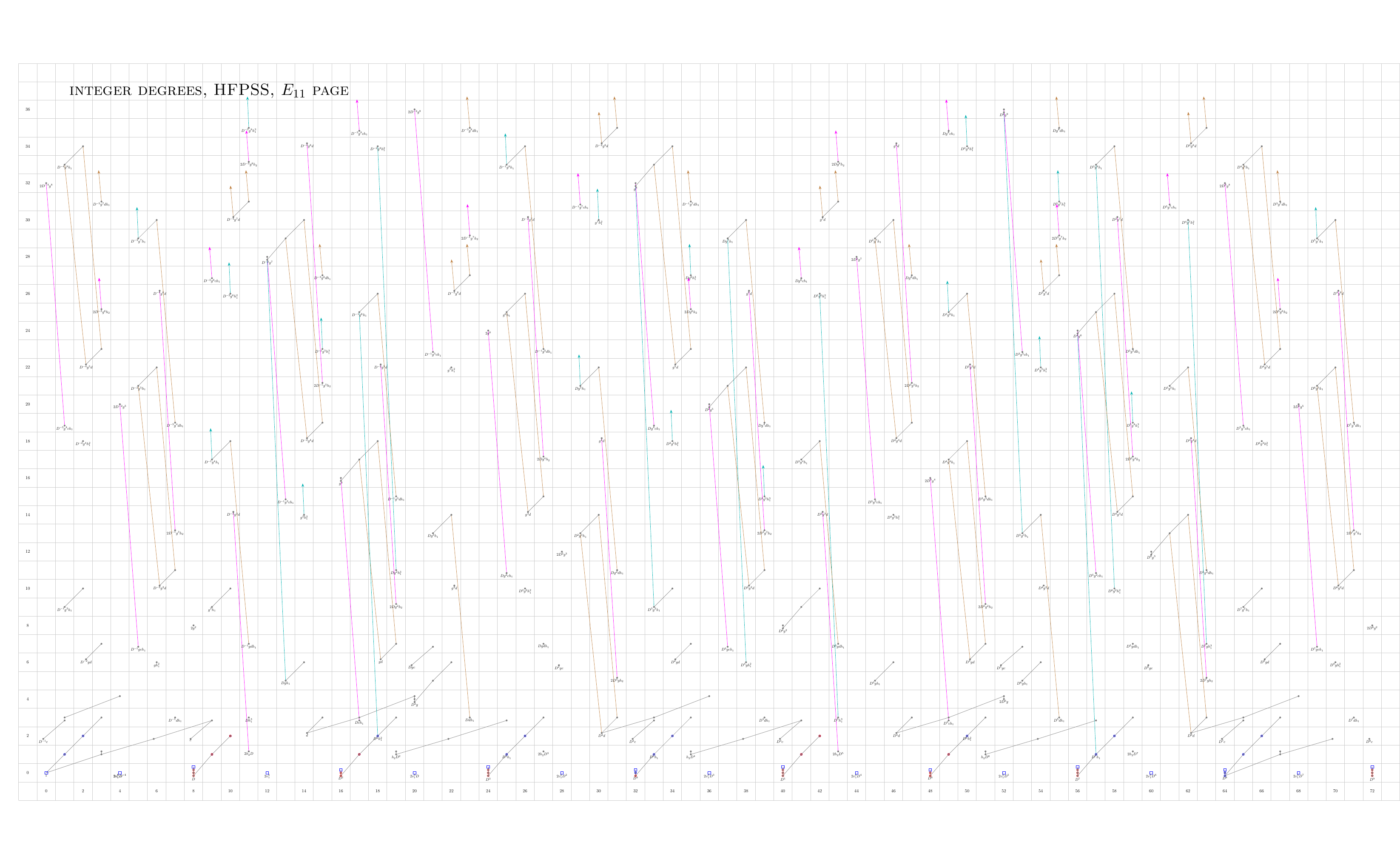}}
\caption{The $E_{11}$-page of the integer-graded $Q_8$-HFPSS($\E_2$). The {\color{brown} brown lines} are $d_{11}$-differentials. The {\color{magenta} magenta lines} are $d_{13}$-differentials. The {\color{darkgreen}green lines} are $d_{13}$-differentials.}
\label{fig:integerE11}
\hfill
\end{center}
\end{figure}

\begin{figure}[H]
\begin{center}
\makebox[0.95\textwidth]{\includegraphics[angle=90,trim={0cm, 0cm, 0cm, 0cm},clip,page=1,scale=0.7]{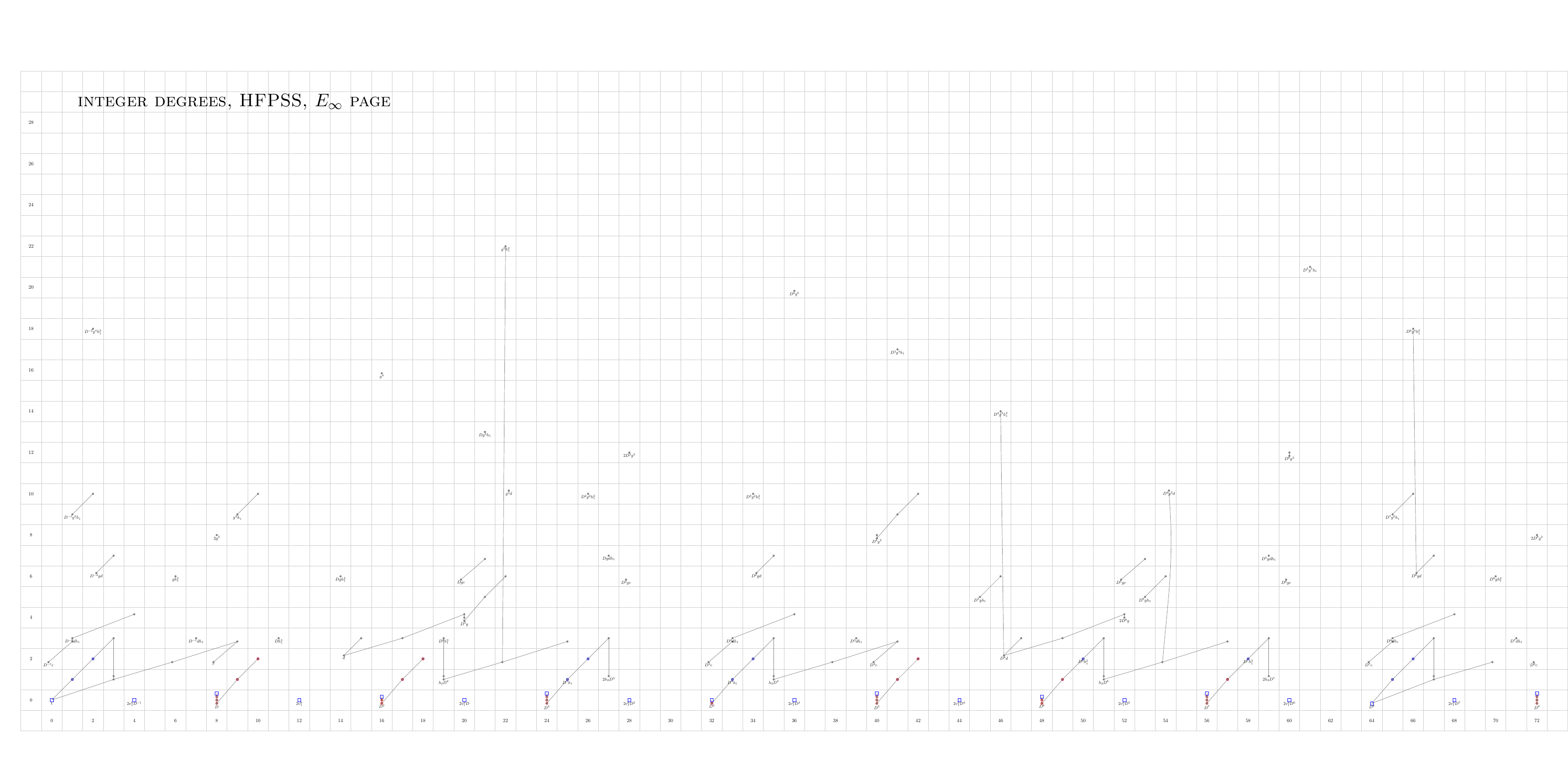}}
\caption{The $E_\infty$-page of the integer-graded $Q_8$/$SD_{16}$-HFPSS$(\E_2)$.}
\label{fig:integerEinf}
\hfill
\end{center}
\end{figure}

\newpage

\begin{figure}[H]
\begin{center}
\makebox[0.95\textwidth]{\includegraphics[angle=90, trim={0cm, 0cm, 58cm, 0cm},clip,page=1,scale=0.78]{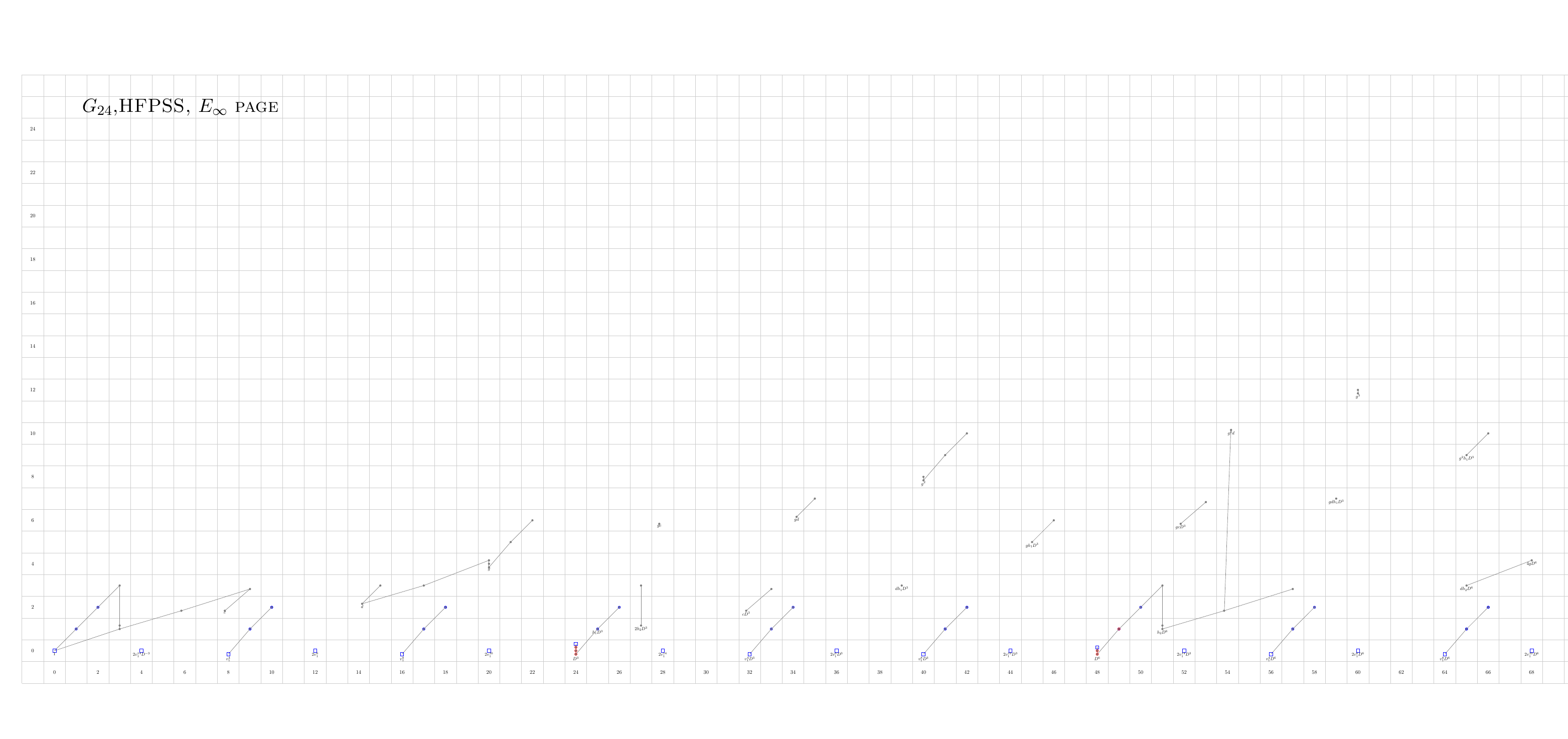}}
\caption{The $E_\infty$-page of the integer-graded $G_{24}$/$G_{48}$-HFPSS$(\E_2)$ (stem 0 -- 68).}
\label{fig:G24integerEinf1}
\hfill
\end{center}
\end{figure}

\begin{figure}[H]
\begin{center}
\makebox[0.95\textwidth]{\includegraphics[angle=90, trim={29cm, 0cm, 29cm, 0cm},clip,page=1,scale=0.78]{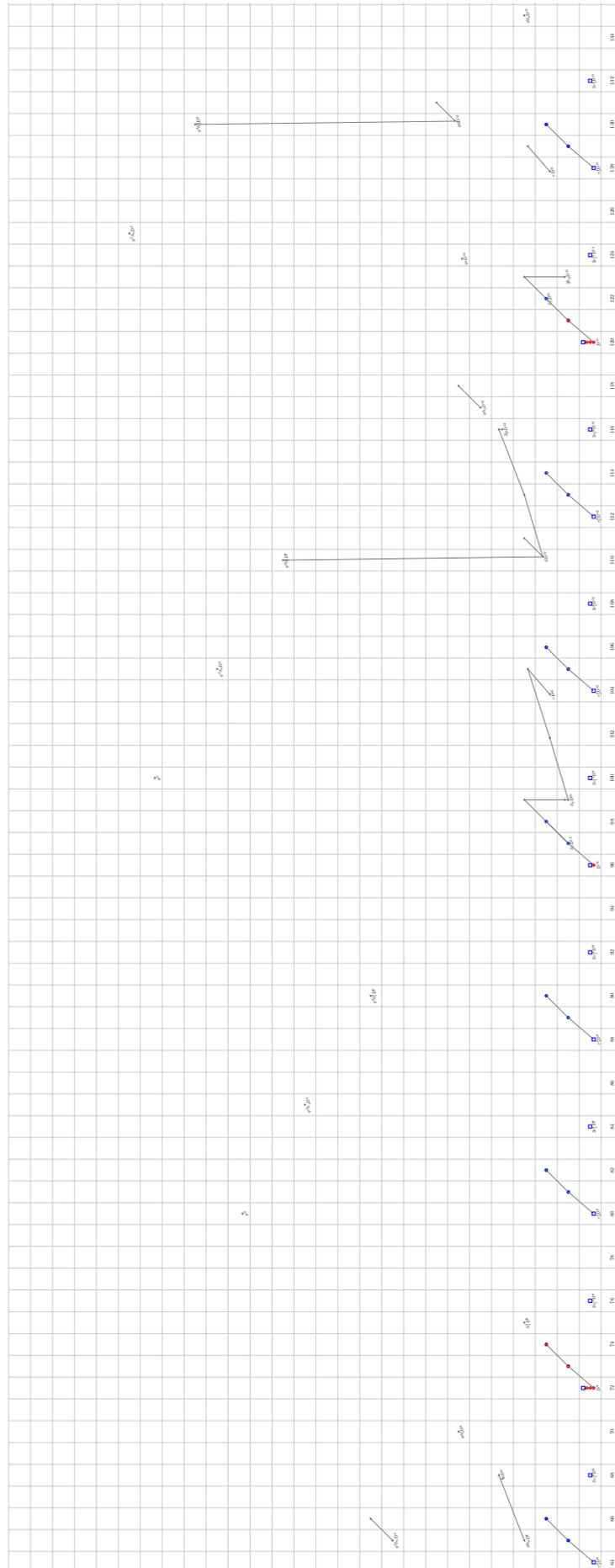}}\\
\caption{The $E_\infty$-page of the integer-graded $G_{24}$/$G_{48}$-HFPSS$(\E_2)$ (stem 64 -- 134).}
\label{fig:G24integerEinf2}
\hfill
\end{center}
\end{figure}

\begin{figure}[H]
\begin{center}
\makebox[0.95\textwidth]{\includegraphics[angle=90, trim={57cm, 0cm, 1cm, 0cm},clip,page=1,scale=0.78]{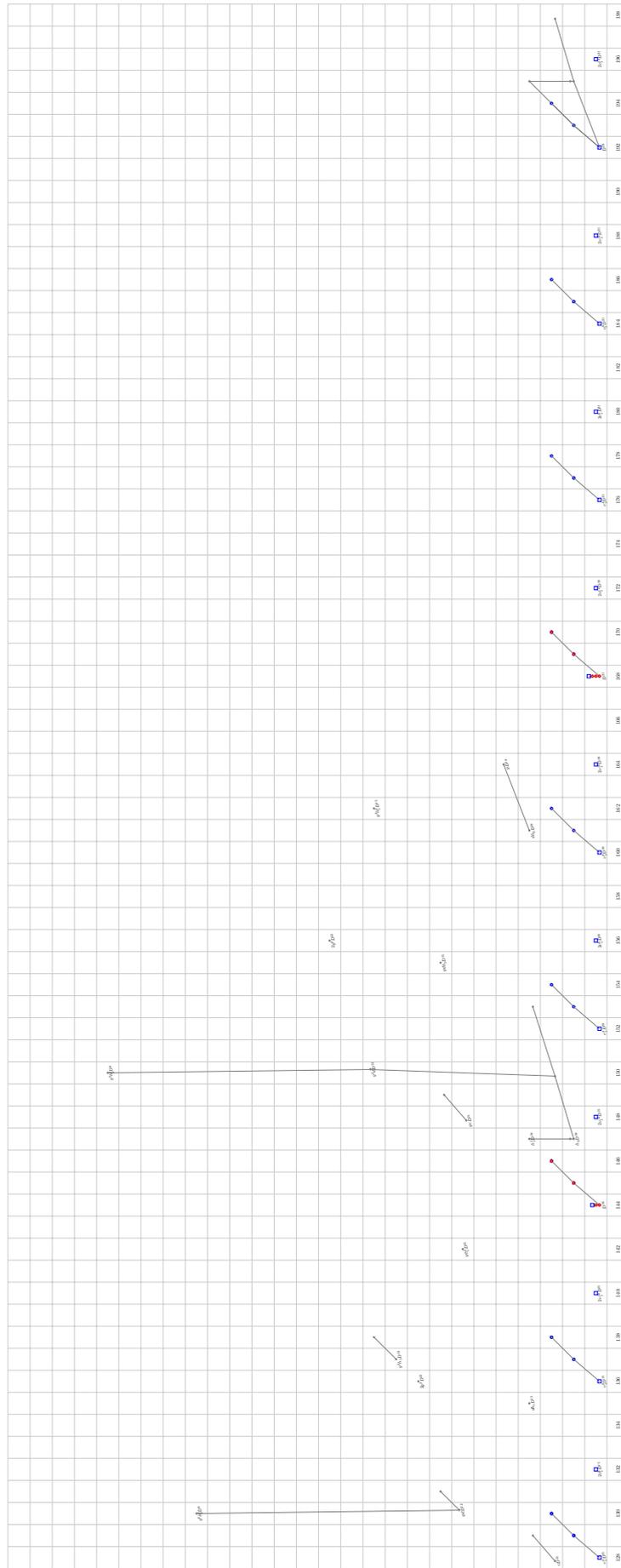}}
\caption{The $E_\infty$-page of the integer-graded $G_{24}$/$G_{48}$-HFPSS$(\E_2)$ (stem 128 -- 198).}
\label{fig:G24integerEinf3}
\hfill
\end{center}
\end{figure}

\newpage

\begin{figure}[H]
\begin{center}
\makebox[0.95\textwidth]{\includegraphics[trim={0cm, 0.8cm, 0cm, 0.8cm},clip,page=1,scale=1.4]{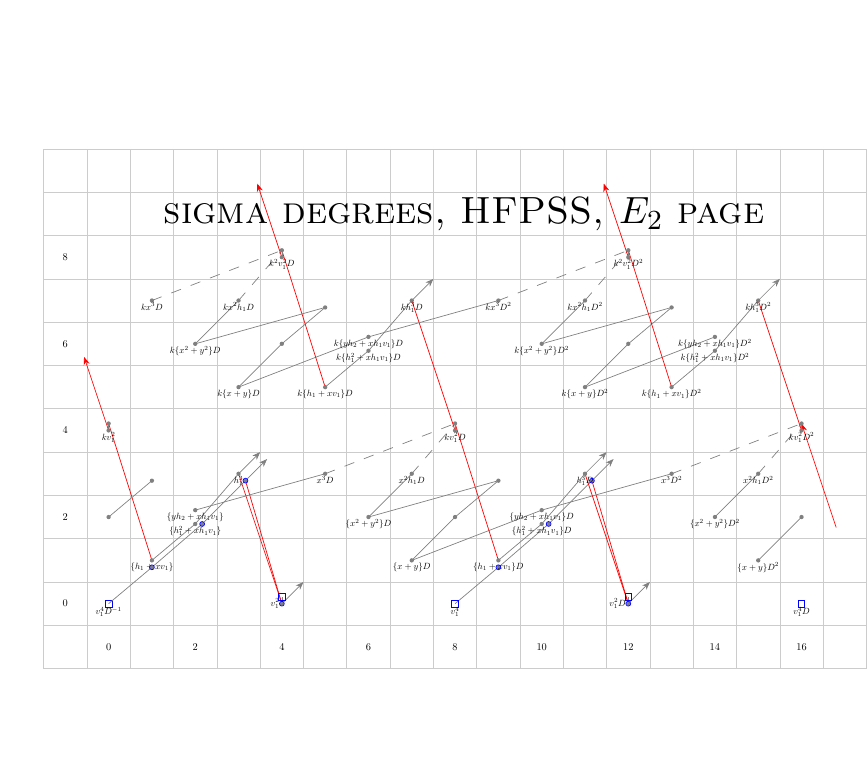}}
\caption{The $E_2$-page of the $(*-\sigma_i)$-graded $Q_8$-HFPSS($\E_2$). The {\color{red}red lines} are $d_3$-differentials.}

\label{fig:sigmaE2}
\hfill
\end{center}
\end{figure}

\begin{figure}[H]
\begin{center}
\makebox[0.95\textwidth]{\includegraphics[angle=90,trim={0cm, 0cm, 0cm, 0cm},clip,page=1,scale=0.7]{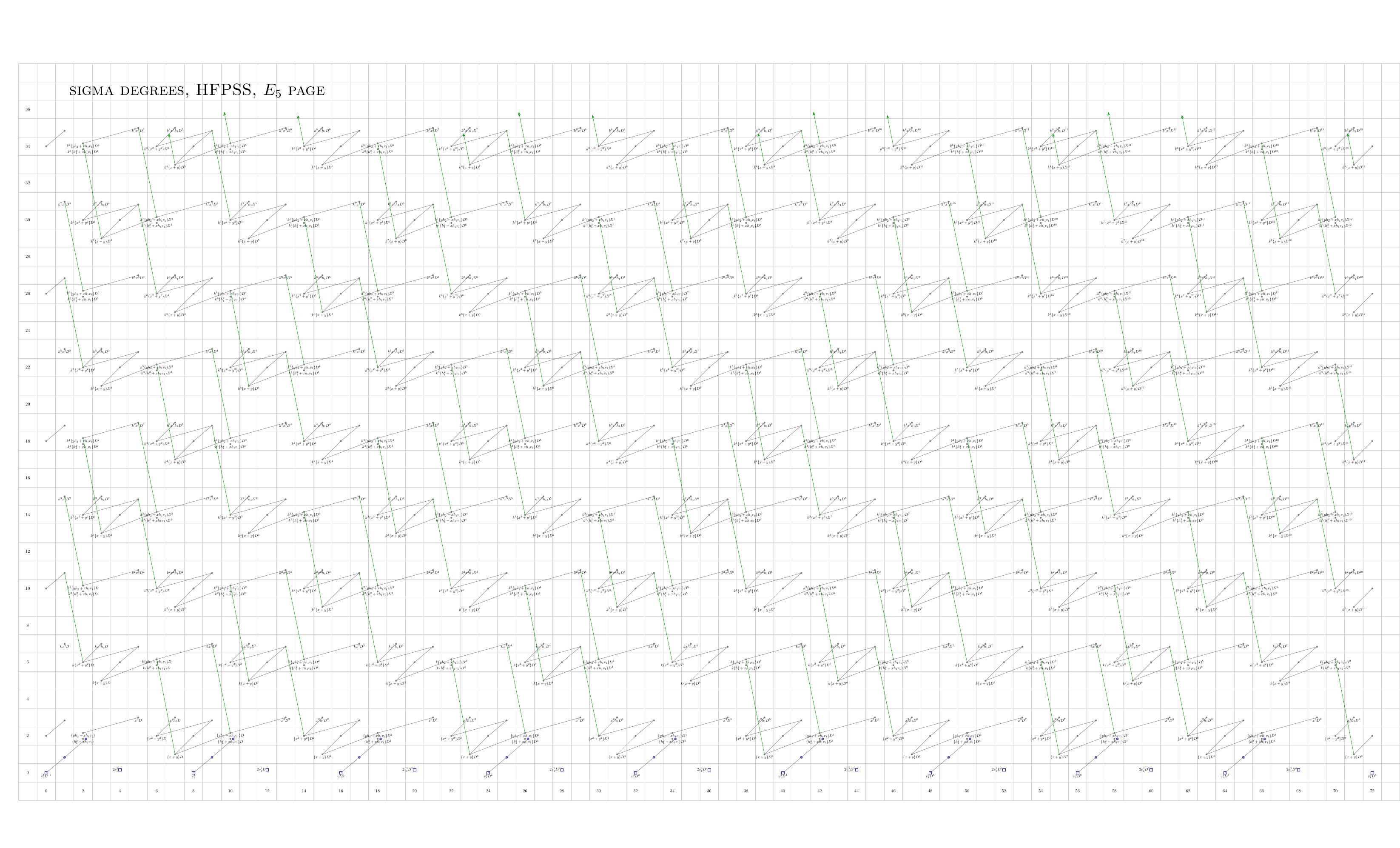}}
\caption{The $E_5$-page of the $(*-\sigma_i)$-graded $Q_8$-HFPSS($\E_2$). The {\color{darkgreen}green lines} are $d_5$-differentials.}
\label{fig:sigmaE5}
\hfill
\end{center}
\end{figure}

\begin{figure}[H]
\begin{center}
\makebox[0.95\textwidth]{\includegraphics[angle=90,trim={0cm, 0cm, 0cm, 0cm},clip,page=1,scale=0.7]{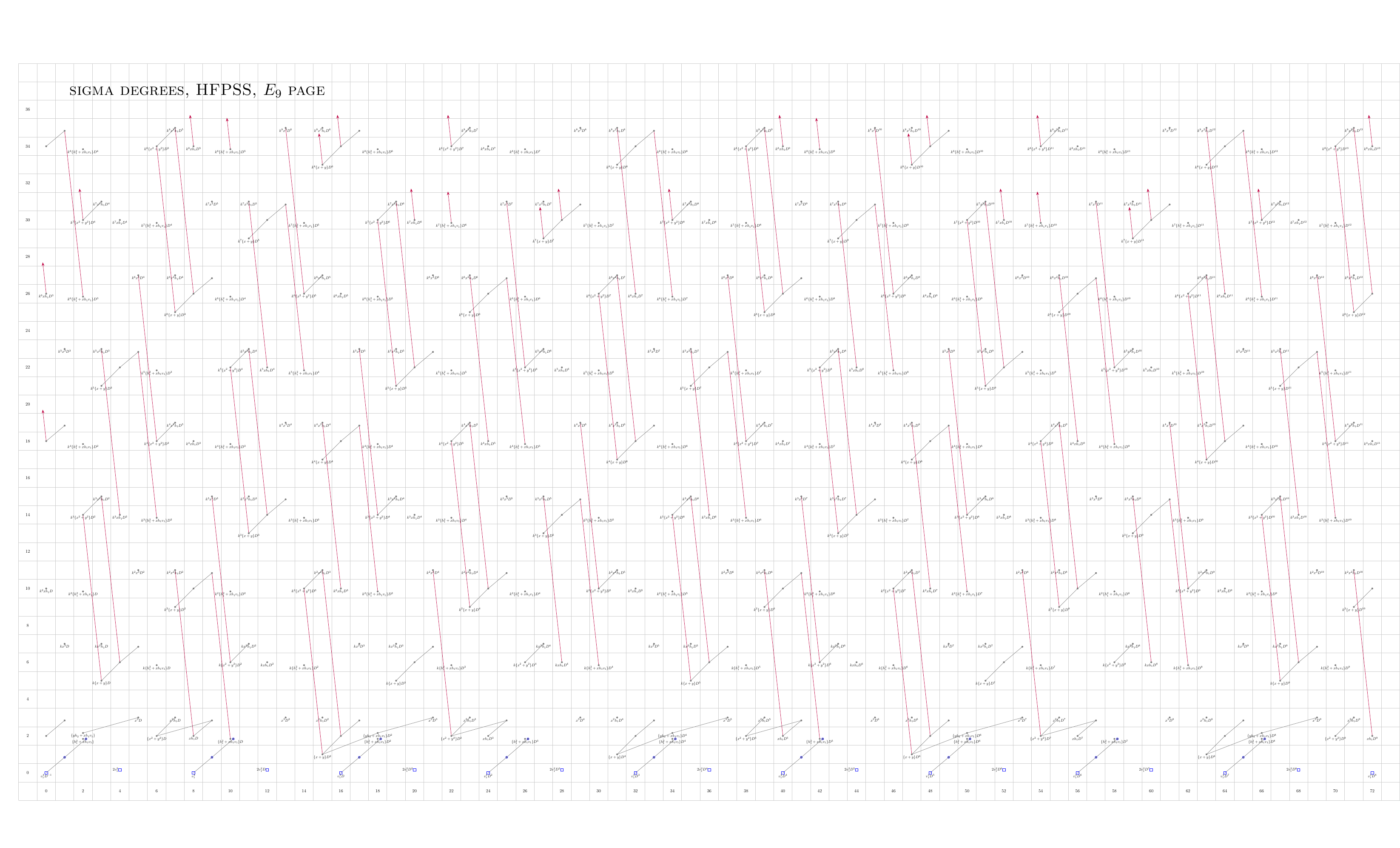}}
\caption{
The $E_9$-page of the $(*-\sigma_i)$-graded $Q_8$-HFPSS($\E_2$). The {\color{purple}purple lines} are $d_9$-differentials.}
\label{fig:sigmaE7}
\hfill
\end{center}
\end{figure}

\begin{figure}[H]
\begin{center}
\makebox[0.95\textwidth]{\includegraphics[angle=90,trim={0cm, 0cm, 0cm, 0cm},clip,page=1,scale=0.7]{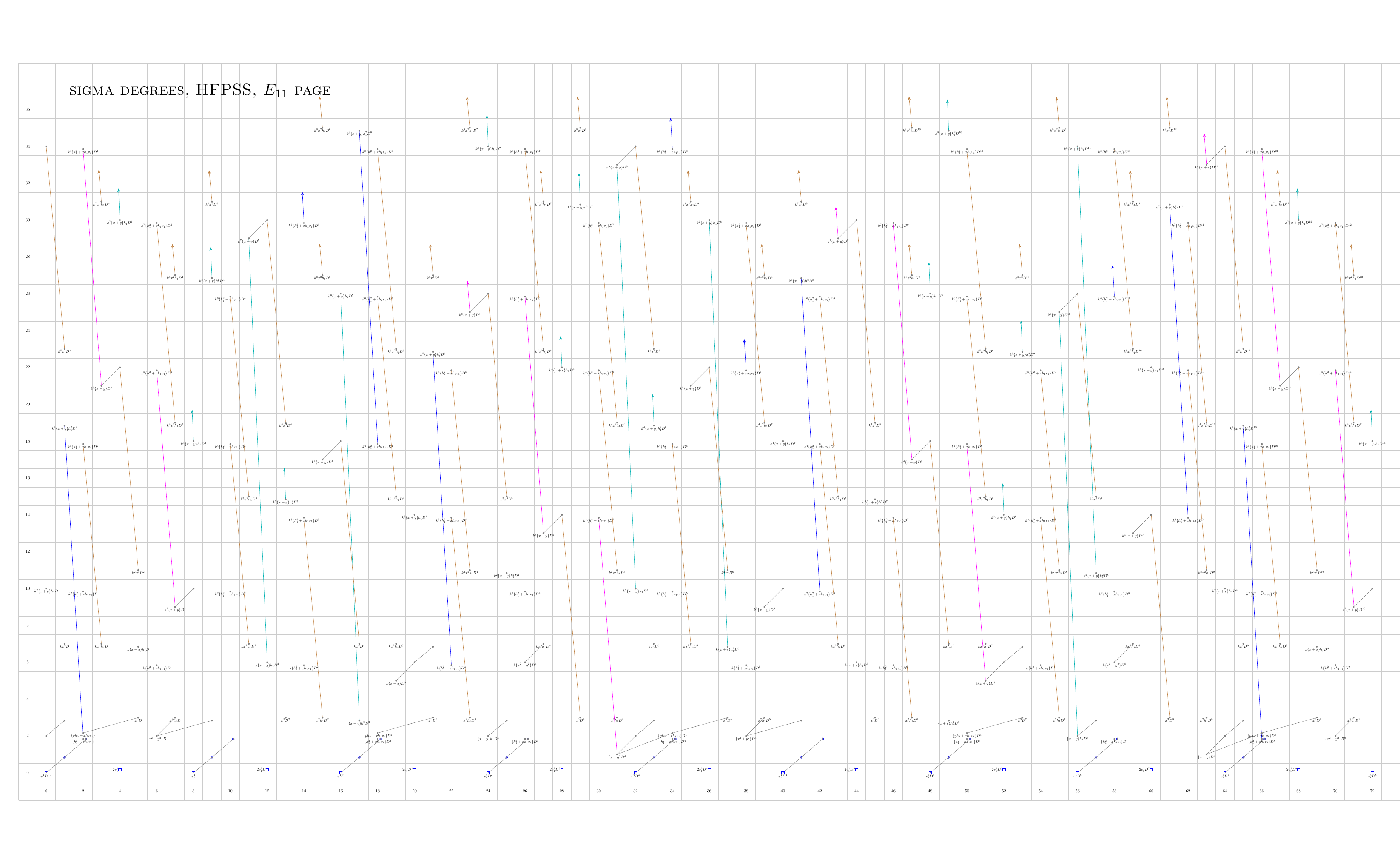}}
\caption{
The $E_{11}$-page of the $(*-\sigma_i)$-graded $Q_8$-HFPSS($\E_2$). The {\color{brown} brown lines} are $d_{11}$-differentials. The {\color{magenta} magenta lines} are $d_{13}$-differentials. The {\color{blue}blue lines} are $d_{17}$-differentials. The {\color{darkgreen}green lines} are $d_{13}$-differentials.
}
\label{fig:sigmaE9}
\hfill
\end{center}
\end{figure}

\begin{figure}[H]
\begin{center}
\makebox[0.95\textwidth]{\includegraphics[angle=90,trim={0cm, 0cm, 0cm, 0cm},clip,page=1,scale=0.7]{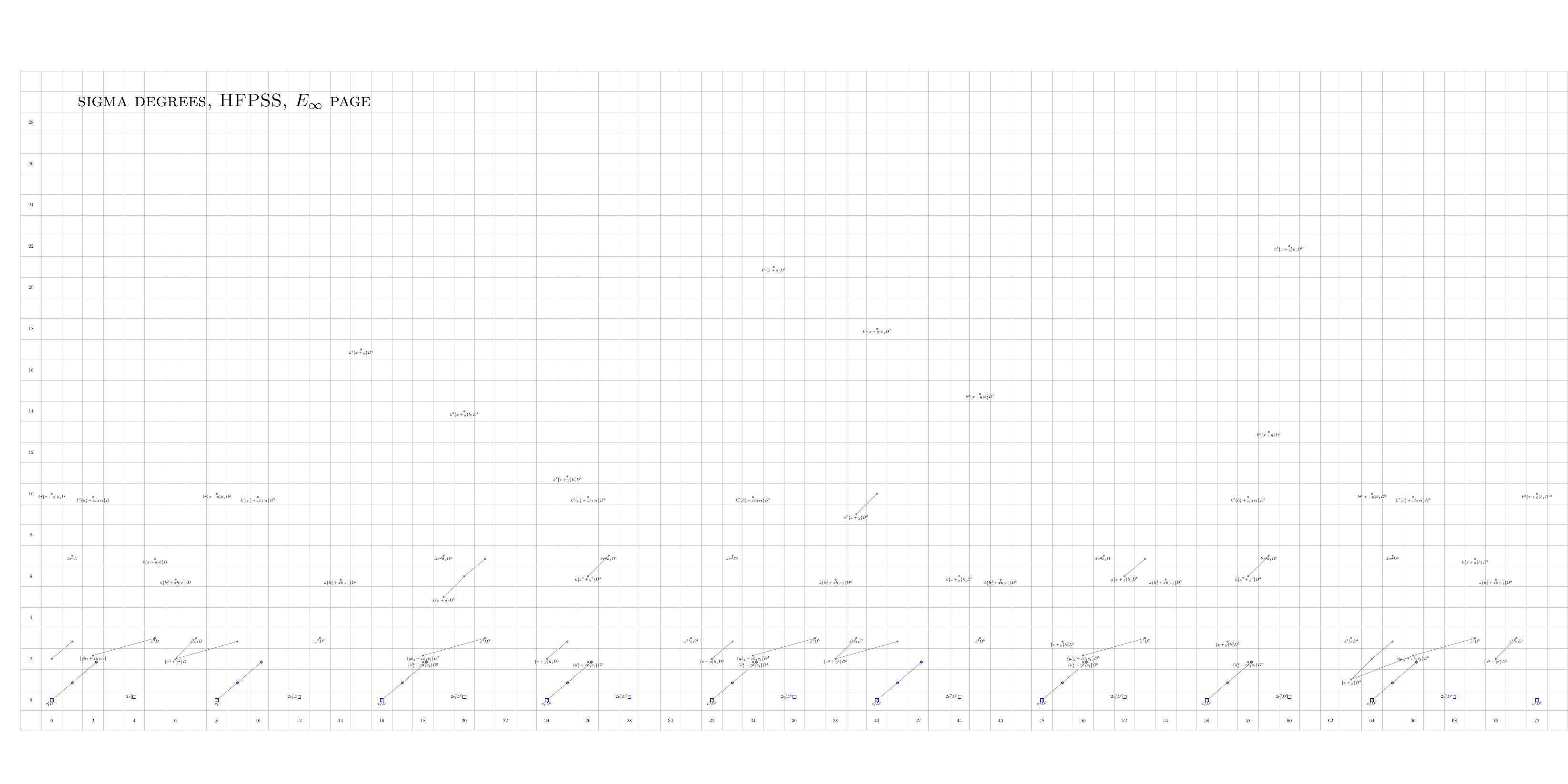}}
\caption{The $E_\infty$-page of the $(*-\sigma_i)$-graded $Q_8$/$SD_{16}$-HFPSS$(\E_2)$.}
\label{fig:sigmaEinf}
\hfill
\end{center}
\end{figure}

\appendix 

\section{Group Cohomology}\label{sec:groupcoho}

In this appendix, we collect and present examples of computations of group cohomology. There are two main applications: one is to calculate it as the input for the $E_2$-page of the integer- and $(*-\sigma_i)$-graded homotopy fixed points spectral sequences for $\E_2$, the other is to utilize restrictions, transfers, and norm maps for proofs of differentials. All the rests needed for our computation of the $Q_8$-HFPSS for $\E_2$ are listed in \cref{RestrictionList}.

Let $Q_8$ be presented as 
\[Q_8=\langle i,j\,|\,i^4,i^2=j^2,ijij^{-1}\rangle\] with its real representation ring $RO(Q_8) = \mathbb{Z}\{1, \sigma_i, \sigma_j, \sigma_k, \mathbb{H}\}$. To calculate $H^*(Q_8,A)$ we will use the following 4-periodic free $\mathbb{Z}[Q_8]$-resolution:
\[
0\leftarrow \Z\xleftarrow{\nabla}X_0\xleftarrow{d_0} X_1 \xleftarrow{d_1} X_2 \xleftarrow{d_2} ...,
\]
where $X_0=\mathbb{Z}[Q_8]\{a_0\}$, $\nabla(a_0)=1$, and for $k\geq 0$,
\[
\begin{array}{ll}
    X_{4k+1}=\mathbb{Z}[Q_8]\{b_{k,1},b_{k,2}\},\quad & d(b_{k,1})=(i-1)a_k,\,\\[2pt]
    & d(b_{k,2})=(j-1)a_k,\\[2pt]
    X_{4k+2}=\mathbb{Z}[Q_8]\{c_{k,1},c_{k,2}\}, \quad & d(c_{k,1})=(1+i)b_{k,1}-(1+j)b_{k,2},\\[2pt]
    & d(c_{k,2})=(1+ij)b_{k,1}+(i-1)b_{k,2},\\[2pt]
    X_{4k+3}=\mathbb{Z}[Q_8]\{e_k\}, \quad &  d(e_k)=(i-1)c_{k,1}-(ij-1)c_{k,2},\\[2pt]
    X_{4k+4}=\mathbb{Z}[Q_8]\{a_{k+1}\},\quad  & d(a_{k+1})=\textstyle \sum_{g\in Q_8} g\cdot e_{k}
\end{array}
\]

Suppose that $A$ is a $Q_8$-module, then $H^*(Q_8;A)$ is the cohomology of the cochain complex 
\begin{equation*}A \xrightarrow{d_0} A\oplus A 
\xrightarrow{d_1} A\oplus A \xrightarrow{d_2} A \xrightarrow{d_3} A\rightarrow ...\end{equation*}
where the differentials (by abuse of notation) are given by the following matrices
\[
    d_{4k}=\begin{pmatrix} i-1 \\ j-1 \end{pmatrix},
    \quad d_{4k+1}=\begin{pmatrix} 1+i & -1-j\\
        1+ij &-1+i \end{pmatrix}, 
        \quad d_{4k+2}=\begin{pmatrix} -1+i & 1-ij \end{pmatrix}, 
\]
and $d_{4k+3}=\sum_{g\in Q_8}g.$

We record here the group cohomology of $Q_8$ with trivial $\mathbb{Z}$ coefficients
\begin{equation*}
\begin{aligned}
    H^{4k+2}(Q_8,\mathbb{Z})&=\mathbb{Z}/2\oplus \mathbb{Z}/2,\\
H^{4k+4}(Q_8,\mathbb{Z})&=\mathbb{Z}/8,\\
H^{2q+1}(Q_8,\mathbb{Z})&=0,\\
\end{aligned}
\end{equation*}
where $k\geq 0, q\geq 0$, and the generator of $H^4(Q_8,\mathbb{Z})$ gives the 4-periodicity. 

In addition to the integer-graded  $Q_8$-HFPSS for $\E_2$, we also compute the $(*-\sigma_i)$-graded part. For this purpose, we study the structure of $\pi_* \E_2\otimes \sigma_i$ as a  $Q_8$-module,  which is given by the following analog of \cite[Lemma 4.6]{HM17} :

\begin{lem}\label{lem:twist}
Let $E$ be a $Q_8$-spectrum. Then \[\pi_*^e(E\wedge S^{1-\sigma_i})\cong \pi_*^e E \otimes \sigma_i\] as $Q_8$-modules.
\end{lem}

Recall that we defined $v_1=u_1u^{-1}$ and its $Q_8$-action was given in (\ref{eq:typeone}). By \cref{lem:E2completion}, we may first compute $H^*(Q_8,\W[u^{-1},v_1])$,  then invert $D$ and complete at $I=(2,u_1)$.

\begin{rem} If we define $s=i_*(u^{-1})$ and denote $u^{-1}$ by $t$, then the actions of $Q_8$ on $s,t$ are given by
\begin{equation*}\label{eq:typeone2}
    \begin{aligned}
    i_*(s)&=-t, & i_*(t)&=s\\
    j_*(s)&=-\zeta^2s+\zeta t, & j_*(t)&=\zeta s+\zeta^2 t\\
    k_*(s)&=\zeta s+\zeta^2 t, & k_*(t)&=\zeta^2 s-\zeta t
    \end{aligned}
\end{equation*}
For computational purposes, it is equivalent to replacing generators $u^{-1},v_1$ by $s,t$, and the form of the action turns out to be more compact.
\end{rem}

We first calculate the $0^\text{th}$cohomology ring. Behrens and Ormsby \cite{BO16} have determined the $C_4\langle i\rangle$-invariants:
\begin{prop} Let $b_2=s^2+t^2$, $b_4=s^3t-st^3$ and $\delta=s^2t^2$, then \[
H^0(C_4, \W[u^{-1},v_1])=\W[b_2,b_4,\delta]/(b_4^2-b_2^2\delta+4\delta^2).
\]
\end{prop}
The $j$-actions on $b_2,b_4,\delta$ are the following:
\begin{equation*}
\begin{aligned}
    j_*(b_2)&=-b_2,\\
    j_*(b_4)&=-(2\zeta+1)b_2^2+7b_4+8(2\zeta+1)\delta,\\
    j_*(\delta)&=b_2^2+2(2\zeta+1)b_4-7\delta.
\end{aligned}
\end{equation*}

\begin{prop} We have the $0^\text{th}$ cohomology ring
\[
H^0(Q_8,\W[u^{-1},v_1])= \W[s_1,s_2,s_3]/(s_1^3=4(2\zeta+1)s_1^2s_2+16s_1s_2^2)
\]
where $s_1=b_2^2$, $s_2=b_4+(2\zeta+1)\delta$, and $s_3=b_2^3+2(2\zeta+1)b_4b_2-8b_2\delta.$ 
\end{prop}

\begin{proof}
Since $\pi_* \E_2$ is 16-periodic as a $Q_8$-module, it suffices to compute the $j$-invariants of $H^0(C_4, \W[u^{-1},v_1])$ in low degrees. The result follows by direct computation.
\end{proof}

In the main computations, we sometimes need to rely on explicit group cohomology results. The following is an example. 

\begin{example}\label{prop:grouphomology_sigma}\rm
The calculation of $H^4(Q_8,\pi_4 \E_2\otimes \sigma_i)\cong \W/4$.

The cochain complex at degree $4$ looks like
\[\W\{s^2,st,t^2\}\xrightarrow{d_3}\W\{s^2,st,t^2\}\xrightarrow{d_4}\W\{s^2,st,t^2\}^2\]
By Lemma \ref{lem:twist}, the actions are
\begin{equation*}
\begin{aligned}
    i_*(s^2)&=t^2, & i_*(st)&=-st, &i_*(t^2)&=s^2\\
j_*(s^2)&=-\zeta s^2+2st-\zeta^2t^2, & j_*(st)&=s^2+(2\zeta +1)st-t^2, & j_*(t^2)&=-\zeta^2 s^2-2st-\zeta t^2.
    \end{aligned}
    \end{equation*}
Therefore, $\ker{d_4}=\ker{(i-1)}\cap \ker{(j-1)}=\ker{(i-1)}=\W\{s^2+t^2\}.$

Meanwhile, since we have
\begin{equation*}
    \begin{aligned}
    d_3(s^2)&=4(s^2+t^2),\\
      d_3(st)&=0,\\
      d_3(t^2)&=4(s^2+t^2), 
    \end{aligned}
\end{equation*} 
we conclude that  $H^4(Q_8,\pi_4 \E_2\otimes \sigma_i)\cong \W/4$.
\end{example}

We also calculate a couple of restriction maps in group cohomology. In the case of the integer-graded part, most calculations are easy. By Proposition \ref{prop:hurewiczC4} we deduce that the generators $\eta,\nu,c,d,g$ have to restrict non-trivially to their $C_4$-counterparts, which lie in the Hurewicz image. For the $(*-\sigma_i)$-graded part, some chain-level calculations seem to be inevitable. 

\begin{example}\rm
In the integer-graded part, calculate $\Res^{Q_8}_{C_4\langle i\rangle}D^{-2}d\neq 0$. This is used in the proof of Proposition \ref{prop:d11}.

The class $D^{-2}d$ lies in bigrading $(-2,2)$. We are looking at the degree 0 part of $\W[u^{-1},v_1]$. The generator of $H^2(Q_8,\W\{1\})$ is given by the cochain \begin{gather*}
    \alpha: \Z[Q_8]\{c_{0,1},c_{0,2}\} \rightarrow \W\{1\},\\c_{0,1}\mapsto 1, \quad c_{0,2}\mapsto 0.
\end{gather*}
Restricting to $C_4\langle i\rangle$, we rewrite $X_2=\Z[Q_8]\{c_{0,1},c_{0,2}\}$ as $\Z[C_4\langle i\rangle]\{c_{0,1},jc_{0,1},c_{0,2},jc_{0,2}\}$, and similarly for $X_1$. Then $\alpha$ restricts to the cochain 
\begin{gather*}
    \alpha: \Z[Q_8]\{c_{0,1},jc_{0,1},c_{0,2},jc_{0,2}\} \rightarrow \W\{1\},\\c_{0,1} ,\,jc_{0,1}\mapsto 1, \quad   c_{0,2}, \,jc_{0,2}\mapsto 0.
\end{gather*}
Now we check the image of $d_1$. Let $\beta_1,\beta_2,\beta_3,\beta_4$ be the dual basis of $b_{0,1},jb_{0,1},b_{0,2},jb_{0,2}$ in $\textrm{Hom}_{C_4\langle i\rangle}(X_1,\W\{1\})$. The image of $\beta_1$ is calculated by evaluating $\beta_1\circ d_1$ at the $C_4\langle i\rangle$-basis of $X_2$. As an example, we have \[(\beta_1\circ d_1)(c_{1,0})=\beta_1((1+i)b_{0,1}-b_{0,2}-jb_{0,2})=2.\]
Similarly, we verify that the restriction of $\alpha$ does not lie in the coboundary; hence the restriction is non-trivial.
\end{example}

Sometimes the restriction to $C_4\langle i\rangle$ is trivial, but it becomes non-trivial when restricted to $C_4\langle j\rangle$ or $C_4\langle k \rangle$. By similar calculations we have $\Res^{Q_8}_{\langle j\rangle}(x+y)u_{\sigma_j}=0$, while  $\Res^{Q_8}_{\langle j\rangle}(x+y)u_{\sigma_j}\neq 0$. 

Finally, we present the collection of calculated results. 

\begin{prop}\label{RestrictionList}
\textbf{Summary of calculated group cohomology}
\begin{itemize}
\item $H^3(Q_8,\mathbb{Z})=0$.
    \item $H^4(Q_8,\pi_4 \E_2\otimes \sigma_i)=\W/4$.
    \item $H^3(Q_8,\pi_4 \E_2\otimes \sigma_i)=\W/2$.
    \item $H^2(Q_8,\pi_4 \E_2\otimes \sigma_i)=\W/2 \oplus \W/2$.
    \item $H^1(Q_8,\pi_0 \E_2\otimes \sigma_i)=\W/2$.
\end{itemize}

\textbf{Summary of calculated restrictions}
\begin{itemize}
    \item $\Res^{Q_8}_{\langle i\rangle} h_1\neq 0.$
    
    \item $\Res^{Q_8}_{\langle i\rangle}h_2\neq 0.$
    
    \item $\Res^{Q_8}_{\langle i\rangle} d\neq 0$.
    
    \item $\Res^{Q_8}_{\langle i\rangle} g\neq 0$.
    
    \item $\Res^{Q_8}_{\langle i \rangle}\{x^2+y^2\}u_{\sigma_i} \neq 0.$
\end{itemize}
\end{prop}

In fact, the restriction map from $H^*(Q_8,\pi_* \E_2)$ to $H^*(C_4,\pi_* \E_2)$ is determined by the Hurewicz image of $\E_2^{hC_4}$. 
The direct algebraic computation we give above could potentially adapt to computations of higher heights. 

We recall the known result of the Hurewicz image result of $\E_2^{hC_4}$. We follow names introduced in \cref{prop:E2C4Slicename}.
\begin{prop}{{\rm (see \cite[Figure 12]{HSWX2018}\rm )}}\label{prop:hurewiczC4}
The following classes on the $E_\infty$-page of the $C_4$-$\mathrm{HFPSS}$ for $\E_2$ detects images of the Hurewicz map: $S^0\rightarrow \E_2^{hC_4}$:
\begin{itemize}
\item
$\bar{s}_1a_{\sigma_2}$ at $(1,1)$ detects the image of $\eta \in \pi_1S^0$, 
\item $\done u_{\lambda}a_{\sigma}$ at $(3,1)$ detects the image of $\nu \in \pi_3S^0$,
\item $\done^4 u_{4\sigma}a_{4\lambda}$ at $(8,8)$ detects the image of $\epsilon \in \pi_8S^0$, \item$\done^4 u_{4\lambda}u_{2\sigma}a_{2\sigma}$ at $(14,2)$ detects the image of $\kappa \in \pi_{14}S^0$,
\item$\done^6u_{4\lambda}u_{6\sigma}a_{2\lambda}$ at $(20,4)$ detects the image of $\bar\kappa \in \pi_{20}S^0$.
\end{itemize}
\end{prop}

The unit map $S^0\rightarrow \E_2^{hC_4}$ factors as
\[
S^0\xrightarrow{\text{unit}} \E_2^{hQ_8}\xrightarrow{\text{res}} \E_2^{hC_4}.
\]

There is a map of spectral sequences from the Adams--Novikov spectral sequence of the sphere to the $C_4$-HFPSS for $\E_2$, and it factors through the $Q_8$-HFPSS for $\E_2$. By comparing the Adams--Novikov spectral sequence of the sphere (e.g., see \cite[Table~2]{Rav77}) and the $C_4$-HFPSS for $\E_2$, we see that the classes detecting $\eta,\nu,g,d$ with no filtration jump under this map. Hence in the $Q_8$-HFPSS for $\E_2$, these classes are detected by classes $h_1,h_2,d,g$, and the $C_4$-restriction of these classes are non-trivial as follows.

\begin{prop}\label{prop:restrictionHurewicz}
The restriction map from the $E_2$-page of the $Q_8$-$\mathrm{HFPSS}$ for $\E_2$ to the $E_2$-page of the $C_4$-$\mathrm{HPFSS}$ for $\E_2$ is determined by the following and the multiplicative structure.
\begin{align*}
    &\Res^{Q_8}_{C_4}(h_1)=\sone a_{\sigma_2}, & &\Res^{Q_8}_{C_4}(h_2)=\done u_{\lambda}a_{\sigma},\\
    &\Res^{Q_8}_{C_4}(c)=0, & &\Res^{Q_8}_{C_4}(d)=\done^4 u_{4\lambda}u_{2\sigma}a_{2\sigma},\\
    &\Res^{Q_8}_{C_4}(g)=\done^{6}u_{4\lambda}u_{6\sigma}a_{2\lambda}.\\
    \end{align*}
\end{prop}

The element $\epsilon \in \pi_{8} S^0$ is detected by a class at filtration $2$ in the Adams--Novikov spectral sequence of the sphere. However, the image of $\epsilon$ in $\pi_8\E_2^{hC_4}$ is detected by $\done^4 u_{4\sigma} a_{4\lambda}$ at filtration $8$ in the $C_4$-HFPSS for $\E_2$. There is a filtration jump by $6$. For degree reasons, in $Q_8$-HFPSS($\E_2$), the image of $\epsilon$ could be potentially detected by a class of filtration $2\leq f\leq 8$. By the fact that the unit map $S^0\rightarrow \E_2^{hQ_8}$ further factors through $S^0\xrightarrow{\mathrm{unit}} \E_2^{hG_{24}}$, the image of $\epsilon$ is detected by the class $c$ at $(8,2)$ (up to a unit) in $Q_8$-HFPSS($\E_2$). Therefore, there is an exotic restriction in HFPSS from $Q_8$ to $C_4$ that maps the class $c$ to the class $\done^4 u_{4\sigma}a_{4\lambda}$.

 \bibliographystyle{alpha} 
 \bibliography{ref}

\begin{thebibliography}{GHMR05}

\bibitem[Ada84]{Ada84}
J.~F. Adams.
\newblock Prerequisites (on equivariant stable homotopy) for {Carlsson}'s lecture.
\newblock Algebraic topology, {Proc}. {Conf}., {Aarhus} 1982, {Lect}. {Notes} {Math}. 1051, 483-532 (1984)., 1984.

\bibitem[AM16]{AM69}
M.~F. Atiyah and I.~G. MacDonald.
\newblock {\em Introduction to commutative algebra}.
\newblock Addison-Wesley Series in Mathematics. Westview Press, Boulder, CO, economy edition, 2016.
\newblock For the 1969 original see MR 0242802.

\bibitem[Bau08]{Bau08}
Tilman Bauer.
\newblock Computation of the homotopy of the spectrum {\tt tmf}.
\newblock {\em Geometry \& Topology Monographs}, 13:11--40, 2008.

\bibitem[BBHS20]{BBHS20}
Agn\`es Beaudry, Irina Bobkova, Michael~A. Hill, and Vesna Stojanoska.
\newblock Invertible {$K(2)$}-local {$E$}-modules in {$C_4$}-spectra.
\newblock {\em Algebr. Geom. Topol.}, 20(7):3423--3503, 2020.

\bibitem[Bea15]{Bea15}
Agn\`es Beaudry.
\newblock The algebraic duality resolution at {$p=2$}.
\newblock {\em Algebr. Geom. Topol.}, 15(6):3653--3705, 2015.

\bibitem[Bea17a]{Bea17'}
Agn\`es Beaudry.
\newblock The chromatic splitting conjecture at {$n = p = 2$}.
\newblock {\em Geom. Topol.}, 21(6):3213--3230, 2017.

\bibitem[Bea17b]{Bea17}
Agn\`es Beaudry.
\newblock Towards the homotopy of the {$K(2)$}-local {M}oore spectrum at {$p=2$}.
\newblock {\em Adv. Math.}, 306:722--788, 2017.

\bibitem[BG18]{BG18}
Irina Bobkova and Paul~G. Goerss.
\newblock Topological resolutions in {$K(2)$}-local homotopy theory at the prime 2.
\newblock {\em J. Topol.}, 11(4):918--957, 2018.

\bibitem[BGH22]{BGH22}
Agn\`es Beaudry, Paul~G. Goerss, and H.-W. Henn.
\newblock Chromatic splitting for the {$K(2)$}-local sphere at {$p = 2$}.
\newblock {\em Geom. Topol.}, 26(1):377--476, 2022.

\bibitem[BM94]{BM94}
M.~B\"{o}kstedt and I.~Madsen.
\newblock Topological cyclic homology of the integers.
\newblock {\em Ast\'{e}risque}, (226):7--8, 57--143, 1994.
\newblock $K$-theory (Strasbourg, 1992).

\bibitem[BMQ20]{BMQ20}
Mark Behrens, Mark Mahowald, and J.~D. Quigley.
\newblock The 2-primary {H}urewicz image of tmf.
\newblock 2020.

\bibitem[BO16]{BO16}
Mark Behrens and Kyle Ormsby.
\newblock On the homotopy of {$Q(3)$} and {$Q(5)$} at the prime 2.
\newblock {\em Algebr. Geom. Topol.}, 16(5):2459--2534, 2016.

\bibitem[Bou79]{Bou79}
A.~K. Bousfield.
\newblock The localization of spectra with respect to homology.
\newblock {\em Topology}, 18(4):257--281, 1979.

\bibitem[Buj12]{Buj12}
C.~Bujard.
\newblock Finite subgroups of extended {M}orava stabilizer groups.
\newblock {\em arXiv preprint arXiv:1206.1951}, 2012.

\bibitem[DH04]{DH04}
Ethan~S. Devinatz and Michael~J. Hopkins.
\newblock Homotopy fixed point spectra for closed subgroups of the {M}orava stabilizer groups.
\newblock {\em Topology}, 43(1):1--47, 2004.

\bibitem[DLS22]{DLS2022}
Zhipeng Duan, Guchuan Li, and XiaoLin~Danny Shi.
\newblock Vanishing lines in chromatic homotopy theory.
\newblock {\em arXiv preprint arXiv:2204.08600}, 2022.

\bibitem[GH04]{GH04}
Paul~G. Goerss and Michael~J. Hopkins.
\newblock Moduli spaces of commutative ring spectra.
\newblock In {\em Structured ring spectra}, volume 315 of {\em London Math. Soc. Lecture Note Ser.}, pages 151--200. Cambridge Univ. Press, Cambridge, 2004.

\bibitem[GHM14]{GHM14}
Paul~G. Goerss, Hans-Werner Henn, and Mark Mahowald.
\newblock The rational homotopy of the {$K(2)$}-local sphere and the chromatic splitting conjecture for the prime 3 and level 2.
\newblock {\em Doc. Math.}, 19:1271--1290, 2014.

\bibitem[GHMR05]{GHMR05}
Paul~G. Goerss, H.-W. Henn, Mark Mahowald, and Charles Rezk.
\newblock A resolution of the {$K(2)$}-local sphere at the prime 3.
\newblock {\em Ann. of Math. (2)}, 162(2):777--822, 2005.

\bibitem[GM95]{GM95b}
J.~P.~C. Greenlees and J.~P. May.
\newblock {\em Generalized {Tate} cohomology}, volume 543 of {\em Mem. Am. Math. Soc.}
\newblock Providence, RI: American Mathematical Society (AMS), 1995.

\bibitem[Gre18]{Gre18}
J.~P.~C. Greenlees.
\newblock Four approaches to cohomology theories with reality.
\newblock In {\em An alpine bouquet of algebraic topology. Alpine algebraic and applied topology conference, Saas-Almagell, Switzerland, August 15--21, 2016. Proceedings}, pages 139--156. Providence, RI: American Mathematical Society (AMS), 2018.

\bibitem[GS22]{BC22}
Bertrand~J. Guillou and Carissa Slone.
\newblock The slices of quaternionic {E}ilenberg--{M}ac {L}ane spectra.
\newblock {\em arXiv preprint arXiv:2204.03127}, 2022.

\bibitem[Hen07]{Hen07}
H.-W. Henn.
\newblock On finite resolutions of {$K(n)$}-local spheres.
\newblock In {\em Elliptic cohomology}, volume 342 of {\em London Math. Soc. Lecture Note Ser.}, pages 122--169. Cambridge Univ. Press, Cambridge, 2007.

\bibitem[Hen19]{Hen19}
H.-W. Henn.
\newblock The centralizer resolution of the {$K(2)$}-local sphere at the prime 2.
\newblock In {\em Homotopy theory: tools and applications}, volume 729 of {\em Contemp. Math.}, pages 93--128. Amer. Math. Soc., Providence, RI, 2019.

\bibitem[Hew95]{Hew95}
Thomas Hewett.
\newblock Finite subgroups of division algebras over local fields.
\newblock {\em J. Algebra}, 173(3):518--548, 1995.

\bibitem[HHR16]{HHR16}
Michael~A. Hill, Michael~J. Hopkins, and Douglas~C. Ravenel.
\newblock On the nonexistence of elements of {K}ervaire invariant one.
\newblock {\em Ann. of Math. (2)}, 184(1):1--262, 2016.

\bibitem[HHR17]{HHR17}
Michael~A. Hill, Michael~J. Hopkins, and Douglas~C. Ravenel.
\newblock The slice spectral sequence for the {$C_4$} analog of real {$K$}-theory.
\newblock {\em Forum Math.}, 29(2):383--447, 2017.

\bibitem[HL16]{HL16}
Michael~A. Hill and Tyler Lawson.
\newblock Topological modular forms with level structure.
\newblock {\em Invent. Math.}, 203(2):359--416, 2016.

\bibitem[HM14]{HM14}
Michael~J. Hopkins and Mark Mahowald.
\newblock From elliptic curves to homotopy theory.
\newblock In {\em Topological modular forms}, volume 201 of {\em Math. Surveys Monogr.}, pages 261--285. Amer. Math. Soc., Providence, RI, 2014.

\bibitem[HM17]{HM17}
Michael~A. Hill and Lennart Meier.
\newblock The {$C_2$}-spectrum {${\rm Tmf}_1(3)$} and its invertible modules.
\newblock {\em Algebr. Geom. Topol.}, 17(4):1953--2011, 2017.

\bibitem[HMS94]{HMS94}
Michael~J. Hopkins, Mark Mahowald, and Hal Sadofsky.
\newblock Constructions of elements in {P}icard groups.
\newblock In {\em Topology and representation theory ({E}vanston, {IL}, 1992)}, volume 158 of {\em Contemp. Math.}, pages 89--126. Amer. Math. Soc., Providence, RI, 1994.

\bibitem[Hop02]{Hop02}
Michael~J. Hopkins.
\newblock Algebraic topology and modular forms.
\newblock In {\em Proceedings of the {I}nternational {C}ongress of {M}athematicians, {V}ol. {I} ({B}eijing, 2002)}, pages 291--317. Higher Ed. Press, Beijing, 2002.

\bibitem[HS99]{HS99}
Mark Hovey and Neil~P. Strickland.
\newblock Morava {$K$}-theories and localisation.
\newblock {\em Mem. Amer. Math. Soc.}, 139(666):viii+100, 1999.

\bibitem[HS20]{HS20}
Jeremy Hahn and XiaoLin~Danny Shi.
\newblock Real orientations of {L}ubin--{T}ate spectra.
\newblock {\em Invent. Math.}, 221(3):731--776, 2020.

\bibitem[HSWX23]{HSWX2018}
Michael~A. Hill, XiaoLin~Danny Shi, Guozhen Wang, and Zhouli Xu.
\newblock {\em The slice spectral sequence of a {{\(C_4\)}}-equivariant height-4 {Lubin}-{Tate} theory}, volume 1429 of {\em Mem. Am. Math. Soc.}
\newblock Providence, RI: American Mathematical Society (AMS), 2023.

\bibitem[Isa09]{Isa09}
Daniel~C. Isaksen.
\newblock The cohomology of motivic {$A(2)$}.
\newblock {\em Homology Homotopy Appl.}, 11(2):251--274, 2009.

\bibitem[Isa18]{Isa18}
Daniel~C. Isaksen.
\newblock The homotopy of {$\mathbb{C}$}-motivic modular forms, 2018.

\bibitem[LSWX19]{LSWX19}
Guchuan Li, XiaoLin~Danny Shi, Guozhen Wang, and Zhouli Xu.
\newblock Hurewicz images of {R}eal bordism theory and {R}eal {J}ohnson--{W}ilson theories.
\newblock {\em Adv. Math.}, 342:67--115, 2019.

\bibitem[LT65]{LT65}
Jonathan Lubin and John Tate.
\newblock Formal complex multiplication in local fields.
\newblock {\em Ann. of Math. (2)}, 81:380--387, 1965.

\bibitem[Lur18]{Lur18}
Jacob Lurie.
\newblock Elliptic cohomology ii: orientations.
\newblock {\em preprint}, 2018.

\bibitem[Mor85]{Mor85}
Jack Morava.
\newblock Noetherian localisations of categories of cobordism comodules.
\newblock {\em Ann. of Math. (2)}, 121(1):1--39, 1985.

\bibitem[MSZ20]{MSZ20}
Lennart Meier, XiaoLin~Danny Shi, and Mingcong Zeng.
\newblock The localized slice spectral sequence, norms of {R}eal bordism, and the {S}egal conjecture.
\newblock 2020.

\bibitem[MSZ23]{LSDZ23}
Lennart Meier, XiaoLin~Danny Shi, and Mingcong Zeng.
\newblock The localized slice spectral sequence, norms of real bordism, and the {Segal} conjecture.
\newblock {\em Adv. Math.}, 412:74, 2023.
\newblock Id/No 108804.

\bibitem[Qui69]{Qui69}
Daniel Quillen.
\newblock On the formal group laws of unoriented and complex cobordism theory.
\newblock {\em Bull. Amer. Math. Soc.}, 75:1293--1298, 1969.

\bibitem[Rav78]{Rav77}
Douglas~C. Ravenel.
\newblock A novice's guide to the {A}dams--{N}ovikov spectral sequence.
\newblock 658:404--475, 1978.

\bibitem[Rav92]{Rav92}
Douglas~C. Ravenel.
\newblock {\em Nilpotence and periodicity in stable homotopy theory}, volume 128 of {\em Annals of Mathematics Studies}.
\newblock Princeton University Press, Princeton, NJ, 1992.
\newblock Appendix C by Jeff Smith.

\bibitem[Rez98]{Rez98}
Charles Rezk.
\newblock Notes on the {H}opkins--{M}iller theorem.
\newblock 220:313--366, 1998.

\bibitem[Sch11]{Schwede}
Stefan Schwede.
\newblock Lectures on equivariant stable homotopy theory, 2011.

\bibitem[Tod62]{Tod62}
Hiroshi Toda.
\newblock {\em Composition methods in homotopy groups of spheres}.
\newblock Annals of Mathematics Studies, No. 49. Princeton University Press, Princeton, N.J., 1962.

\bibitem[Ull13]{Ull13}
John~Richard Ullman.
\newblock {\em On the {R}egular {S}lice {S}pectral {S}equence}.
\newblock ProQuest LLC, Ann Arbor, MI, 2013.
\newblock Thesis (Ph.D.)--Massachusetts Institute of Technology.

\end{thebibliography}

\end{document}